\documentclass{article}
\usepackage{graphicx} 
\usepackage[T1]{fontenc}
\usepackage[utf8]{inputenc}
\usepackage[english]{babel}
\usepackage{amsmath}
\usepackage{amsthm}
\usepackage{amssymb}
\usepackage{mathtools}
\usepackage{esint}
\usepackage{tikz}
\usepackage{faktor}
\usepackage{xcolor, colortbl}
\usepackage{makecell}
\usepackage{caption}
\usepackage{subcaption}
\usepackage{comment}
\usepackage{setspace}
\usepackage[12pt]{moresize}
\usepackage{float}
\usepackage[margin=1in]{geometry}
\usepackage{mathrsfs}
\usepackage{url}
\usepackage{authblk}
\usepackage{bbm}
\usepackage{hyperref}
\usepackage{cleveref}
\usepackage{enumitem}
\usepackage{bm}
\hypersetup{
colorlinks = true,
urlcolor   = blue,
citecolor  = blue,
linkcolor  = blue,
}

\DeclareMathOperator{\dist}{dist}

\DeclareMathOperator{\loc}{loc}

\DeclareMathOperator{\supp}{supp}

\DeclareMathOperator{\diam}{diam}
\DeclareMathOperator{\interior}{int}

\theoremstyle{plain}
\newtheorem{thm}{Theorem}[section]

\newtheorem*{thm*}{Theorem}
\newtheorem{lem}[thm]{Lemma}
\newtheorem{prop}[thm]{Proposition}
\newtheorem{cor}[thm]{Corollary}
\newtheorem{defn}[thm]{Definition}
\theoremstyle{remark}
\newtheorem{rmk}[thm]{Remark}

\newcommand{\R}{\mathbb{R}}
\newcommand{\Sph}{\mathbb{S}}

\newcommand{\N}{\mathbb{N}}

\newcommand{\eps}{\varepsilon}

\newcommand{\res}{\mathop{\hbox{\vrule height 7pt width .5pt depth 0pt
			\vrule height .5pt width 6pt depth 0pt}}\nolimits}

\numberwithin{equation}{section}

\title{The fractional Laplacian in Lipschitz domains: \\Dahlberg's Theorem and $L^{2}$-solvability}
\author{Roberto Colombo, Xavier Fernández-Real, Xavier Ros-Oton}
\date{}

\begin{document}

\maketitle

\begin{abstract}
    Given $s\in (0,1)$ and a bounded Lipschitz domain $\Omega\subset \R^{n}$, we establish a quantitative Dahlberg theory for the
$s$-harmonic measure of $\Omega$, $\omega_s^x$. In the nonlocal setting, the natural reference measure is an integral weight $\sigma_{s}$ in $\Omega^{c}$ that behaves like $(1-s)\dist(\cdot, \partial \Omega)^{-s}$ close to the boundary. Our main result is a scale-invariant reverse-H\"{o}lder estimate for the density $d\omega_{s}^{x}/d\sigma_{s}$ on boundary-centered balls. As a consequence, we obtain $L^2(\Omega^c,\sigma_s)$-solvability of the exterior Dirichlet problem, with estimates for a nonlocal non-tangential maximal function and uniqueness in the natural distributional class. A weighted Gehring argument improves the reverse-H\"{o}lder exponent beyond $2$ and consequently yields $L^{q}$-solvability for a range of exponents extending strictly below $2$. Our results apply to general symmetric stable operators comparable to the fractional Laplacian. Moreover, the proofs are compatible with the limit $s\to 1^-$ and thus yield the corresponding results for the Laplacian in the nonlocal-to-local limit. The main new step is to convert a fractional Pohozaev identity for the Green function into uniform square estimates on distance level sets of a Lipschitz boundary. As applications, we derive optimal Sobolev regularity estimates for the homogeneous weighted Dirichlet problem and for the inhomogeneous Poisson problem with zero exterior data. 
\end{abstract}

%\tableofcontents


\section{Introduction}

The study of harmonic functions in domains has been a central line of research in PDE and harmonic analysis for the last century. In this context, the distinction between smooth and Lipschitz domains is fundamental in the study of boundary value problems. In a $C^{1,\alpha}$-domain, rescaling around a boundary point makes the boundary progressively flatter, and perturbative arguments may be combined with the classical boundary regularity theory. A Lipschitz boundary, by contrast, is invariant as a geometric class under rescaling: corners and oscillations may persist at every scale, and there is no improvement of the boundary geometry under blow-up. Lipschitz domains, therefore, constitute a natural endpoint class for scale-invariant boundary estimates, and their analysis requires tools that depend only on quantitative geometric objects such as corkscrew points and Harnack chains \cite{kenig1994harmonic,prats2026harmonic}. This is analogous to the distinction between elliptic equations with continuous coefficients, for which Schauder-type estimates are available, and equations with merely bounded measurable coefficients, where completely different techniques are required. 

For the Laplacian, the resulting theory is both deep and remarkably complete. It includes the boundary Harnack principle, quantitative estimates for harmonic measure, and the solvability of boundary value problems with rough data. Dahlberg's theorem is one of its central results: if $\Omega\subset\R^n$ is a bounded Lipschitz domain and $\omega^{x}$ denotes harmonic measure with pole $x\in\Omega$, then $\omega^{x}$ is quantitatively absolutely continuous with respect to the surface measure on $\partial\Omega$, and its Poisson kernel satisfies an $L^{2}$-reverse-H\"{o}lder estimate on surface balls \cite{dahlberg1977estimates}. Together with the Gehring-type self-improvement of reverse-H\"{o}lder inequalities, this estimate is a principal ingredient in the solvability of the Dirichlet problem with $L^q$ boundary data and non-tangential maximal function bounds; see \cite{huntWheeden1968boundary,dahlberg1977estimates,dahlberg1979poisson,jerisonkenig1980identity,jerison1981dirichlet,kenig1994harmonic}. This theory was further developed for Neumann boundary conditions and for more general elliptic operators \cite{jerison1981neumann,fabes1984necessary,dahlberg1987hardy,kenig1993neumann}. Related Sobolev estimates for the inhomogeneous Dirichlet problem were obtained in \cite{jerison1995inhomogeneous} (see also \cite{jerison2026homogeneous}). Analogous questions have also been studied for the $p$-Laplacian \cite{lewis2010boundary}. For the sharp geometric characterization of quantitative absolute continuity of harmonic measure with respect to surface measure, equivalently of $L^{p}$-solvability of the Dirichlet problem for some finite $p$, under Ahlfors--David regularity and an interior corkscrew condition, see \cite{azzamHofmannMartellMourgoglouTolsa2020} and the references therein. We refer the reader to \cite{kenig1994harmonic,prats2026harmonic} for comprehensive treatments of this extensively studied topic.

The purpose of this paper is to develop an analogue of this theory for the fractional Laplacian \eqref{eq:fract_laplacian} and the corresponding class of $s$-harmonic functions (see \cite{XROXFRbook}):
\begin{equation*}
    (-\Delta)^s u=0\qquad\text{in $\Omega$},
    \qquad 0<s<1.
\end{equation*}
Nonlocal equations of this type have attracted great interest in the PDE community in the last two decades, especially since the works of Caffarelli and Silvestre \cite{caffarelli2007extension,caffarelli2009regularity, CaffarelliSilvestre2011}. Most boundary regularity results for the fractional Laplacian and related nonlocal operators have been proved in $C^{k,\alpha}$ or smoother domains \cite{rosoton2014dirichlet,grubb2015fractional,abatangelo2020obstacle}. In Lipschitz domains much less is known: the robust regularity results available so far are essentially H\"{o}lder continuity up to the boundary and the boundary Harnack principle \cite{bogdan1997boundaryHarnack,bogdan1999representationMartin}. In particular, the following questions have remained open in general Lipschitz domains:
\begin{enumerate}[label=\textup{(\roman*)}]
\item \emph{Is there an analogue of Dahlberg's theorem for the fractional Laplacian?}
\item \emph{Can one prove an $L^2$-solvability result for the exterior Dirichlet problem?}
\end{enumerate}

\noindent These are the questions that we tackle in this paper.

\bigskip 

A decisive structural feature of the classical local setting is that the boundary data, harmonic measure, and the reference surface measure all live on the same codimension-one set $\partial\Omega$. In the nonlocal context, instead, the equation couples every point of $\Omega$ to the whole complement $\Omega^c$. Dirichlet data are therefore prescribed on an exterior set of full dimension rather than only on $\partial\Omega$, and consequently, even the correct formulation of a fractional Dahlberg's theorem is not immediate; cf.~\cite{DavidEngelsteinMayboroda2021Duke}. More precisely, let $\omega_s^x$ denote the $s$-harmonic measure of $\Omega$ with pole $x\in\Omega$. Probabilistically, it is the exit distribution from $\Omega$ of the isotropic $2s$-stable process starting at $x$. In the present setting, $\omega_{s}^{x}$ is absolutely continuous with respect to Lebesgue measure on $\Omega^c$, and its density
\begin{equation}\label{eq:def-poisson-intro}
    P^{x}(y)\coloneqq\frac{d\omega_{s}^{x}}{d\mathscr{L}^{n}}(y)
    \qquad \forall (x, y)\in\Omega\times \Omega^c,
\end{equation}
is called the (nonlocal) Poisson kernel; see
\cite{ikedaWatanabe1962harmonic,chenSong1998green, XROXFRbook}. The exit distribution is thus spread throughout the complement rather than carried by the codimension-one boundary. The surface measure cannot serve as a reference measure, while the unweighted Lebesgue measure encodes neither the singular concentration of $P^x$ near $\partial\Omega$ nor its decay at infinity. 
The appropriate reference measure is identified by the half-space model. If
$\mathbb{H}=\{x_n>0\}$, then 
\begin{equation}\label{eq:half-space-poisson-intro}
    P_{\mathbb{H}}^x(y)    = C_{n, s}\frac{x_n^s}{(-y_n)^s|x-y|^n}\qquad
    \forall (x, y)\in \mathbb{H}\times \mathbb{H}^c;
\end{equation}
see, for instance,
\cite{bogdan1997boundaryHarnack,bogdan1999representationMartin, XROXFRbook}.
The factor $\dist(y,\partial\mathbb{H})^{-s}$ describes the singular behavior of the exit distribution near the boundary, while the remaining denominator determines its long-range decay. This leads, for a bounded Lipschitz domain, to the exterior measure
\begin{equation}\label{eq:def-sigma}
     \sigma_{s}\coloneqq(1-s)\frac{\delta^{-s}}{1+\delta^{n+s}}\,\mathscr{L}^{n}\res \Omega^{c},
    \end{equation}
where $\delta:\R^{n}\to [0,\infty)$ denotes the distance function from the boundary 
\begin{equation*}
    \delta(y)\coloneqq\dist(y,\partial\Omega).
\end{equation*}
On boundary-centered balls of sufficiently small radius, $\sigma_{s}(B_r(\xi))\approx r^{n-s}$, so $\sigma_s$ has the effective local (fractional) dimension $n-s$, which approaches the boundary dimension $n-1$ as $s\to 1^-$. The normalization by $1-s$ gives the corresponding weak limit $\sigma_{s} \rightharpoonup \mathcal{H}^{n-1}\res \partial \Omega$. Distance-weighted trace spaces and robust localization limits closely related to \eqref{eq:def-sigma} have been developed in
\cite{grubeHensiek2024trace,grubeKassmann2025trace}. The role established in the present paper is different: $\sigma_s$ is the reference measure with respect to which fractional harmonic measure has a quantitative reverse-H\"{o}lder structure.

The central question of this paper is therefore whether the scale-invariant harmonic measure theory of Dahlberg in the rough geometric class of Lipschitz domains survives in the nonlocal setting. We prove that it does. More precisely, our main estimate is a scale-invariant $L^{2}$-reverse-H\"{o}lder bound for the density $d\omega_{s}^{x_0}/d\sigma_{s}$ with respect to the exterior weight $\sigma_{s}$; see Theorem~\ref{thm:Dahlberg-intro}. Its direct consequence is the $L^{2}(\Omega^{c},\sigma_{s})$-solvability of the exterior Dirichlet problem with a non-tangential maximal function estimate; see Theorem \ref{thm:Dirichlet-solvability-intro}. A weighted Gehring argument then self-improves the reverse-H\"{o}lder estimate to some exponent above $2$, yielding as a further consequence $L^q$-solvability for the Dirichlet problem with $q$ in a range that extends strictly below $2$. We expect these results to be useful in the study of Poisson, Neumann, and regularity problems for the fractional Laplacian in Lipschitz domains. For instance, as corollaries of our theory, we derive some optimal Sobolev regularity estimates for the homogeneous weighted exterior Dirichlet problem and for the inhomogeneous Poisson problem with zero exterior data; see Section \ref{subsec:applications} below.  

\subsection{Main results and relation to previous work}

Our first theorem proves, in the nonlocal setting, Dahlberg's $L^{2}$ reverse-H\"{o}lder estimate in Lipschitz domains \cite{dahlberg1977estimates}; see Definition \ref{defi:Lip_char} for the notion of Lipschitz character. In the nonlocal context, the surface measure of $\partial \Omega$ is replaced by the exterior weight $\sigma_{s}$ introduced in \eqref{eq:def-sigma}.

\begin{thm}[Dahlberg's theorem for the fractional Laplacian]\label{thm:Dahlberg-intro}
    Let $\Omega\subset\R^n$ be a bounded Lipschitz domain, let $x_0\in\Omega$, and let $\sigma_s$ be the exterior measure in \eqref{eq:def-sigma}. Then, for every ball $B\subset\R^n$ centered on $\partial\Omega$,
    \begin{equation}\label{eq:reverse-holder-intro}
        \left( 
        \fint_{B\cap\Omega^c}
        \left(\frac{d\omega_s^{x_0}}{d\sigma_s}\right)^2
        d\sigma_s\right)^{1/2}
        \le C\,\frac{\omega_{s}^{x_0}(B\cap \Omega^{c})}{\sigma_{s}(B\cap \Omega^{c})},
    \end{equation}
    where $C$ depends only on $n$, $s$, the Lipschitz character of $\Omega$, and $\delta(x_0)$.
\end{thm}

Thus, $d\omega_s^{x_0}/d\sigma_s$ belongs quantitatively to the analog of the $RH_2(\sigma_s)$-class in which the reverse-H\"{o}lder condition is required only on boundary-centered balls. In particular, Theorem \ref{thm:Dahlberg-intro} yields an $A_{\infty}$-type comparison between fractional harmonic measure and $\sigma_{s}$. By a weighted Gehring lemma (see Lemma \ref{lem:gehring}), this estimate self-improves: there exists $p_0>2$ such that the density belongs to the corresponding $RH_p(\sigma_s)$-class for every $p\in[2,p_0)$. The result is in fact proved for the larger family of doubling exterior measures (even if, in this case, Remark \ref{rmk:sto1} does not apply for the whole range of $\beta$):
\begin{equation}\label{eq:def-sigma-beta-intro}
     \sigma_\beta 
    =(1-\beta)\frac{\delta^{-\beta}}
    {1+\delta^{n+2s-\beta}}\,\mathscr{L}^{n}\res \Omega^{c},
    \qquad \beta\in(2s-1,s];
\end{equation}
see Theorem \ref{thm:Dahlberg-general}. Moreover, the same conclusions hold for symmetric stable operators whose angular density is bounded above and below; see Section \ref{sec:general_kernels}. In particular, the result is not tied to rotational invariance.

This complements the existing potential theory for stable processes, including boundary Harnack principles, pointwise estimates for Green functions and Poisson kernels, and Martin representations
\cite{bogdan1997boundaryHarnack,chenSong1998green,chen1998martin,bogdan1999representationMartin,jakubowski2002estimates,bogdan2015boundary,armstrong2025caloric}. It is also distinct from the variational, trace-space, and weighted exterior data theories developed in \cite{felsinger2015dirichlet,bogdan2020extension,grubeHensiek2024trace,grube2024dirichlet,grubeKassmann2025trace}.  

The $RH_2$-type estimate in Theorem \ref{thm:Dahlberg-intro} leads naturally to an $L^2$-solvability theory for the fractional Dirichlet problem with respect to the exterior weight $\sigma_{s}$. The relevant non-tangential approach regions are necessarily nonlocal in nature: for every ``vertex'' $y\in\Omega^c$, the corresponding nonlocal cone inside $\Omega$ is defined as
\begin{equation}\label{eq:def-non-tangential-cone-Intro}
    \Gamma(y)\coloneqq\left\{x\in\Omega:\delta(x)\ge\frac{|x-y|}{2}\right\}.
\end{equation}
Then, for any function $u$ in $\Omega$, the nonlocal non-tangential maximal function $u^{*}:\Omega^{c}\to [0,\infty]$ is given by
\begin{equation}\label{eq:non-tangential-max-function-Intro}
    u^*(y)\coloneqq\sup_{x\in\Gamma(y)}|u(x)|\qquad \forall y\in \Omega^{c},
\end{equation}
with the convention that $u^*(y)\coloneqq 0$ when $\Gamma(y)=\varnothing$; see Definition \ref{def:non-tangential-maximal-function}. Clearly, when restricted to the boundary of the domain, $\Gamma$ and $u^{*}$ reduce to the standard notions of non-tangential cone and non-tangential maximal function adopted in the local theory. In the nonlocal case, we obtain the following result, which is the first nonlocal analog of \cite{dahlberg1979poisson,jerison1981dirichlet}:

\begin{thm}[$L^2$-solvability of the exterior Dirichlet problem]\label{thm:Dirichlet-solvability-intro}
    Let $\Omega\subset\R^n$ be a bounded Lipschitz domain and let $\sigma_s$ be the exterior measure in \eqref{eq:def-sigma}. For every $g\in L^2(\Omega^c,\sigma_s)$, the function
    \begin{equation}\label{eq:Poisson-solution-Dirichlet-Intro}
        u(x)\coloneqq\int_{\Omega^c}P^x(y)g(y)\,dy
        \qquad \forall x\in\Omega
    \end{equation}
    is well-defined, and once extended by $u=g$ on $\Omega^c$, is the unique distributional solution of
    \begin{equation*}
    \left\{
    \begin{array}{rclll}
         (-\Delta)^{s}u&=&0\quad &\text{in $\Omega$},\\
            u&=&g\quad &\text{in $\Omega^{c}$},\\
            \lVert u^{*}\rVert_{L^{2}(\Omega^{c}, \sigma_{s})}&<&\infty,
    \end{array}
    \right.
    \end{equation*}
    where $u^{*}$ is the nonlocal non-tangential maximal function in \eqref{eq:non-tangential-max-function-Intro}.
    Moreover,
    \begin{equation}\label{eq:L2-estimate-dirichlet-intro}
        \lVert u^*\rVert_{L^2(\Omega^c,\sigma_s)}
        \le C\lVert g\rVert_{L^2(\Omega^c,\sigma_s)},
    \end{equation}
    where $C$ depends only on $n$, $s$, and the Lipschitz character of $\Omega$.
\end{thm}

The conclusion is stronger than variational well-posedness: the datum is measured in a weighted $L^2$-space on the exterior, the solution is represented by harmonic measure, and we control a nonlocal non-tangential maximal function. The maximal function bound then implies in particular corresponding $L^2$-estimates in the interior. The self-improvement of Theorem \ref{thm:Dahlberg-intro} yields a stronger result for Theorem \ref{thm:Dirichlet-solvability-intro} as a consequence: there exists $q_0\in [1,2)$ such that, for every $q>q_{0}$, the Dirichlet problem is solvable for exterior data in $L^{q}(\Omega^{c},\sigma_{s})$. Here $q_{0}$ can be chosen as the dual exponent of $p_{0}>2$, the largest power for which a reverse-H\"{o}lder as in Theorem \ref{thm:Dahlberg-intro} holds. Moreover, we obtain analogous results for the whole family of weights $\sigma_\beta$ in \eqref{eq:def-sigma-beta-intro}; see Theorem \ref{thm:Dirichlet-problem-general}. Theorem \ref{thm:Dirichlet-problem-general-LTheta} gives its anisotropic stable-operator counterpart.

To the best of our knowledge, results related to the content of Theorem \ref{thm:Dirichlet-solvability-intro} have only been obtained so far in situations in which a stronger explicit expansion of the Poisson kernel close to the boundary is available. This is the case, for instance, when $\Omega$ has $C^{1,\alpha}$-boundary (see \cite{grube2024dirichlet,XROXFRbook} and references therein) or when $\Omega$ is a circular cone \cite{bogdan2005probleme}. 

\begin{rmk}[Weaker exterior data]
    We remark that an even weaker class of exterior data for which a solvability result as in Theorem \ref{thm:Dirichlet-solvability-intro} holds consists of functions which are in $L^{2}$ with respect to the weight $(1-s)\delta^{-s}$ close to $\partial \Omega$, and only $L^{1}$ with respect to $(1-s)\delta^{-n-2s}$ in the tails far from $\partial \Omega$. We do not follow this line here, for the sake of readability. However, one can adapt the proofs to extend to this broader class, with few technical difficulties. 
\end{rmk}
\begin{rmk}[The nonlocal-to-local limit]
\label{rmk:sto1}
    A quick inspection of the proof reveals that the constants appearing in the reverse-H\"{o}lder inequality \eqref{eq:reverse-holder-intro} and in the $L^{2}$-estimate \eqref{eq:L2-estimate-dirichlet-intro} are stable as $s\to 1^{-}$. In this limit, the reference measures $\sigma_{s}$
    converge towards the surface measure of $\partial\Omega$. Likewise, the nonlocal cone in \eqref{eq:def-non-tangential-cone-Intro} becomes the usual non-tangential approach region for harmonic functions as the vertex $y$ approaches the boundary, and the $L^2(\Omega^c,\sigma_{s})$-norm of the nonlocal maximal function \eqref{eq:non-tangential-max-function-Intro} corresponds in the limit to the $L^2(\partial\Omega)$-norm of the usual non-tangential maximal function considered in local problems. Thus, in the nonlocal-to-local limit, Theorems \ref{thm:Dahlberg-intro} and \ref{thm:Dirichlet-solvability-intro} can be used to recover Dahlberg's classical theorem and $L^2$-solvability of the Dirichlet problem for the Laplacian in Lipschitz domains. The self-improved estimates then give the corresponding nearby $L^q$-range.
\end{rmk}

\subsection{Ideas of the proof}\label{subsec-ideas-proof}

The proof requires a genuinely nonlocal substitute for the strategy adopted in the classical theory \cite{jerisonkenig1980identity}. The main new difficulty is that fractional harmonic measure is distributed throughout the whole exterior, while the estimate to be proved is quadratic and scale invariant. We will need to extract such an estimate from a global Pohozaev identity for smooth domains and to localize it in a form that remains uniform on rough boundaries.

Let $\Omega\subset \R^{n}$ be a smooth bounded domain, $x_{0}\in \Omega$, and let $G^{x_{0}}$ and $P^{x_{0}}$ denote respectively the Green function and the Poisson kernel with pole $x_{0}$. By the smoothness of $\partial \Omega$ and the results in \cite{rosoton2014dirichlet}, the Green function has a well-defined and continuous normal fractional derivative $\partial_{\nu}^{s}G^{x_{0}}\coloneqq (G^{x_{0}}/\delta^{s})|_{\partial \Omega}$ at the boundary.  
The starting point is the Pohozaev identity in Lemma \ref{lem:Pohozaev} for the Green function, which takes the form
\begin{equation}\label{eq:pohozaev-strategy-intro}
     \int_{\partial \Omega}(\partial_{\nu}^{s}G^{x_{0}}(\xi))^{2}(\xi-x_{0})\cdot \nu \,d\mathcal{H}^{n-1}(\xi)=C(n,s)\int_{\Omega^{c}}\frac{P^{x_{0}}(y)}{|y-x_{0}|^{n-2s}}\,dy.
\end{equation}
See also \cite{rosoton2014pohozaev,rosoton2017pohozaev,djitteSueur2023representation,dieb2025note}.
This identity is global and the boundary weight appearing in the left-hand side has no fixed sign on a general domain, so it does not directly imply a local square estimate for the fractional normal derivative $\partial_{\nu}^{s}G^{x_{0}}$. We overcome this obstruction by constructing auxiliary star-shaped domains that agree with the original domain on a prescribed boundary patch. Combined with a change of pole formula, boundary Harnack estimates, and localization, this turns \eqref{eq:pohozaev-strategy-intro} into a scale-invariant $L^{2}(\partial \Omega)$-estimate for $\partial_{\nu}^{s}G^{x_{0}}$: 
\begin{equation}\label{eq:L2-a-priori-boundary-estimate-Intro}
        \left(\fint_{B_{r}(\xi)\cap \partial\Omega}(\partial_{\nu}^{s}G^{x_{0}})^{2}d\mathcal{H}^{n-1}\right)^{1/2}\approx \frac{G^{x_{0}}(A_{r}(\xi))}{r^{s}}.
\end{equation}
Here, $A_r(\xi)\in \Omega$ is an interior corkscrew point at small distance $r$ from a boundary point $\xi$ and such that $\delta(A_r(\xi))\approx r$. This localization of the Pohozaev identity is the first key step of the proof, and is completed in Proposition \ref{prop:L2-normal-a-priori}.

At this point, the principal geometric difficulty is that for a general Lipschitz domain $\Omega$, a pointwise fractional normal trace is not available on all of $\partial\Omega$. Therefore, we cannot directly make sense of \eqref{eq:L2-a-priori-boundary-estimate-Intro} on Lipschitz boundaries by approximation with smooth domains. Instead, we need to transfer the
boundary estimate \eqref{eq:L2-a-priori-boundary-estimate-Intro} to the interior parallel surfaces of $\Omega$. There, the uniform stability of the Green functions with respect to the Hausdorff topology as the domain varies allows us to properly pass to the limit in the smooth approximation.   
This is done in Theorem \ref{thm:uniform-L2-Green-function}, where we show that the following local square estimate on the Green function holds on the parallel level surface $\{\delta=t\}$, for every $0<t<cr$:
\begin{equation}\label{eq:green-level-strategy-intro}
    \frac{1}{t^s}
    \left(
    \fint_{B_r(\xi)\cap\Omega\cap\{\delta=t\}}
    \bigl(G^{x_0}\bigr)^2\,d\mathcal H^{n-1}
    \right)^{1/2}
    \approx
    \frac{G^{x_0}(A_r(\xi))}{r^s}.
\end{equation}
The essential point is that the constants are independent of the level
$t$. 

A second nonlocal comparison converts
\eqref{eq:green-level-strategy-intro} into a square estimate for the Poisson
kernel on the exterior parallel surfaces. In
Corollary \ref{cor:uniform-L2-Poisson-kernel} we obtain
\begin{equation}\label{eq:poisson-level-strategy-intro}
    \frac{t^s}{1-s}
    \left(
    \fint_{B_r(\xi)\cap\Omega^c\cap\{\delta=t\}}
    \bigl(P^{x_0}\bigr)^2\,d\mathcal H^{n-1}
    \right)^{1/2}
    \approx
    \frac{\omega_s^{x_0}(B_r(\xi)\cap \Omega^{c})}{r^{n-s}}.
\end{equation}
This estimate is the bridge between the codimension-one information provided by the Pohozaev identity and the full-dimensional exterior measure appearing in the statement of Dahlberg's theorem. The coarea formula integrates \eqref{eq:poisson-level-strategy-intro} against the singular distance weight in \eqref{eq:def-sigma} (and actually also against any weight in the family \eqref{eq:def-sigma-beta-intro}), producing the $RH_2$-type-estimate in \eqref{eq:reverse-holder-intro}. A weighted Gehring lemma adapted to these exterior measures (see Lemma \ref{lem:gehring}) then yields the strict gain $p_0>2$; see Theorem \ref{thm:Dahlberg-general}.

The Dirichlet theory presents a further difficulty that has no direct local counterpart: both the datum and the maximal function are indexed by points of the full complement $\Omega^c$. We introduce a Martin kernel representation adapted to this exterior geometry and prove that the nonlocal non-tangential maximal function of a general Poisson integral is controlled by a boundary-centered Hardy-Littlewood maximal operator relative to harmonic measure; see Lemma \ref{lem:non-tangential-maximal-function-bound}. Combining this pointwise estimate with the $RH_2$-bound from Theorem \ref{thm:Dahlberg-intro} gives the $L^2$-estimate in
Theorem \ref{thm:Dirichlet-solvability-intro}. Finally, a separate uniqueness argument (see Lemma~\ref{lem:general-uniqueness}) shows that the maximal function condition determines the natural solvability class. Using the more general result in Theorem \ref{thm:Dahlberg-general} we get solvability for the broader range of exponents and weights; see Theorem \ref{thm:Dirichlet-problem-general}.

\subsection{Some applications of the theory}\label{subsec:applications}

As in the local case, the theory developed in this work provides a solid starting point for the study of fractional elliptic problems of various types in rough domains, including not only the Dirichlet problem, but also Neumann, Poisson, and regularity problems. Below, we present three corollaries of our results.

The first provides a uniform quadratic estimate on the parallel level sets above a Lipschitz graph for $s$-harmonic functions that vanish in the exterior, which addresses \cite[Open question 2.3]{XROXFRbook}. The proof follows from a localization argument combining Theorem \ref{thm:uniform-L2-Green-function} with the boundary Harnack principle.
\begin{cor}\label{cor:Linfty-L2-parallel-sets-harmonic}
    Let $\Omega=\{x=(x',x_{n}):x_{n}>\phi(x')\}$ be a Lipschitz epigraph, where $\phi:\R^{n-1}\to \R$ is $L$-Lipschitz and $\phi(0)=0$. Let $u\in C(B_{1})\cap L^{1}(\R^{n},w_{s})$ be a solution of 
    \begin{equation*}
    \left\{
    \begin{array}{rclll}
         (-\Delta)^{s}u&=&0\quad &\text{in $B_{1}\cap \Omega$},\\
            u&=&0\quad &\text{in $B_{1}\setminus \Omega$}.
    \end{array}
    \right.
    \end{equation*}
    Then, 
     \begin{equation}\label{eq:Linfty-L2-lip-graph-intro}
        \limsup_{t \to 0^+} \left(\fint_{B_{1/2}\cap \Omega\cap \{\delta=t\}}\left(\frac{u}{\delta^{s}}\right)^{2}\,d\mathcal{H}^{n-1}\right)^{1/2}\le C\left(\lVert u\rVert_{L^{\infty}(B_{1})}+(1-s)\lVert u\rVert_{L^{1}(\R^{n},w_{s})}\right),
    \end{equation}
    where $C$ depends only on $n,s$ and $L$ and is uniform as $s\to 1^{-}$.
\end{cor}

Next, as a consequence of the $L^{2}$-theory from Theorem \ref{thm:Dirichlet-solvability-intro}, we show that exterior weighted $L^{2}$-data generate interior fractional Sobolev regularity for solutions of the homogeneous Dirichlet problem in bounded Lipschitz domains. The proof combines the estimate \eqref{eq:L2-estimate-dirichlet-intro} for the non-tangential maximal function with the fractional Caccioppoli inequality from Lemma \ref{lem:Caccioppoli} applied at all scales. 
\begin{cor}\label{cor:Sobolev-Dirichlet}
    Let $\Omega\subset \R^{n}$ be a bounded Lipschitz domain, and let $\beta\in (\max\{2s-1, 0\},s]$. Given $g\in L^{2}(\Omega^{c},\sigma_{\beta})$, let $u$ be the Poisson integral solution defined in \eqref{eq:Poisson-solution-statement-Lp-solvability} of the Dirichlet problem $(-\Delta)^{s}u=0$ in $\Omega$ with exterior datum $g$. Then, $u\in H^{\beta/2}(\Omega)$ and 
    \begin{equation}\label{eq:bound-Hbetahalfs-dirichlet}
        \lVert u\rVert_{H^{\beta/2}(\Omega)}\le C\lVert g\rVert_{L^{2}(\Omega^{c},\sigma_{\beta})},
    \end{equation}
    where $C$ depends only on $n, s, \beta$, and the Lipschitz character of $\Omega$. 
\end{cor}
Related results were previously obtained in \cite{grube2024dirichlet} for $C^{1,\alpha}$ domains, for the natural regime $\beta=s$. We remark that the exponent $\beta/2$ is optimal, already for $\Omega=B_1$: for every $0<\eta<(1-\beta)/2$, the datum 
$
g_\eta(y)=\mathbbm{1}_{B_2\setminus B_1}(y)(|y|-1)^{(\beta+\eta-1)/2}
$
belongs to $L^2(B_1^c,\sigma_\beta)$, whereas its $s$-harmonic Poisson extension does not belong to $H^{\beta/2+\eta}(B_1)$, by the same Poisson-kernel and Hardy-inequality argument as in \cite[Remark~1.5]{grube2024dirichlet}.

Finally, by means of a duality argument, we may pass from the estimates obtained for the homogeneous Dirichlet problem in Corollary \ref{cor:Sobolev-Dirichlet} to corresponding Sobolev bounds for the inhomogeneous Poisson problem with zero exterior data:  

\begin{cor}\label{cor:optimal-Sobolev-Poisson} 
    Let $\Omega\subset \R^n$ be a bounded Lipschitz domain, and let $\beta\in[s,\min\{s+1/2,2s\})$. Given $f\in H^{\beta-2s}(\Omega)$, let $u\in H^{s}(\R^n)$ be the weak solution of
    \begin{equation*}
       \left\{
       \begin{array}{rclll}
         (-\Delta)^s u&=&f &\text{in }\Omega,\\
            u&=&0 &\text{in }\Omega^c.
       \end{array}
       \right.
    \end{equation*}
    Then $u\in H^\beta(\R^n)$ and
    \begin{equation*}
       \lVert u\rVert_{H^\beta(\R^n)}
       \le C\lVert f\rVert_{H^{\beta-2s}(\Omega)},
    \end{equation*}
    where $C$ depends only on $n$, $s$, $\beta$, and the Lipschitz character of $\Omega$.
\end{cor}

Sharp Besov regularity on bounded Lipschitz domains was previously obtained in \cite{borthagaray2023besov} by different methods based on the local difference quotient technique from \cite{savare1998regularity}. Accordingly, the contribution of Corollary \ref{cor:optimal-Sobolev-Poisson} is a new derivation from the exterior weighted $L^2$-theory and harmonic measure estimates. The upper threshold $s+1/2$ is sharp in general. In the unit ball, the solution of $(-\Delta)^su=1$ with zero exterior values is a constant multiple of $(1-|x|^2)_+^s$ \cite{dyda2012fractional}; the resulting failure of the endpoint $H^{s+1/2}(\R^n)$ regularity is consistent with the sharp Besov analysis in \cite{borthagaray2023besov}.

\subsection{Anisotropic operators}

The method is not specific to the rotationally invariant fractional
Laplacian. In Section \ref{sec:general_kernels} we extend the main results to symmetric,
translation-invariant stable operators of the form
\begin{equation*}
    \mathcal{L}_{\Theta} u(x)= c_{n, s}\,{\rm P.V.}\int_{\mathbb R^n}\frac{u(x)-u(y)}{|x-y|^{n+2s}}\Theta\left(\frac{y-x}{|y-x|}\right)\,dy,
\end{equation*}
where the angular density $\Theta$ is even and uniformly bounded above and below:
\begin{equation*}
    0<\lambda\le \Theta(\vartheta)\le \Lambda<\infty,
    \qquad
    \Theta(\vartheta)=\Theta(-\vartheta)\qquad \text{for a.e. $\vartheta\in \Sph^{n-1}$}.
\end{equation*}
For this class, the necessary boundary Harnack and Green function estimates remain available with constants depending additionally on the ellipticity constants $\lambda, \Lambda$. The corresponding Pohozaev identity contains an anisotropic boundary factor, but uniform ellipticity bounds this factor above and below. Consequently, the
the same strategy illustrated in Section \ref{subsec-ideas-proof} continues to apply. In this setting, the analogues of Theorems~\ref{thm:Dahlberg-intro} and \ref{thm:Dirichlet-solvability-intro} are Theorems \ref{thm:Dahlberg-general-LTheta} and \ref{thm:Dirichlet-problem-general-LTheta}, respectively. Likewise, the analogues of Corollaries \ref{cor:Linfty-L2-parallel-sets-harmonic}, \ref{cor:Sobolev-Dirichlet}, and \ref{cor:optimal-Sobolev-Poisson} are Corollaries \ref{cor:Linfty-L2-parallel-sets-LTheta}, \ref{cor:Sobolev-Dirichlet-LTheta}, and \ref{cor:optimal-Sobolev-Poisson-LTheta}, respectively.
Thus, the arguments used are stable under uniformly elliptic anisotropic perturbations of the kernel and provide a framework that
can be adapted to a broader class of nonlocal operators. 

\subsection{Organization of the paper}

In Section \ref{sec:prelim}, we introduce the notation and collect basic comparability estimates for $s$-harmonic measures, Green functions, and Poisson kernels in bounded Lipschitz domains. In Section \ref{sec:Dahlberg}, we prove the fractional Dahlberg theorem following the strategy outlined in Section \ref{subsec-ideas-proof}: we begin with the Green function Pohozaev identity, derive square estimates for fractional normal derivatives in smooth domains, transfer them to parallel surfaces of Lipschitz domains, and conclude with the $RH_2$-type estimate and its self-improvement. In Section \ref{sec:Dir_pb}, we consider the Dirichlet problem. We first prove a general comparison between nonlocal non-tangential maximal functions of Poisson integrals and Hardy--Littlewood maximal functions relative to harmonic measure. We then combine this result with the reverse-H\"{o}lder inequality and a separate uniqueness criterion to prove the weighted $L^{2}$-solvability theorem and its extension to a broader class of exponents and weights. Section \ref{sec:proof-corollaries} contains the proofs of the three corollaries stated in Section \ref{subsec:applications}, while Section \ref{sec:general_kernels} extends the main results to symmetric stable operators comparable to the fractional Laplacian. Finally, the appendix contains technical tools used throughout, such as the fractional Caccioppoli inequality and the weighted Gehring lemma.

\section{Preliminaries}
\label{sec:prelim}
In this section, we fix our main notations and collect some preliminary results that will be used throughout. We start by recalling the definition of fractional Laplacian, related function spaces, and notion of solutions to fractional elliptic problems. Next, after fixing our conventions for the Lipschitz character of a domain, we recall two fundamental H\"{o}lder regularity results for $s$-harmonic functions up to the Lipschitz boundary. Finally, we introduce $s$-harmonic measure and associated Green function and Poisson kernel, and we conclude the section by proving several useful comparability results involving these objects. 

\subsection{Some basic notation}

\paragraph{The fractional Laplacian.} Let $n\ge 2$ be the dimension and $s\in (0,1)$ the fractional exponent. We call $w_{s}\in \mathcal{M}_{+}(\R^{n})$ the weighted measure 
\begin{equation*}
    w_{s}\coloneqq\frac{1}{1+|x|^{n+2s}}\mathscr{L}^{n}.
\end{equation*}
For every open set $\Omega\subseteq \R^{n}$, if $u\in C^{2s+\eps}_{\loc}(\Omega)\cap L^{1}(\R^{n}, w_{s})$ for some $\eps>0$, the fractional Laplacian of $u$ is defined pointwise in $\Omega$ as
\begin{equation}
\label{eq:fract_laplacian}
    (-\Delta)^{s}u(x)\coloneqq c_{n,s}\,{\rm P.V.}\int_{\R^{n}}\frac{u(x)-u(y)}{|x-y|^{n+2s}}\,dy,\qquad c_{n,s}\coloneqq 2^{2s}s\frac{\Gamma\left(\frac{n+2s}{2}\right)}{\Gamma(1-s)}\pi^{-\frac{n}{2}}.
\end{equation}
Given $u,v\in C^{2s+\eps}_{\loc}(\Omega)\cap L^{2}(\R^{n}, w_{s})$, the following identity holds:
\begin{equation*}
    (-\Delta)^{s}(uv)(x)=u(x)(-\Delta)^{s}v(x)+v(x)(-\Delta)^{s}u(x)-2B_{s}(u,v)(x)\qquad\forall x\in \Omega,
\end{equation*}
where the carré du champ $B_{s}(u,v):\Omega\to \R$ is given by
\begin{equation}\label{eq:def-quadratic-form}
    B_{s}(u,v)(x)\coloneqq \frac{c_{n,s}}{2}\int_{\R^{n}}\frac{(u(x)-u(y))(v(x)-v(y))}{|x-y|^{n+2s}}\,dy.
\end{equation}
Given any distribution $f\in \mathscr{D}'(\Omega)$, we say that $u\in L^{1}(\R^{n}, w_{s})$ is a distributional solution of
\begin{equation*}
    (-\Delta)^{s}u=f\qquad \text{in $\Omega$}
\end{equation*}
whenever 
\begin{equation*}
    \int_{\R^{n}}u(-\Delta)^{s}\varphi=\langle f,\varphi\rangle\qquad \forall \varphi \in C^{\infty}_{c}(\Omega).
\end{equation*}
The fundamental solution of $(-\Delta)^{s}$ in $\R^{n}$, solving $(-\Delta)^{s}\Phi_{n,s}=\delta_{0}$ in the sense of distributions, is given by 
\begin{equation}\label{eq:def-fond-sol-frac-lap}
    \Phi_{n,s}(x)=\frac{\kappa_{n,s}}{|x|^{n-2s}},\qquad \kappa_{n,s}\coloneqq 2^{-2s}\frac{\Gamma\left(\frac{n-2s}{2}\right)}{\Gamma(s)}\pi^{-\frac{n}{2}}.
\end{equation}

\paragraph{Fractional Sobolev spaces.} For any open set $\Omega\subseteq \R^{n}$ and any $s\in (0,1)$, the fractional Sobolev space $H^{s}(\Omega)$ is defined as
\begin{equation*}
    H^{s}(\Omega)\coloneqq\left\{u\in L^{2}(\Omega): \lVert u\rVert_{H^{s}(\Omega)}<\infty\right\},
\end{equation*}
where the norm is
\begin{equation*}
    \lVert u\rVert_{H^{s}(\Omega)}\coloneqq \lVert u\rVert_{L^{2}(\Omega)}+[u]_{H^{s}(\Omega)},
\end{equation*}
and the seminorm is given by
\begin{equation*}
    [u]_{H^{s}(\Omega)}^{2}\coloneqq \frac{c_{n,s}}{2}\int_{\Omega}\int_{\Omega}\frac{(u(x)-u(y))^{2}}{|x-y|^{n+2s}}\,dx\,dy.
\end{equation*}
For any $\gamma>0$ such that $\gamma=k+s$, for some $k\in \N_{\ge 1}$ and $s\in (0,1)$, we also define higher order Sobolev spaces as
\begin{equation*}
    H^{\gamma}(\Omega)\coloneqq \left\{u\in H^{k}(\Omega): \lVert u\rVert_{H^{\gamma}(\Omega)}<\infty\right\},\qquad \lVert u\rVert_{H^{\gamma}(\Omega)}\coloneqq  \lVert u\rVert_{H^{k-1}(\Omega)}+\lVert \nabla^{k}u\rVert_{H^{s}(\Omega)}. 
\end{equation*}
Let $H_{0}^{\gamma}(\Omega)$ be the closure of $C^{\infty}_{c}(\Omega)$ with respect to the $\lVert\cdot\rVert_{H^{\gamma}}$-norm. Then, we define the negative order Sobolev space $H^{-\gamma}(\Omega)$ as the dual of $H^{\gamma}_{0}(\Omega)$, i.e.
\begin{equation*}
    H^{-\gamma}(\Omega)\coloneqq\left\{f\in \mathscr{D}'(\Omega): \lVert f\rVert_{H^{-\gamma}(\Omega)}<\infty\right\},\qquad \lVert f\rVert_{H^{-\gamma}(\Omega)}\coloneqq \sup\left\{\frac{\langle f,\varphi\rangle}{\lVert \varphi\rVert_{H^{\gamma}(\Omega)}}:\varphi \in C^{\infty}_{c}(\Omega), \varphi\neq 0\right\}.
\end{equation*}
For given $s\in (0,1)$ and $f\in H^{-s}(\Omega)$, we say that $u\in H^{s}_{\loc}(\Omega)\cap L^{1}(\R^{n},w_{s})$ is a weak solution of 
\begin{equation*}
    (-\Delta)^{s}u=f\qquad \text{in $\Omega$}
\end{equation*}
whenever 
\begin{equation*}
    \frac{c_{n,s}}{2}\int_{\R^{n}}\int_{\R^{n}}\frac{(u(x)-u(y))(\varphi(x)-\varphi(y))}{|x-y|^{n+2s}}\,dx\,dy= \langle f,\varphi\rangle\qquad \forall \varphi\in C^{\infty}_{c}(\Omega).
\end{equation*}

\paragraph{Lipschitz domains.} Given $L>0$, we say that $\Omega\subset \R^{n}$ is a Lipschitz epigraph with Lipschitz constant $L>0$ if there is an orthonormal coordinate system $x=(x',x_{n})\in \R^{n-1}\times \R$ and a Lipschitz function $\phi:\R^{n-1}\to \R$ such that $\phi(0)=0$, $|\nabla \phi|\le L$, and 
\begin{equation*}
    \Omega=\{(x',x_{n})\in \R^{n}:x_{n}>\phi(x')\}.
\end{equation*}

\begin{defn}
\label{defi:Lip_char}
    Given $\bar r,L>0, S>0$, we say that $\Omega\subset \R^{n}$ is a bounded Lipschitz domain with Lipschitz character $(\bar r,L,S)$ if $\diam \Omega \le S$ and for every $\xi\in \partial \Omega$, there is a Lipschitz epigraph $\Omega_{\xi}\subset \R^{n}$ such that 
    \begin{equation*}
        \Omega \cap B_{\bar r}(\xi)=\left(\xi+\Omega_{\xi}\right)\cap B_{\bar r}(\xi).
    \end{equation*}
    We call $\bar r$ the localization radius, $L$ the Lipschitz constant, and $S$ the size of the domain $\Omega$, respectively.
\end{defn}

In the sequel we will use the notation $\delta_{\Omega}$ to denote the distance function from the boundary of the domain $\dist(\cdot,\partial \Omega):\R^{n}\to [0,\infty)$. Note that the distance function $\delta_{\Omega}$ will be used on both sides, $\Omega$ and $\Omega^{c}$. When clear from the context, we will often drop the subscript $\Omega$.

Any bounded Lipschitz domain $\Omega\subset \R^{n}$ satisfies the interior and exterior corkscrew conditions. Namely, there exists a constant $c\in (0,1)$ depending only on the Lipschitz constant $L$ such that the following holds: for every $r\in (0,\bar r/2)$, with $\bar r$ the localization radius, and any $\xi \in \partial \Omega$, there exist points $A_{r}(\xi)\in \Omega$ and $A_{r}'(\xi)\in \interior \Omega^{c}$ such that
    \begin{equation}\label{eq:corkscrew-condition}
        \delta_{\Omega}(A_{r}(\xi)),\, \delta_{\Omega}(A'_{r}(\xi))\ge cr\qquad \text{and}\qquad |A_{r}(\xi)-\xi|,\,|A'_{r}(\xi)-\xi|=r.
    \end{equation}
The points $A_{r}(\xi)$ and $A'_{r}(\xi)$ are called, respectively, interior and exterior corkscrew points at distance $r$ from the boundary point $\xi$. The precise choice of corkscrew points will not play any role in the sequel, once condition \eqref{eq:corkscrew-condition} is satisfied. 

\paragraph{Fractional normal derivatives.} For any $\beta\in \R$, and any function $u:\Omega\to \R$ such that $u/\delta^{\beta}$ extends continuously to $\overline{\Omega}$, we denote
\begin{equation}\label{eq:def-fractional-normal-der}
    \partial_{\nu}^{\beta}u(\xi)\coloneqq \lim_{\Omega\ni x\to \xi}\frac{u(x)}{\delta^{\beta}(x)}\qquad \forall \xi \in \partial \Omega.
\end{equation}
We recall that $s$-harmonic functions in smooth domains vanishing continuously in the exterior have a well-defined and continuous fractional normal derivative $\partial_{\nu}^{s}u$ on $\partial \Omega$; see \cite{rosoton2014dirichlet}.

\subsection[The s-harmonic measure in Lipschitz domains]{The $s$-harmonic measure in Lipschitz domains}\label{subsec:harmonic-measure}

\paragraph{H\"{o}lder regularity up to the boundary for $s$-harmonic functions.}

Here, we recall two fundamental boundary regularity results for $s$-harmonic functions in Lipschitz domains. We refer the reader to \cite{bogdan1997boundaryHarnack}, \cite[Lemmas 3 and 4]{bogdan1999representationMartin}, \cite{caffarelli2018bounds}, \cite[Chapter 3]{XROXFRbook}, and references therein for several different proofs.

\begin{thm}\label{thm:holder-boundary-lip-domain}
    Let $\Omega\subset \R^{n}$ be a Lipschitz epigraph with Lipschitz constant $L>0$, and let $u\in C(B_{r})\cap L^{1}(\R^{n}, w_{s})$, $u\ge 0$ be a solution of 
        \begin{equation*}
    \left\{
    \begin{array}{rclll}
         (-\Delta)^{s}u&=&0\quad &\text{in $B_{r}\cap \Omega$},\\
            u&=&0\quad &\text{in $B_{r}\setminus \Omega$}.
    \end{array}
    \right.
    \end{equation*}
    Then, there exist constants $C>0$ and $\alpha\in (0,s)$ depending only on $n,s$ and $L$ such that $u\in C^{\alpha}(B_{r/2})$, and  
    \begin{equation*}
        \frac{1}{C}\left(\frac{\delta(x)}{r}\right)^{2s-\alpha}u\left(\frac{re_{n}}{2}\right)\le u(x)\le C\left(\frac{\delta(x)}{r}\right)^{\alpha}u\left(\frac{re_{n}}{2}\right)\qquad \forall x\in B_{r/2}\cap \Omega.
    \end{equation*}
\end{thm}
\begin{thm}[Boundary Harnack principle]\label{thm:BHP}
    Let $\Omega\subset \R^{n}$ be a Lipschitz epigraph with Lipschitz constant $L>0$, and let $u_{1}, u_{2} \in C(B_{r})\cap L^{1}(\R^{n}, w_{s})$, $u_{i}\ge 0$ be solutions of 
    \begin{equation*}
        \left\{
    \begin{array}{rclll}
         (-\Delta)^{s}u_{i}&=&0\quad &\text{in $B_{r}\cap \Omega$},\\
            u_{i}&=&0\quad &\text{in $B_{r}\setminus \Omega$}.
    \end{array}
    \right.\qquad i=1,2,\qquad  u_{1}\left(\frac{re_{n}}{2}\right)=u_{2}\left(\frac{re_{n}}{2}\right)>0.
    \end{equation*}
    Then, there exist constants $C>0$ and $\beta \in (0,1)$ such that the ratio $h\coloneqq u_{1}/u_{2}\in C^{\beta}\left(\overline{B_{r/2}\cap\Omega}\right)$, and the following holds:
    \begin{equation*}
        \frac{1}{C}\le h(x)\le C,\qquad |h(x)-h(y)|\le C\left(\frac{|x-y|}{r}\right)^{\beta}\qquad \forall x,y\in B_{r/2}\cap \Omega.
    \end{equation*}
\end{thm}

\paragraph{The $s$-harmonic measure: Poisson kernel and Green function.}
Next, we recall the construction of $s$-harmonic measure. 
Let $\Omega\subset \R^{n}$ be a Lipschitz domain. For every bounded continuous function $g\in C_{b}(\Omega^{c})$, there is a unique distributional solution $u_{g}\in C_{b}(\R^{n})$ of the Dirichlet problem below (see, for instance, \cite{bogdan1999representationMartin,abatangelo2013large,XROXFRbook})
\begin{equation*}
        \left\{
    \begin{array}{rclll}
         (-\Delta)^{s}u_{g}&=&0\quad &\text{in $\Omega$},\\
            u_{g}&=&g\quad &\text{in $\Omega^{c}$}.
    \end{array}
    \right.
\end{equation*}
Note that $g\mapsto u_{g}$ is linear, $u_{1}\equiv 1$, and $\lVert u_{g}\rVert_{L^{\infty}}\le \lVert g\rVert_{L^{\infty}}$ by the maximum principle. 
In particular, for every $x\in \Omega$, duality defines the probability measure $\omega_{\Omega}^{x}\in\mathscr{P}(\R^{n})$ concentrated on $\Omega^{c}$ such that
\begin{equation*}
    \int_{\Omega^{c}} g\,d\omega_{\Omega}^{x}=u_{g}(x),\qquad \forall g\in C_{b}(\Omega^{c}).
\end{equation*}
$\omega_{\Omega}^{x}$ is called the $s$-harmonic measure of $\Omega$ with pole $x$. Compared with the notation used in the introduction, we drop the subscript $s$, since only nonlocal problems will be considered in the rest of the paper.

For every $x\in \Omega$, $\omega^{x}_{\Omega}$ is absolutely continuous with respect to $\mathscr{L}^{n}\res \Omega^{c}$ by \cite[Lemma 6]{bogdan1997boundaryHarnack}, and the density
\begin{equation*}
    P_{\Omega}^{x}\coloneqq \frac{d\omega^{x}_{\Omega}}{d\mathscr{L}^{n}}:\interior \Omega^{c}\to [0,\infty)
\end{equation*}
is called the Poisson kernel of $\Omega$ with pole $x$. The function $\Omega \times \interior \Omega^{c}\ni (x,y)\mapsto P^{x}_{\Omega}(y)$ is in fact $C^{\infty}$ (\cite[Remark 2]{bogdan1997boundaryHarnack}).

Finally, for every $x\in \Omega$, the Green function of $\Omega$ with pole $x$ is defined as
\begin{equation}\label{eq:def-green-function}
    G_{\Omega}^{x}(y)\coloneqq \Phi_{n,s}(x-y)-R_{\Omega}^{x}(y),
\end{equation}
where $\Phi_{n,s}$ is the fundamental solution of $(-\Delta)^{s}$ in the full space $\R^{n}$ from \eqref{eq:def-fond-sol-frac-lap}, and $R_{\Omega}^{x}$ is the solution of 
\begin{equation*}
        \left\{
    \begin{array}{rclll}
         (-\Delta)^{s}R_{\Omega}^{x}&=&0\quad &\text{in $\Omega$},\\
            R_{\Omega}^{x}&=&\Phi_{n,s}(x-\cdot)\quad &\text{in $\Omega^{c}$}.
    \end{array}
    \right.
\end{equation*}
We have $G^{x}_{\Omega}(y)=G^{y}_{\Omega}(x)$ for every $x,y\in \Omega$, $x\neq y$, and moreover $G_{\Omega}^{x}$ is a distributional solution of
\begin{equation*}
        \left\{
    \begin{array}{rclll}
         (-\Delta)^{s}G_{\Omega}^{x}&=&\delta_{x}\quad &\text{in $\Omega$},\\
            G_{\Omega}^{x}&=&0\quad &\text{in $\Omega^{c}$}.
    \end{array}
    \right.
\end{equation*}
In addition, $P_{\Omega}^{x}$ and $G_{\Omega}^{x}$ are related by the formula
\begin{equation}\label{eq:poisson-kernel-integral-formula}
    P_{\Omega}^{x}(y)=-(-\Delta)^{s}G_{\Omega}^{x}(y)= c_{n,s}\int_{\Omega}\frac{G^{x}_{\Omega}(z)}{|z-y|^{n+2s}}dz\qquad \forall y\in \interior\Omega^{c}.
\end{equation}
From the definition of $G_{\Omega}^{x}$ and the maximum principle one infers the following basic estimates:
    \begin{align}
        G^{x}_{\Omega}(y)& \le \frac{\kappa_{n,s}}{|x-y|^{n-2s}}\qquad \forall x,y \in \Omega,\label{eq:basic-estimate-green-function-1}\\
        G^{x}_{\Omega}(y)& \gtrsim \frac{\kappa_{n,s}}{|x-y|^{n-2s}}\qquad \forall x,y \in \Omega: |x-y|\le \frac{\delta_{\Omega}(x)}{2}\label{eq:basic-estimate-green-function-2}.
\end{align}
In the sequel, when clear from the context, we will often drop the subscript $\Omega$ in $\omega^{x}_{\Omega}, P^{x}_{\Omega}$ and $G^{x}_{\Omega}$.

The following lemma proves a stability result for $s$-harmonic measure, Poisson kernel, and Green function under small perturbations of the domain in the Hausdorff topology. 
\begin{lem}\label{lem:approximation-smooth-domains}
    Let $\Omega_{j}\subset \R^{n}$ be a sequence of bounded domains with a common Lipschitz character converging to $\Omega\subset \R^{n}$ in the Hausdorff sense. Then, for every $x\in\Omega$ and all sufficiently large $j$, let $G_{j}^{x}$, $G^{x}$ denote the Green functions and $P^{x}_{j}$, $P^{x}$ the Poisson kernels with pole $x$ for $\Omega_{j}$ and $\Omega$, respectively. We have
    \begin{equation*}
        G_{j}^{x}\to G^{x}\quad \text{locally uniformly in $\R^{n}\setminus \{x\}$}\qquad \text{and}\qquad P^{x}_{j}\to P^{x}\quad \text{locally uniformly in $\interior \Omega^{c}$}.
    \end{equation*}
    In particular, calling $\omega_{j}^{x}, \omega^{x}\in \mathscr{P}(\R^{n})$ the $s$-harmonic measures with pole $x$ for the domain $\Omega_{j}$ and $\Omega$ respectively, we have
    \begin{equation*}
        \lVert\omega_{j}^{x}-\omega^{x}\rVert_{{\rm TV}}\to 0.
    \end{equation*}
\end{lem}
\begin{proof}
    By the boundary H\"{o}lder regularity of $s$-harmonic functions (Theorem \ref{thm:holder-boundary-lip-domain}) and the maximum principle, the functions $\Phi_{n,s}(\cdot-x)-G^{x}_{j}$ are uniformly bounded, uniformly equicontinuous in $\R^{n}$, and $s$-harmonic in their respective domains $\Omega_{j}$. Therefore, up to a subsequence, they converge locally uniformly to a limiting continuous function which is $s$-harmonic in $\Omega$ and equals $\Phi_{n,s}(\cdot-x)$ in $\Omega^{c}$, and thus necessarily coincides with $\Phi_{n,s}(\cdot-x)-G^{x}$. This proves the uniform convergence of the Green functions. At this point, applying the integral formula in \eqref{eq:poisson-kernel-integral-formula}, we deduce that $P^{x}_{j}$ converge to $P^{x}$ locally uniformly in $\interior \Omega^{c}$. Finally, since the Poisson kernel is the density of harmonic measure with respect to Lebesgue, Scheffé's lemma implies the convergence of $\omega^{x}_{j}$ to $\omega^{x}$ in total variation.
\end{proof}

\paragraph{Fundamental comparison lemmas.}
We conclude this section with some useful lemmas comparing $s$-harmonic measure of boundary balls and pointwise values of Green function and Poisson kernel in corres\-ponding corkscrew points. Some of these results were already contained in \cite{bogdan1997boundaryHarnack,caffarelli2018bounds}. For the reader's convenience, we give below complete proofs based only on the estimates contained in this section.

We begin with the comparison between harmonic measure of a boundary ball and the Green function at a corresponding interior corkscrew point.
\begin{lem}\label{lem:comparison-omega-G}
    Let $\Omega \subset \R^{n}$ be a bounded Lipschitz domain, and let $x_{0}\in \Omega$. Then, there exists $r_{0}>0$ depending only on the Lipschitz character of $\Omega$ such that, for all $\xi\in \partial \Omega$ and $r\in (0,r_{0})$ with $|x_{0}-\xi|>2r$, the following holds:
    \begin{equation}\label{eq:comparison-omega-G}
        \omega^{x_{0}}(B_{r}(\xi))\approx r^{n-2s}G^{x_{0}}(A_{r}(\xi)),
    \end{equation}
    where the comparability constants depend only on $n,s$, and the Lipschitz character of $\Omega$.
\end{lem}
\begin{proof}
    In what follows, $r_{0}$ is chosen sufficiently small with respect to the localization radius of the domain. 
    Let us call $u(x)\coloneqq \omega^{x}(B_{r}(\xi))$, identified with the $s$-harmonic function in $\Omega$ that has $\mathbbm{1}_{\Omega^{c}\cap B_{r}(\xi)}$ as exterior datum. 
    
    We first prove that $u(A_{r}(\xi))\approx 1$. Note that $v\coloneqq 1-u \in [0,1]$ is continuous in $B_{r}(\xi)$ and solves the problem
    \begin{equation*}
    \left\{
    \begin{array}{rclll}
         (-\Delta)^{s}v&=&0\quad &\text{in $B_{r}(\xi)\cap \Omega$},\\
            v&=&0\quad &\text{in $B_{r}(\xi)\setminus \Omega$}.
    \end{array}
    \right.
    \end{equation*}
    Therefore, by Theorem \ref{thm:holder-boundary-lip-domain}, given $\eps\in (0,1)$ we have $v(A_{\eps r}(\xi))\lesssim v(A_{r}(\xi))\eps^{\alpha}\le \eps^{\alpha}$. Hence, for $\eps$ small enough, $v(A_{\eps r}(\xi))\le 1/2$, or $u(A_{\eps r}(\xi))\ge 1/2$. At this point, calling $c\in (0,1)$ the constant that appears in the corkscrew condition \eqref{eq:corkscrew-condition}, a Harnack chain argument gives
    \begin{equation}\label{eq:bound-below-HM-corkscrew-comparability-proof}
       u(z)\approx 1\qquad \forall z\in B_{cr/2}(A_{r}(\xi)).
    \end{equation}

    Now we show that $u(x_{0})\gtrsim r^{n-2s}G^{x_{0}}(A_{r}(\xi))$. 
    For a constant $C_{0}>0$ to be chosen sufficiently large later, we consider the function
    \begin{equation*}
        w(x)\coloneqq \frac{1}{C_{0}}r^{n-2s}G^{x}(A_{r}(\xi))-u(x).
    \end{equation*}
    We wish to prove that $w\le 0$ in $\Omega \setminus B_{cr/4}(A_{r}(\xi))$. Notice that $w$ is nonpositive in $\Omega^{c}$. In addition, combining \eqref{eq:bound-below-HM-corkscrew-comparability-proof} with \eqref{eq:basic-estimate-green-function-1} and choosing $C_{0}$ sufficiently large, we may enforce
    \begin{equation}\label{eq:bound-w-below-proof-comparison}
        w(x)\le -c_{1}\qquad \forall x\in B_{cr/2}(A_{r}(\xi))\setminus B_{cr/4}(A_{r}(\xi)). 
    \end{equation}
    Finally, since $u\ge 0$, integrating \eqref{eq:basic-estimate-green-function-1} we also get
    \begin{equation}\label{eq:integral-bound-w-above-proof-comparison}
        \int_{B_{cr/4}(A_{r}(\xi))}w\le \frac{\kappa_{n,s}}{C_{0}}r^{n-2s}\int_{B_{cr/4}(A_{r}(\xi))}\frac{1}{|x-A_{r}(\xi)|^{n-2s}}\, dx\lesssim \frac{r^{n}}{C_{0}}.
    \end{equation}
    Therefore, if $w$ had a positive supremum in $\overline{\Omega \setminus B_{cr/4}(A_{r}(\xi))}$, this would be achieved at some point $z\in \Omega \setminus B_{cr/2}(A_{r}(\xi))$, and we would have
    \begin{align*}
        0=(-\Delta)^{s}w(z)&=c_{n,s} \,{\rm P.V.} \int_{\R^{n}}\frac{w(z)-w(x)}{|x-z|^{n+2s}}\,dx\\
        &\ge c_{n,s} \int_{B_{cr/2}(A_{r}(\xi))\setminus B_{cr/4}(A_{r}(\xi))}\frac{w(z)-w(x)}{|z-x|^{n+2s}}\,dx +c_{n,s} \int_{B_{cr/4}(A_{r}(\xi))}\frac{w(z)-w(x)}{|z-x|^{n+2s}}\,dx \\
        &\gtrsim c_{n,s}\frac{1}{|z-A_{r}(\xi)|^{n+2s}}\left(c_{1}r^{n}-\frac{r^{n}}{C_{0}}\right),
    \end{align*}
    where we used the maximality of $z$, the fact that $w(z)> 0$, and equations \eqref{eq:bound-w-below-proof-comparison}, \eqref{eq:integral-bound-w-above-proof-comparison}. Choosing $C_{0}$ sufficiently large, the right-hand side can be made strictly positive, leading to the desired contradiction. 

    Next, we prove $u(x_{0})\lesssim r^{n-2s}G^{x_{0}}(A_{r}(\xi))$. Consider a smooth cutoff function $\varphi\in C^{\infty}_{c}(B_{4r/3}(\xi))$ such that $\varphi \in [0,1]$, $\varphi \equiv 1$ in $B_{r}(\xi)$, and $|\nabla^{2}\varphi|\lesssim r^{-2}$. Note that 
    \begin{equation*}
        |(-\Delta)^{s}\varphi(x)|\lesssim \frac{1}{r^{2s}}\mathbbm{1}_{B_{3r/2}(\xi)}(x)+c_{n,s}\frac{r^{n}}{r^{n+2s}+\left|x-\xi\right|^{n+2s}}\qquad \forall x\in \R^{n}.
    \end{equation*}
    Then, since $x_{0}\notin \supp \varphi$, we obtain
    \begin{align*}
        u(x_{0})&\le \int_{\Omega^{c}}P^{x_{0}}\varphi=-\int_{\R^{n}}(-\Delta)^{s}G^{x_{0}}\varphi=-\int_{\R^{n}}G^{x_{0}}(-\Delta)^{s} \varphi\\
        &\lesssim \frac{1}{r^{2s}}\int_{B_{3r/2}(\xi)}G^{x_{0}}+c_{n,s}r^{n}\int_{\R^{n}}\frac{G^{x_{0}}(x)}{r^{n+2s}+|x-\xi|^{n+2s}}\,dx.
    \end{align*}
    Using Theorem \ref{thm:holder-boundary-lip-domain} and Lemma \ref{lem:half-harnack-tails}, respectively, both terms in the right-hand side can be estimated by $r^{n-2s}G^{x_{0}}(A_{r}(\xi))$. This concludes the proof.
\end{proof}
The preceding comparison immediately implies the local doubling property of $s$-harmonic measure away from its pole:
\begin{lem}\label{lem:doubling}
    Let $\Omega \subset \R^{n}$ be a bounded Lipschitz domain, and let $x_{0}\in \Omega$. Then, there exists $r_{0}>0$ depending only on the Lipschitz character of $\Omega$ such that, for all $\xi\in \partial \Omega$ and $r\in (0,r_{0})$ with $|x_{0}-\xi|>4r$, the following holds:
    \begin{equation*}
        \omega^{x_{0}}(B_{2r}(\xi))\lesssim \omega^{x_{0}}(B_{r}(\xi)),
    \end{equation*}
    with constant depending only on $n,s$, and the Lipschitz character of $\Omega$.
\end{lem}

\begin{proof}
    Let $r_{0}>0$ be the constant given by Lemma \ref{lem:comparison-omega-G}. Then, applying \eqref{eq:comparison-omega-G} and the Harnack inequality, we find
    \begin{equation*}
        \omega^{x_{0}}(B_{2r}(\xi))\approx (2r)^{n-2s} G^{x_{0}}(A_{2r}(\xi))\approx r^{n-2s}G^{x_{0}}(A_{r}(\xi))\approx \omega^{x_{0}}(B_{r}(\xi)),
    \end{equation*}
    as desired.
\end{proof}
\begin{rmk}
    In the previous lemma, doubling is proved only for sufficiently small boundary balls far from the pole. However, if we allow the constant to also depend on $\delta(x_{0})$, the doubling condition holds for boundary balls of any radius. This is because $\omega^{x_{0}}$ is a probability measure, and one can use \eqref{eq:comparison-omega-G} and the Harnack inequality to give a uniform lower bound (depending only on $n,s$, the Lipschitz character, and $\delta(x_{0})$) on the harmonic measure of boundary balls with radius comparable to $\min\{r_{0}, \delta(x_{0})\}$. 
\end{rmk}
The next estimate records how the harmonic measure changes as the pole varies. Equation \eqref{eq:change-of-pole-omega} below is often called a ``change of pole formula''.
\begin{lem}\label{lem:change-of-pole}
    Let $\Omega \subset \R^{n}$ be a bounded Lipschitz domain, and let $x_{0}\in \Omega$. Then, there are constants $r_{0}>0$ and $C_{0}>1$ depending only on the Lipschitz character of $\Omega$ such that, for all $\xi,\xi'\in \partial \Omega$ and $0<C_{0}r'<r<r_{0}$ with $|x_{0}-\xi|>2r$, and $|\xi'-\xi|<r-r'$, the following holds:
    \begin{equation}\label{eq:change-of-pole-omega}
        \omega^{A_{r}(\xi)}(B_{r'}(\xi'))\approx \frac{\omega^{x_{0}}(B_{r'}(\xi'))}{\omega^{x_{0}}(B_{r}(\xi))},
    \end{equation}
    where the comparability constants depend only on $n,s$, and the Lipschitz character of $\Omega$.
\end{lem}
\begin{proof}
    Let $r_{0}$ be given by Lemma \ref{lem:comparison-omega-G}, and let $c\in (0,1)$ be the constant appearing in the corkscrew condition \eqref{eq:corkscrew-condition}. We set $C_{0}=4c^{-1}$ and notice that $|A_{r}(\xi)-\xi'|\ge cr=4C_{0}^{-1}r>4r'$. Therefore, we may apply repeatedly Lemma \ref{lem:comparison-omega-G} to deduce the following comparability estimates:
    \begin{equation*}
    \begin{gathered}
        \omega^{A_{r}(\xi)}(B_{r'}(\xi'))\approx (r')^{n-2s}G^{A_{r}(\xi)}(A_{r'}(\xi')),\qquad \omega^{x_{0}}(B_{r'}(\xi'))\approx (r')^{n-2s}G^{x_{0}}(A_{r'}(\xi')),\\
        \omega^{x_{0}}(B_{r}(\xi))\approx {r}^{n-2s}G^{x_{0}}(A_{r}(\xi)).
    \end{gathered}
    \end{equation*}
    Then, applying the boundary Harnack principle (Theorem \ref{thm:BHP}) to the functions $G^{A_{r}(\xi)}$ and $G^{x_{0}}$ away from their poles, we deduce that 
    \begin{equation*}
        \frac{\omega^{A_{r}(\xi)}(B_{r'}(\xi'))}{\omega^{x_{0}}(B_{r'}(\xi'))}\approx \frac{G^{A_{r}(\xi)}(A_{r'}(\xi'))}{G^{x_{0}}(A_{r'}(\xi'))}\approx \frac{G^{A_{r}(\xi)}(A_{r/2}(\xi))}{G^{x_{0}}(A_{r/2}(\xi))}\approx \frac{r^{2s-n}}{G^{x_{0}}(A_{r}(\xi))}\approx \frac{1}{\omega^{x_{0}}(B_{r}(\xi))},
    \end{equation*}
    as desired, where in the penultimate step we used \eqref{eq:basic-estimate-green-function-2} and the Harnack inequality.     
\end{proof}

\begin{rmk}\label{rmk:change-of-pole-green-functions}
    In the framework of Lemma \ref{lem:change-of-pole}, we deduce the following change of pole formula for the Green function:
    \begin{equation}\label{eq:change-of-pole-green-function}
        \frac{G^{x_{0}}(A_{r'}(\xi'))}{G^{A_{r}(\xi)}(A_{r'}(\xi'))}\approx r^{n-2s}G^{x_{0}}(A_{r}(\xi)).
    \end{equation}
    By the boundary Harnack principle, we may send $r'$ to zero and conclude that the H\"{o}lder continuous function $G^{x_{0}}/G^{A_{r}(\xi)}$ is comparable to the same quantity at the boundary: 
    \begin{equation}\label{eq:change-of-pole-green-function-boundary}
        \lim_{\Omega\ni z\to \xi'}\frac{G^{x_{0}}(z)}{G^{A_{r}(\xi)}(z)}\approx r^{n-2s}G^{x_{0}}(A_{r}(\xi))\qquad \forall \xi'\in \partial \Omega\cap B_r(\xi).
    \end{equation}
\end{rmk}

We will also need to compare the value of the Green function at an interior corkscrew point, with the value of the Poisson kernel at the corresponding exterior one. 
\begin{lem}\label{lem:comparison-P-G}
     Let $\Omega \subset \R^{n}$ be a bounded Lipschitz domain, and let $x_{0}\in \Omega$. Then, there exists $r_{0}>0$ depending only on the Lipschitz character of $\Omega$ such that, for all $\xi\in \partial \Omega$ and $r\in (0,r_{0})$ with $|x_{0}-\xi|>2r$, the following holds:
    \begin{equation}\label{eq:comparison-P-G}
        P^{x_{0}}(A_{r}'(\xi))\approx (1-s)\frac{G^{x_{0}}(A_{r}(\xi))}{r^{2s}},
    \end{equation}
    where the comparability constants depend only on $n,s$, and the Lipschitz character of $\Omega$, and are uniform as $s\to 1^-$. \end{lem}

\begin{proof}
    Let $r_{0}$ be sufficiently small with respect to the localization radius of the  domain. We start from formula \eqref{eq:poisson-kernel-integral-formula} computed at the point $A_{r}'(\xi)$:
    \begin{equation*}
        P^{x_{0}}(A_{r}'(\xi))= c_{n,s}\int_{\Omega}\frac{G^{x_{0}}(z)}{|z-A_{r}'(\xi)|^{n+2s}}\,dz.
    \end{equation*}
    To prove the first inequality, note that for all $z\in B_{cr/2}(A_{r}(\xi))$ we have $|z-A_{r}'(\xi)|\approx r$, thus by Harnack, $G^{x_{0}}(z)\approx G^{x_{0}}(A_{r}(\xi))$. Therefore,
    \begin{equation*}
        P^{x_{0}}(A_{r}'(\xi))\ge c_{n,s}\int_{B_{cr/2}(A_{r}(\xi))}\frac{G^{x_{0}}(z)}{|z-A_{r}'(\xi)|^{n+2s}}\gtrsim (1-s)\frac{G^{x_{0}}(A_{r}(\xi))}{r^{2s}}.
    \end{equation*}
    
   Let us now prove the opposite inequality. We fix $r_{1}\coloneqq \min\left\{r_{0},|x_{0}-\xi|/2\right\}>r$. Then, by Theorem \ref{thm:holder-boundary-lip-domain}, we have 
    \begin{equation*}
        G^{x_{0}}(z)\lesssim G^{x_{0}}(A_{t}(\xi))\lesssim G^{x_{0}}(A_{r}(\xi))\left(\frac{t}{r}\right)^{2s-\alpha}\qquad \forall z\in B_{t}(\xi),\quad \forall t\in [r,r_{1}].
    \end{equation*}
    As a consequence,
    \begin{align*}
        \int_{B_{r_{1}}(\xi)}\frac{G^{x_{0}}(z)}{|z-A_{r}'(\xi)|^{n+2s}}\,dz &\le \int_{B_{r}(\xi)}\frac{G^{x_{0}}(z)}{|z-A_{r}'(\xi)|^{n+2s}}\,dz+\int_{r}^{r_{1}}\int_{\partial B_{t}(\xi)}\frac{G^{x_{0}}(z)}{|z-A_{r}'(\xi)|^{n+2s}}\,d\mathcal{H}^{n-1}(z)\,dt\\
        &\lesssim G^{x_{0}}(A_{r}(\xi))\left(|B_{r}(\xi)| r^{-n-2s}+ r^{\alpha-2s}\int_{r}^{r_{1}}t^{-\alpha-1}\,dt\right)\lesssim G^{x_{0}}(A_{r}(\xi)) r^{-2s}.
    \end{align*}
    To conclude, we need to show that the integral over the set $\Omega\setminus B_{r_{1}}(\xi)$ enjoys a similar bound. First of all, we observe that due to Theorem \ref{thm:holder-boundary-lip-domain}, it holds
    \begin{equation*}
        \frac{G^{x_{0}}(A_{r}(\xi))}{r^{2s}}\gtrsim \frac{G^{x_{0}}(A_{r_{1}}(\xi))}{r_{1}^{2s}}. 
    \end{equation*}
    Therefore, it suffices to prove that 
    \begin{equation*}
        \int_{\Omega \setminus B_{r_{1}}(\xi)}\frac{G^{x_{0}}(z)}{|z-A_{r}'(\xi)|^{n+2s}}\,dz\lesssim \frac{G^{x_{0}}(A_{r_{1}}(\xi))}{r_{1}^{2s}}.
    \end{equation*}
    We then distinguish two cases. If $\delta(x_{0})\approx r_{1}$, then by Harnack and \eqref{eq:basic-estimate-green-function-2}, we have $G^{x_{0}}(A_{r_{1}}(\xi))\gtrsim r_{1}^{2s-n}$. As a consequence, using \eqref{eq:basic-estimate-green-function-1} we derive
    \begin{align*}
        \int_{\Omega \setminus B_{r_{1}}(\xi)}\frac{G^{x_{0}}(z)}{|z-A_{r}'(\xi)|^{n+2s}}\,dz&\le \int_{\left(\Omega \setminus B_{r_{1}}(\xi)\right)\cap B_{r_{1}}(x_{0})}\frac{G^{x_{0}}(z)}{|z-A_{r}'(\xi)|^{n+2s}}\,dz\\
        &\quad\,+\int_{\left(\Omega \setminus B_{r_{1}}(\xi)\right)\setminus B_{r_{1}}(x_{0})}\frac{G^{x_{0}}(z)}{|z-A_{r}'(\xi)|^{n+2s}}\,dz\\
        &\lesssim \frac{1}{r_{1}^{n+2s}}\int_{0}^{r_{1}}\rho^{2s-1}\,d\rho+ \int_{r_{1}}^{\infty}\rho^{-n-1}\,d\rho\lesssim \frac{1}{r_{1}^{n}}\lesssim \frac{G^{x_{0}}(A_{r_{1}}(\xi))}{r_{1}^{2s}},
    \end{align*}
    as desired. Let us now assume instead that $\delta(x_{0})\ll r_{1}$. For every $t\in [C\delta(x_{0}),r_{1}]$ we denote by $y_{t}\in \Omega$ a point such that $\delta(y_{t})\approx |y_{t}-x_{0}|= t$. It is convenient to split the integral over the three domains
    \begin{equation*}
        \left(\Omega \setminus B_{r_{1}}(\xi)\right)\cap B_{C\delta(x_{0})}(x_{0}),\quad \left(\Omega \setminus B_{r_{1}}(\xi)\right)\cap \left(B_{r_{1}}(x_{0})\setminus B_{C\delta(x_{0})}(x_{0})\right),\quad \left(\Omega \setminus B_{r_{1}}(\xi)\right)\setminus B_{r_{1}}(x_{0}).
    \end{equation*}
    For the first, by \eqref{eq:basic-estimate-green-function-1} we obtain
    \begin{align*}
        \int_{\left(\Omega \setminus B_{r_{1}}(\xi)\right)\cap B_{C\delta(x_{0})}(x_{0})}\frac{G^{x_{0}}(z)}{|z-A_{r}'(\xi)|^{n+2s}}\,dz\lesssim \frac{1}{r_{1}^{n+2s}}\int_{0}^{C\delta(x_{0})}\rho^{2s-1}\,d\rho\lesssim \frac{\delta(x_{0})^{2s}}{r_{1}^{n+2s}}\lesssim \frac{G^{x_{0}}(A_{r_{1}}(\xi))}{r_{1}^{2s}},
    \end{align*}
    where in the last step we used Theorem \ref{thm:holder-boundary-lip-domain} and \eqref{eq:basic-estimate-green-function-2} to get $G^{x_{0}}(A_{r_{1}}(\xi))/\delta(x_{0})^{2s}\gtrsim G^{y_{r_{1}}}(A_{r_{1}}(\xi))/r_{1}^{2s}\approx r_{1}^{-n}$.
    For the second, notice that an application of the change of pole formula \eqref{eq:change-of-pole-green-function} along with Theorem \ref{thm:holder-boundary-lip-domain} and \eqref{eq:basic-estimate-green-function-2} gives, for all $t\in [C\delta(x_{0}), r_{1}]$ and all $z\in \Omega\cap \partial B_{t}(x_{0})$,
    \begin{equation*}
        G^{x_{0}}(z)\lesssim G^{x_{0}}(y_{t})\approx \frac{G^{x_{0}}(A_{r_{1}}(\xi))}{G^{y_{t}}(A_{r_{1}}(\xi))t^{n-2s}}\lesssim \frac{G^{x_{0}}(A_{r_{1}}(\xi))r_{1}^{2s-\alpha}}{G^{y_{r_{1}}}(A_{r_{1}}(\xi))t^{n-\alpha}}\lesssim \frac{G^{x_{0}}(A_{r_{1}}(\xi))r_{1}^{n-\alpha}}{t^{n-\alpha}}.
    \end{equation*}
    Therefore, 
    \begin{equation*}
        \int_{\left(\Omega \setminus B_{r_{1}}(\xi)\right)\cap \left(B_{r_{1}}(x_{0})\setminus B_{C\delta(x_{0})}(x_{0})\right)}\frac{G^{x_{0}}(z)}{|z-A_{r}'(\xi)|^{n+2s}}\,dz\lesssim \frac{G^{x_{0}}(A_{r_{1}}(\xi))}{r_{1}^{2s+\alpha}}\int_{C\delta(x_{0})}^{r_{1}}t^{\alpha-1}\,dt\lesssim \frac{G^{x_{0}}(A_{r_{1}}(\xi))}{r_{1}^{2s}}.
    \end{equation*}
    Finally, for the integral over the third domain, we notice that for all $z\in \left(\Omega \setminus B_{r_{1}}(\xi)\right)\setminus B_{r_{1}}(x_{0})$, by Theorem \ref{thm:holder-boundary-lip-domain} applied to the functions $G^{\cdot}(z)$, $G^{\cdot}(A_{r_{1}}(\xi))$, and the basic estimates \eqref{eq:basic-estimate-green-function-1}, \eqref{eq:basic-estimate-green-function-2} we have
    \begin{equation*}
        G^{x_{0}}(z)\approx G^{x_{0}}(A_{r_{1}}(\xi))\frac{G^{y_{r_{1}/2}}(z)}{G^{y_{r_{1}/2}}(A_{r_{1}}(\xi))}\lesssim G^{x_{0}}(A_{r_{1}}(\xi)) \frac{|z-x_{0}|^{2s-n}}{r_{1}^{2s-n}}.
    \end{equation*}
    As a consequence, 
    \begin{equation*}
        \int_{\left(\Omega \setminus B_{r_{1}}(\xi)\right)\setminus B_{r_{1}}(x_{0})}\frac{G^{x_{0}}(z)}{|z-A_{r}'(\xi)|^{n+2s}}\,dz\lesssim \frac{G^{x_{0}}(A_{r_{1}}(\xi))}{r_{1}^{2s-n}}\int_{r_{1}}^{\infty}\rho^{-n-1}\,d\rho \lesssim \frac{G^{x_{0}}(A_{r_{1}}(\xi))}{r_{1}^{2s}}.
    \end{equation*}
    This concludes the proof. 
\end{proof}

\begin{rmk}
    The basic estimates contained in this section will be used extensively in the rest of the paper. Sometimes, in the sequel, we will invoke some of these results in combination with the Harnack inequality even when the geometric setup (positioning of the pole, relation between radii, etc.) does not match the assumptions in the corresponding statements exactly. This, however, will be done only when dealing with nonnegative $s$-harmonic functions, provided one can reduce to the precise required framework after a Harnack-chain argument.
\end{rmk}

\section{Dahlberg's Theorem for the fractional Laplacian}\label{sec:Dahlberg}
In this section, we establish Theorem \ref{thm:Dahlberg-intro} in the more general form stated in Theorem \ref{thm:Dahlberg-general}. We follow the overall strategy outlined in Section \ref{subsec-ideas-proof}, supplemented by the estimates from Section \ref{subsec:harmonic-measure}.  

\subsection[L2 estimates for fractional normal derivatives]{\texorpdfstring{$L^{2}$-estimates}{L2 estimates} for fractional normal derivatives}
A key tool for the derivation of $L^{2}$-estimates for fractional normal derivatives is the Pohozaev-type identity contained in the following  lemma. Equation \eqref{eq:Pohozaev} is in the spirit of the fractional Pohozaev identities of \cite{rosoton2014pohozaev}, but it is adapted here to the Green function setting (see also \cite{dieb2025note}). Recall the definition of fractional normal derivatives from \eqref{eq:def-fractional-normal-der}, and observe that if $\Omega$ is smooth, then $\partial_{\nu}^{s} G^{x_{0}}$ is well-defined and continuous on $\partial \Omega$; see \cite{rosoton2014dirichlet}. 

\begin{lem}\label{lem:Pohozaev}
    Let $\Omega\subset \R^{n}$ be a smooth bounded domain, and let $\nu$ denote the outer unit normal to $\partial\Omega$. Then, for every $x_{0}\in \Omega$, the following formula holds:
    \begin{equation}\label{eq:Pohozaev}
        \int_{\partial \Omega}(\partial_{\nu}^{s}G^{x_{0}}(\xi))^{2}(\xi-x_{0})\cdot \nu \,d\mathcal{H}^{n-1}(\xi)=\frac{\kappa_{n,s}(n-2s)}{\Gamma(1+s)^{2}}\int_{\Omega^{c}}\frac{P^{x_{0}}(y)}{|y-x_{0}|^{n-2s}}\,dy,
    \end{equation}
    where $\kappa_{n,s}>0$ is the constant that appears in \eqref{eq:def-fond-sol-frac-lap}.
\end{lem}

\begin{proof} Let $R^{x_{0}}$ be the function introduced in \eqref{eq:def-green-function}. 
Since $R^{x_{0}}$ is smooth in $\R^{n}\setminus \partial \Omega$, we can define, almost everywhere, 
\begin{equation*}
    v^{x_{0}}(y)\coloneqq \nabla R^{x_{0}}(y)\cdot (y-x_{0}).
\end{equation*}
Moreover, we have the identity
\begin{equation*}
    (-\Delta)^{s}v^{x_{0}}(y)= 2s(-\Delta)^{s}R^{x_{0}}(y)+(y-x_{0})\cdot \nabla (-\Delta)^{s}R^{x_{0}}(y)=0\qquad \forall y\in \Omega.
\end{equation*}
Then, a straightforward computation shows that $v^{x_{0}}$ solves the following problem:
\begin{equation*}
    \left\{
    \begin{array}{rclll}
         (-\Delta)^{s}v^{x_{0}}&=&0\quad &\text{in $\Omega$},\\
            v^{x_{0}}(y)&=&\kappa_{n,s}(2s-n)\frac{1}{|y-x_{0}|^{n-2s}}\quad &\forall y\in \interior \Omega^{c},\\
            \partial_{\nu}^{s-1}v^{x_{0}}(\xi)&=&s\partial_{\nu}^{s}G^{x_{0}}(\xi)(\xi-x_{0})\cdot \nu\quad &\forall \xi\in \partial \Omega.
    \end{array}
    \right.
\end{equation*}
Therefore, by \cite[Theorem 1.2.3]{abatangelo2013large} and \cite[Appendix B]{chan2021blow}, $v^{x_{0}}$ admits the following representation:
\begin{equation*}
    v^{x_{0}}(z)= \int_{\interior \Omega^{c}}\frac{\kappa_{n,s}(2s-n)}{|y-x_{0}|^{n-2s}}P^{z}(y)\, dy+\Gamma(s)\Gamma(1+s)\int_{\partial\Omega}s\partial_{\nu}^{s}G^{x_{0}}(\xi)(\xi-x_{0})\cdot \nu \partial_{\nu}^{s}G^{z}(\xi)\, d\mathcal{H}^{n-1}(\xi)\qquad \forall z\in \Omega.
\end{equation*}
The desired identity is obtained by evaluating the latter at $z=x_{0}$, noting that $v^{x_{0}}(x_{0})=0$, and using $s\Gamma(s)=\Gamma(1+s)$.
\end{proof}

The Pohozaev identity above yields the following scale-invariant quadratic estimate for the fractional normal derivative of the Green function. In the proposition below, we still work under the assumption that the boundary is smooth, in order to make sense of the fractional normal derivative at the boundary. However, the constants in \eqref{eq:L2-a-priori-boundary-estimate} below depend on the geometry of $\Omega$ only through its Lipschitz character. 

\begin{prop}\label{prop:L2-normal-a-priori}
    Let $\Omega\subset \R^{n}$ be a smooth bounded domain, and let $x_{0}\in \Omega$.  Then, there exists $r_{0}>0$ depending only on the Lipschitz character of $\Omega$ such that, for all $\xi\in \partial \Omega$ and $r\in (0,r_{0})$ with $|x_{0}-\xi|>2r$, the following holds:
    \begin{equation}\label{eq:L2-a-priori-boundary-estimate}
        \left(\fint_{B_{r}(\xi)\cap \partial\Omega}(\partial_{\nu}^{s}G^{x_{0}})^{2}d\mathcal{H}^{n-1}\right)^{1/2}\approx \frac{G^{x_{0}}(A_{r}(\xi))}{r^{s}},
    \end{equation}
    where the comparability constants depend only on $n,s$, and the Lipschitz character of $\Omega$.
\end{prop}
\begin{proof}
   As a consequence of \eqref{eq:change-of-pole-green-function-boundary}, we have $\partial_{\nu}^{s}G^{x_{0}}\approx \partial_{\nu}^{s}G^{A_{r}(\xi)} r^{n-2s} G^{x_{0}}(A_{r}(\xi))$. Moreover, since $\Omega$ is Lipschitz, $\mathcal{H}^{n-1}(\partial \Omega \cap B_{r}(\xi))\approx r^{n-1}$. Therefore, \eqref{eq:L2-a-priori-boundary-estimate} can be rewritten equivalently as
    \begin{equation}\label{eq:L2-a-priori-boundary-estimate-simplified}
        \int_{\partial\Omega \cap B_{r}(\xi)}(\partial_{\nu}^{s}G^{A_{r}(\xi)})^{2}\,d\mathcal{H}^{n-1}\approx r^{2s-1-n}.
    \end{equation}
    A localization argument based on the Harnack inequality shows that it is sufficient to prove \eqref{eq:L2-a-priori-boundary-estimate-simplified} with domain of integration in the left-hand side replaced by $\partial \Omega\cap B_{c_{0}r}(\xi)$, where $c_{0}$ is a small number depending only on the parameters of the problem. Up to rigid motion, we may assume that $A_{r}(\xi)=0$, $\xi=-re_{n}$. Let $\Gamma\subset \R^{n}$ be the cone 
    \begin{equation*}
        \Gamma\coloneqq\left\{x\in \R^{n}:|x|<-x_{n}\sqrt{1+\alpha^{2}}\right\}.
    \end{equation*}
    Here $\alpha>0$ is chosen sufficiently small depending only on the Lipschitz constant of the domain, in such a way that $x\cdot \nu \gtrsim r$ for all $x\in \partial \Omega \cap \Gamma$. Then, calling $c_{0}\coloneqq\tfrac{\alpha}{4\sqrt{1+\alpha^{2}}}$, we have $B_{4c_{0}r}(\xi)\subset \Gamma$. In the rest of the proof, we use a comparison argument based on Lemma \ref{lem:Pohozaev} to prove the desired comparability.

    Let $\widetilde\Omega\subset \Omega$ be a smooth star-shaped domain such that $\partial\widetilde\Omega\cap B_{2c_{0}r}(\xi)=\partial \Omega \cap B_{2c_{0}r}(\xi)$, such that $\diam \widetilde{\Omega}\lesssim r$, $\widetilde \delta(0)\gtrsim r$, and $x\cdot \widetilde \nu \approx r$ for all $x\in \partial \widetilde \Omega$. Here $\widetilde \delta\coloneqq \delta_{\widetilde \Omega}$, and $\widetilde \nu$ is the outer unit normal to $\partial \widetilde \Omega$; see Figure \ref{fig:auxiliary-domain}. Calling $\widetilde{G}$ the Green function for the domain $\widetilde \Omega$, by the boundary Harnack principle (Theorem \ref{thm:BHP}) and the basic estimates for Green functions \eqref{eq:basic-estimate-green-function-1} and \eqref{eq:basic-estimate-green-function-2}, we have
    \begin{equation*}
        \partial_{\nu}^{s}G^{0}(\xi')\approx \partial_{\widetilde\nu}^{s}\widetilde{G}^{0}(\xi')\qquad \forall \xi'\in B_{c_{0}r}(\xi)\cap\partial \Omega = B_{c_{0}r}(\xi)\cap\partial \widetilde\Omega.  
    \end{equation*}
    Therefore, it will suffice to bound the square integral of $\partial_{\widetilde\nu}^{s}\widetilde{G}^{0}$ over $B_{c_{0}r}(\xi)\cap\partial \widetilde\Omega$.

     Note that $|y|\gtrsim r$ for all $y\in \widetilde{\Omega}^{c}$ and $\xi'\cdot \widetilde\nu \gtrsim r$ for all $\xi'\in \partial \widetilde\Omega$, thanks to the above hypotheses on the domain $\widetilde \Omega$. Therefore, by \eqref{eq:Pohozaev}, we may bound
    \begin{equation}\label{eq:bound-L2-desG-starshaped}
        \int_{\partial \widetilde\Omega}(\partial_{\widetilde\nu}^{s}\widetilde{G}^{0})^{2}\, d\mathcal{H}^{n-1}\lesssim \frac{1}{r}\int_{\partial \widetilde\Omega}(\partial_{\widetilde\nu}^{s}\widetilde G^{0}(\xi'))^{2}\xi'\cdot \widetilde{\nu}\,d\mathcal{H}^{n-1}(\xi') \approx \frac{1}{r}\int_{\widetilde\Omega^{c}}\frac{\widetilde P^{0}(y)}{|y|^{n-2s}}\,dy\lesssim r^{2s-1-n}\int_{\widetilde \Omega^{c}}\widetilde P^{0}=r^{2s-1-n}.
    \end{equation}
    Hence, it only remains to prove the bound from below.  
    Let $\eps\in (0,c_{0}/2)$ be a small number to be chosen later. For every $y\in \widetilde \Omega\setminus B_{c_{0}r}(\xi)$, the function $x\mapsto \widetilde G^{x}(y)\ge 0$ is $s$-harmonic in $\widetilde \Omega \cap B_{c_{0}r}(\xi)$ and vanishes continuously in $\widetilde\Omega^{c}\cap B_{c_{0}r}(\xi)$. Therefore, having $|\xi-(1-\eps)\xi|\le |\xi-(1-c_{0}/2)\xi|\le rc_{0}/2$, we can apply the H\"{o}lder estimate up to the boundary from Theorem \ref{thm:holder-boundary-lip-domain}, and obtain
    \begin{equation*}
        \widetilde G^{(1-\eps)\xi}(y)\lesssim \eps^{\alpha}\widetilde G^{(1-c_{0}/2)\xi}(y)\qquad \forall y\in \widetilde\Omega\setminus B_{c_{0}r}(\xi).
    \end{equation*}
    For $y\in \widetilde\Omega\setminus B_{c_{0}r}(\xi)$ sufficiently close to the boundary, the Harnack inequality guarantees $\widetilde G^{(1-c_{0}/2)\xi}(y)\lesssim \widetilde G^{0}(y)$. Therefore, taking fractional normal derivatives, we also find  
    \begin{equation}\label{eq:closepole-farpoint-desG-starshaped-pointwise}
        \partial_{\widetilde\nu}^{s}\widetilde G^{(1-\eps)\xi}(\xi')\lesssim  \eps^{\alpha}\partial_{\widetilde\nu}^{s}\widetilde G^{0}(\xi')\qquad \forall \xi' \in \partial \widetilde\Omega \setminus B_{c_{0}r}(\xi).
    \end{equation}
    Putting \eqref{eq:bound-L2-desG-starshaped} and \eqref{eq:closepole-farpoint-desG-starshaped-pointwise} together, we get
    \begin{equation}\label{eq:closepole-farpoint-desG-starshaped-integral}
        \int_{\partial \widetilde\Omega \setminus B_{c_{0}r}(\xi)}(\partial_{\widetilde\nu}^{s}\widetilde G^{(1-\eps)\xi})^{2}\, d\mathcal{H}^{n-1}\lesssim \eps^{2\alpha}\int_{\partial \widetilde \Omega}(\partial_{\widetilde\nu}^{s}\widetilde G^{0})^{2}\, d\mathcal{H}^{n-1}\lesssim \eps^{2\alpha}r^{2s-1-n}.
    \end{equation}
   At this point, a further application of \eqref{eq:Pohozaev} gives
   \begin{align*}
       \int_{\partial \widetilde\Omega\cap B_{c_{0}r}(\xi)}(\partial_{\widetilde\nu}^{s}\widetilde G^{(1-\eps)\xi})^{2} \, d\mathcal{H}^{n-1}&\gtrsim \frac{1}{r}\int_{\partial \widetilde\Omega\cap B_{c_{0}r}(\xi)}(\partial_{\widetilde\nu}^{s}\widetilde G^{(1-\eps)\xi}(\xi'))^{2}(\xi'- (1-\eps)\xi)\cdot \widetilde\nu \, d\mathcal{H}^{n-1}(\xi')\\
       &= \frac{\kappa_{n,s}(n-2s)}{\Gamma(1+s)^{2}}\frac{1}{r}\int_{\widetilde \Omega^{c}}\frac{\widetilde P^{(1-\eps)\xi}(y)}{|y-(1-\eps)\xi|^{n-2s}}\, dy\\
       &\quad\,- \frac{1}{r}\int_{\partial \widetilde\Omega \setminus B_{c_{0}r}(\xi)}(\partial_{\widetilde\nu}^{s}\widetilde G^{(1-\eps)\xi}(\xi'))^{2}(\xi'- (1-\eps)\xi)\cdot \widetilde\nu \, d\mathcal{H}^{n-1}(\xi')\\
       &\ge C_{1}\eps^{2s-n}r^{2s-1-n}-C_{2}\eps^{2\alpha} r^{2s-1-n}, 
   \end{align*}
   where in the last step we used $\widetilde\omega^{(1-\eps)\xi}(B_{\eps r}(\xi))\approx 1$ and \eqref{eq:closepole-farpoint-desG-starshaped-integral}. 
   Choosing $\eps\coloneqq\left(C_{1}/2C_{2}\right)^{1/(n+2\alpha-2s)}\wedge (c_{0}/2)$, and arguing by Harnack chains, we finally get
   \begin{equation*}
      \int_{\partial \widetilde\Omega\cap B_{c_{0}r}(\xi)}(\partial_{\widetilde\nu}^{s}\widetilde G^{0})^{2}\, d\mathcal{H}^{n-1}\gtrsim \int_{\partial \widetilde\Omega\cap B_{c_{0}r}(\xi)}(\partial_{\widetilde\nu}^{s}\widetilde G^{(1-\eps)\xi})^{2}\, d\mathcal{H}^{n-1}\gtrsim r^{2s-1-n},
   \end{equation*}
   which is the desired bound from below. This concludes the proof.
\end{proof}

\begin{figure}[H]
\centering
\includegraphics[width=0.6\textwidth]{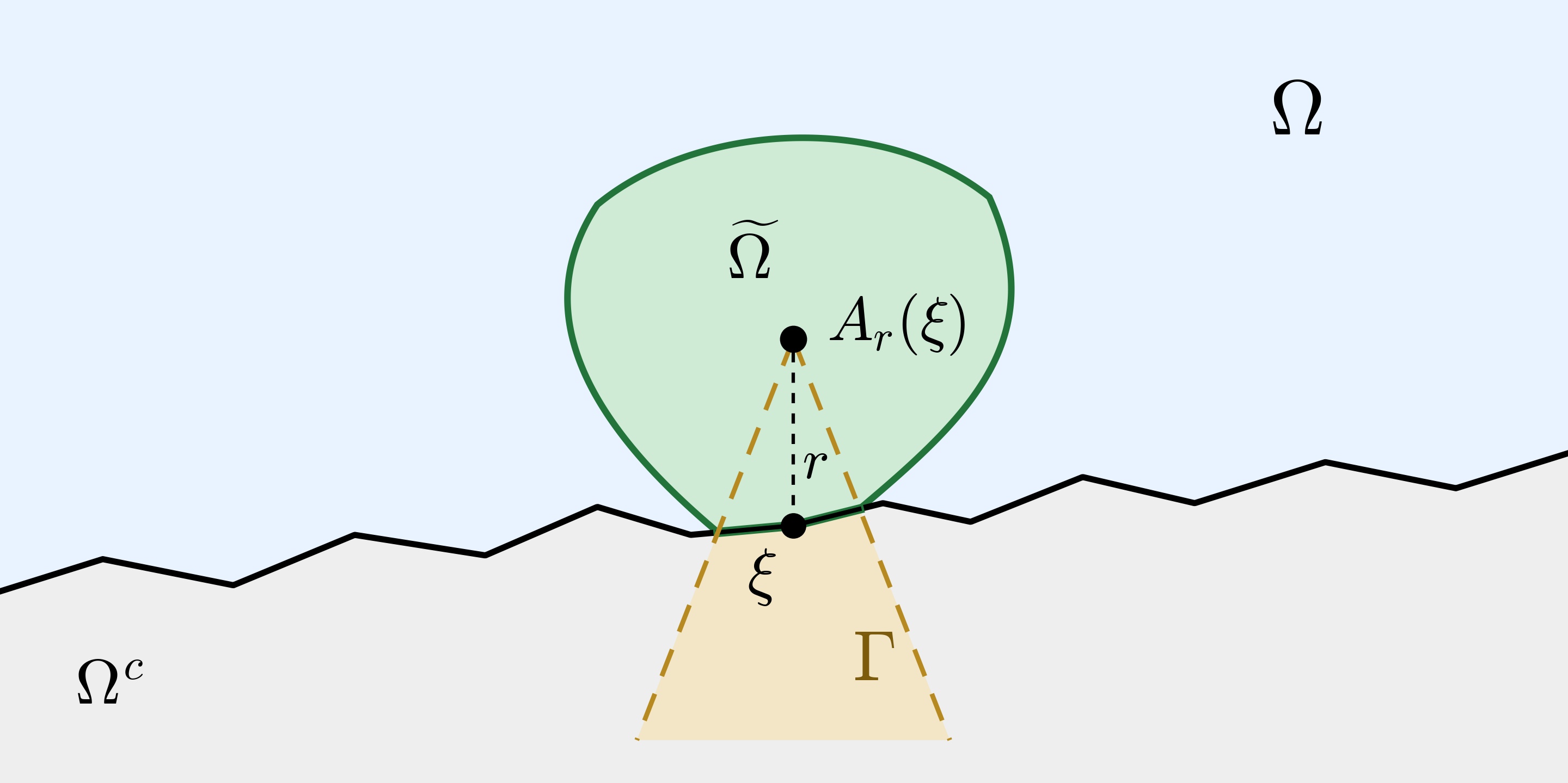}
\caption{The auxiliary star-shaped domain $\widetilde \Omega$ in the proof of Proposition \ref{prop:L2-normal-a-priori}.}
\label{fig:auxiliary-domain}
\end{figure}

\subsection[Reverse-Holder inequality for the s-harmonic measure]{Reverse-H\"{o}lder inequality for the $s$-harmonic measure}

We now transfer the boundary square estimate derived for smooth domains in Proposition \ref{prop:L2-normal-a-priori} to parallel level sets
inside general Lipschitz domains.
\begin{thm}\label{thm:uniform-L2-Green-function}
    Let $\Omega\subset \R^{n}$ be a bounded Lipschitz domain, and let $x_{0}\in \Omega$. Then, there exist $r_{0}>0$ and $c_{0}\in (0,1)$ depending only on the Lipschitz character of $\Omega$ such that, for all $\xi\in \partial \Omega$ and $r\in (0,r_{0})$ with $|x_{0}-\xi|>2r$, the following holds:
    \begin{equation}\label{eq:Linfty-L2-Green-function}
        \frac{1}{t^{s}}\left(\fint_{B_{r}(\xi)\cap \Omega\cap \{\delta=t\}}(G^{x_{0}})^{2}d\mathcal{H}^{n-1}\right)^{1/2}\approx \frac{G^{x_{0}}(A_{r}(\xi))}{r^{s}}\qquad \forall t\in (0,c_{0}r),
    \end{equation}
    where the comparability constants depend only on $n,s$, and the Lipschitz character of $\Omega$. 
\end{thm}
\begin{proof}
    Thanks to the approximation result from Lemma \ref{lem:approximation-smooth-domains}, it suffices to consider the case in which $\Omega$ is a smooth domain with prescribed Lipschitz character.
     Provided that $r_{0}$ and $c_{0}$ are chosen sufficiently small depending only on the Lipschitz character of $\Omega$, Proposition \ref{prop:L2-normal-a-priori} gives
   \begin{equation*}
       \fint_{B_{t}(\xi')\cap \partial\Omega}(\partial_{\nu}^{s}G^{x_{0}})^{2}d\mathcal{H}^{n-1}\approx \frac{\left(G^{x_{0}}(A_{t}(\xi'))\right)^{2}}{t^{2s}}\qquad \forall \xi' \in B_{r/2}(\xi)\cap \partial \Omega,\quad  \forall t\in (0,c_{0}r).
   \end{equation*}
   Then, by the properties of corkscrew points \eqref{eq:corkscrew-condition}, the Harnack inequality, Fubini's theorem, and again Proposition \ref{prop:L2-normal-a-priori},
   \begin{align*}
       \frac{1}{t^{2s}}\int_{B_{r/2}(\xi)\cap \Omega\cap \{\delta=t\}}(G^{x_{0}})^{2}d\mathcal{H}^{n-1}&\approx\int_{B_{r/2}(\xi)\cap \partial\Omega}\fint_{B_{t}(\xi')\cap \partial\Omega}(\partial_{\nu}^{s}G^{x_{0}})^{2}\,d\mathcal{H}^{n-1}\,d\mathcal{H}^{n-1}(\xi')\\
       &\approx \int_{B_{r/2}(\xi)\cap \partial \Omega}(\partial_{\nu}^{s}G^{x_{0}})^{2}\,d\mathcal{H}^{n-1}\approx r^{n-1-2s}\left(G^{x_{0}}(A_{r/2}(\xi))\right)^{2},
   \end{align*}
   where in the second step we used the fact that squared integrals of $\partial_{\nu}^{s}G^{x_{0}}$ over small boundary balls with the same center and comparable radii are equivalent (by Proposition \ref{prop:L2-normal-a-priori} and the Harnack inequality), while in the third step we used $\mathcal{H}^{n-1}(B_{r/2}(\xi)\cap \partial \Omega)\approx r^{n-1}$, which is uniform for $\Omega$ in a prescribed Lipschitz class. Rearranging terms, we deduce \eqref{eq:Linfty-L2-Green-function} with radius $r/2$, and the general result follows from a localization argument. 
\end{proof}

Combining the Green function estimate from Theorem \ref{thm:uniform-L2-Green-function} with the comparison result from Lemma \ref{lem:comparison-P-G} gives the corresponding square estimate outside the domain.
\begin{cor}\label{cor:uniform-L2-Poisson-kernel}
    Let $\Omega\subset \R^{n}$ be a bounded Lipschitz domain, and let $x_{0}\in \Omega$. Then, there exist $r_{0}>0$ and $c_{0}\in (0,1)$ depending only on the Lipschitz character of $\Omega$ such that, for all $\xi\in \partial \Omega$ and $r\in (0,r_{0})$ with $|x_{0}-\xi|>2r$, the following holds:
    \begin{equation}\label{eq:Linfty-L2-Poisson-kernel}
       \frac{t^{s}}{1-s} \left(\fint_{B_{r}(\xi)\cap \Omega^{c}\cap \{\delta=t\}}(P^{x_{0}})^{2}d\mathcal{H}^{n-1}\right)^{1/2}\approx \frac{\omega^{x_{0}}(B_{r}(\xi))}{r^{n-s}}\qquad \forall t\in (0,c_{0}r),
    \end{equation}
    where the comparability constants depend only on $n,s$, and the Lipschitz character of $\Omega$. 
\end{cor}

\begin{proof}
    The equivalence in \eqref{eq:Linfty-L2-Poisson-kernel} follows immediately from \eqref{eq:Linfty-L2-Green-function} after using Lemma \ref{lem:comparison-omega-G} to compare the harmonic measure of a ball with the Green function at the interior corkscrew point, and Lemma \ref{lem:comparison-P-G} to compare the Poisson kernel at an exterior corkscrew point with the Green function at the corresponding interior one. 
\end{proof}

We have now all the ingredients to prove Dahlberg's theorem for the fractional Laplacian (Theorem \ref{thm:Dahlberg-intro}). In fact, we prove the following more general version of the reverse-H\"{o}lder estimate \eqref{eq:reverse-holder-intro}, which is valid for exponents slightly above $2$ and for the whole family of weights
\begin{equation}\label{eq:def-sigma-beta}
    \sigma_{\beta}\coloneqq (1-\beta)\frac{\delta^{-\beta}}{1+\delta^{n+2s-\beta}}\mathscr{L}^{n}\res \Omega^{c},\qquad \beta\in (2s-1,s].
\end{equation}

\begin{thm}\label{thm:Dahlberg-general}
    Let $\Omega\subset \R^{n}$ be a bounded Lipschitz domain, and let $x_{0}\in \Omega$. Given $\beta \in (2s-1,s]$, let $\sigma_{\beta}$ be the measure in \eqref{eq:def-sigma-beta}. There is $p_{0}>2$ depending only on $n,s, \beta$, and the Lipschitz character of $\Omega$ such that, for all $p\in [2,p_{0})$, and all balls $B\subset \R^{n}$ centered on $\partial \Omega$, we have
    \begin{equation}\label{eq:reverse-holder-general}
        \left(\fint_{B} \left(\frac{d\omega^{x_{0}}}{d\sigma_{\beta}}\right)^{p}d\sigma_{\beta}\right)^{1/p}\le C\frac{\omega^{x_{0}}(B)}{\sigma_{\beta} (B)},
    \end{equation}
    where $C$ depends only on $n, s, \beta, p$, the Lipschitz character of $\Omega$, and $\delta(x_{0})$.
\end{thm}

\begin{proof}
    Let $\beta \in (2s-1,s]$ and let $B=B_{r}(\xi)$ with $\xi\in \partial \Omega$ and $r>0$ be a boundary ball. We first prove \eqref{eq:reverse-holder-general} in the case in which $r<r_{0}$ and $p=2$. Here $r_{0}$ is chosen sufficiently small with respect to the localization radius of the domain and $\delta(x_{0})$. In this case, we may split the integral into two parts:
    \begin{equation*}
        \int_{B_{r}(\xi)}\left(\frac{d\omega^{x_{0}}}{d\sigma_{\beta}}\right)^{2}d\sigma_{\beta}\approx \underbrace{\frac{1}{1-\beta}\int_{B_{r}(\xi)\cap \Omega^{c}\cap \{\delta < c_{0}r\}}(P^{x_{0}})^{2}\delta^{\beta}}_{I}+ \underbrace{\frac{1}{1-\beta}\int_{B_{r}(\xi)\cap \Omega^{c}\cap \{\delta \ge  c_{0}r\}}(P^{x_{0}})^{2}\delta^{\beta}}_{II}.
    \end{equation*}
    To bound $I$, we use \eqref{eq:Linfty-L2-Poisson-kernel} and the coarea formula:
    \begin{equation}\label{eq:bound-I-rev-hold}
    \begin{aligned}
        I&= \frac{1}{1-\beta}\int_{0}^{c_{0}r}t^{\beta}\int_{B_{r}(\xi)\cap \Omega^{c}\cap \{\delta =t\}}(P^{x_{0}})^{2}\, d\mathcal{H}^{n-1}\, dt
        \\
        &\approx \frac{\omega^{x_{0}}(B_{r}(\xi))^{2}}{r^{n+1-2s}}\frac{(1-s)^{2}}{1-\beta}\int_{0}^{c_{0}r}t^{\beta-2s}\,dt\approx \frac{\omega^{x_{0}}(B_{r}(\xi))^{2}}{r^{n-\beta}},
    \end{aligned}
    \end{equation}
    where in the equivalence we also used the fact that $\mathcal{H}^{n-1}(B_{r}(\xi)\cap \Omega^{c} \cap \{\delta=t\})\approx r^{n-1}$ for all $t\in (0,c_{0}r)$. To bound $II$ instead, we notice that $P^{x_{0}}(y)\approx (1-s)r^{-n}\omega^{x_{0}}(B_{r}(\xi))$ for all $y\in B_{r}(\xi)\cap \Omega^{c}\cap \{\delta\ge c_{0}r\}$, thanks to Lemmas \ref{lem:comparison-omega-G} and \ref{lem:comparison-P-G}. Therefore, we have
    \begin{equation}\label{eq:bound-II-rev-hold}
        II\approx \frac{(1-s)^{2}}{1-\beta}\frac{\omega^{x_{0}}(B_{r}(\xi))^{2}}{r^{2n-\beta}}|B_{r}(\xi)\cap \Omega^{c}\cap \{\delta\ge c_{0}r\}|\approx \frac{\omega^{x_{0}}(B_{r}(\xi))^{2}}{r^{n-\beta}}. 
    \end{equation}
    Putting together \eqref{eq:bound-I-rev-hold} and \eqref{eq:bound-II-rev-hold}, and using the fact that $\sigma_{\beta}(B_{r}(\xi))\approx r^{n-\beta}$, we derive \eqref{eq:reverse-holder-general} in the case $p=2$.

    Next, still working with boundary balls of sufficiently small radius $r<r_{0}$, using the Gehring-type result from Lemma \ref{lem:gehring} we improve the exponent of the reverse-H\"{o}lder \eqref{eq:reverse-holder-general} from $2$ to any $p\in [2,p_{0})$, for some $p_{0}>2$. Although Lemma \ref{lem:gehring} is stated for dyadic cubes in $(0,1]^{n}$, after a proper local straightening of the Lipschitz boundary, the same result applies to our situation. Under these transformations, in the notation of Lemma \ref{lem:gehring}, $\sigma$ corresponds to $\sigma_{\beta}\approx (1-\beta)\delta^{-\beta}$, $f$ corresponds to $d\omega^{x_{0}}/d\sigma_{\beta}\approx P^{x_{0}}\delta^{\beta}/(1-\beta)$, assumption $(i)$ with exponent $2$ is the reverse-H\"{o}lder we just proved, and assumption $(ii)$ follows from $P^{x_{0}}(y)\approx (1-s)r^{-n}\omega^{x_{0}}(B_{r}(\xi))$ for $y\in B_{r}(\xi)\cap \Omega^{c}$, $\delta(y)\gtrsim r$, which is obtained combining Lemmas \ref{lem:comparison-omega-G} and \ref{lem:comparison-P-G}.

    Finally, let us see how to treat the case of boundary balls with radius $r$ larger than $r_{0}$. In this case, the right-hand side in \eqref{eq:reverse-holder-general} will be approximately bounded below by $1$, so it only remains to show that the left-hand side is approximately bounded above by $1$. To do that, we split 
    \begin{equation*}
        U^{-}\coloneqq B_{r}(\xi)\cap \Omega^{c}\cap \{\delta<r_{0}/4\}\subseteq \bigcup_{j=1}^{N}B_{r_{0}/2}(\xi_{j}),\qquad U^{+}\coloneqq B_{r}(\xi)\cap \Omega^{c}\cap \{\delta\ge r_{0}/4\},
    \end{equation*}
    where $\xi_{j}$ are points on $\partial \Omega$ and $N\in \N$ depends only on the Lipschitz character of $\Omega$. Then, on the one hand, using the first part of the proof, we have
    \begin{equation}\label{eq:reverse-big-balls-1}
        \int_{U^{-}}\left(\frac{d\omega^{x_{0}}}{d\sigma_{\beta}}\right)^{p}d\sigma_{\beta}\le \sum_{j=1}^{N}\sigma_{\beta}(B_{r_{0}/2}(\xi_{j}))\fint_{B_{r_{0}/2}(\xi_{j})}\left(\frac{d\omega^{x_{0}}}{d\sigma_{\beta}}\right)^{p}d\sigma_{\beta}\lesssim \sum_{j=1}^{N} \frac{\omega^{x_{0}}(B_{r_{0}/2}(\xi_{j}))^{p}}{\sigma_{\beta}(B_{r_{0}/2}(\xi_{j}))^{p-1}}\lesssim 1.
    \end{equation}
    On the other hand, since $P^{x_{0}}\approx (1-s)\delta^{-n-2s}$ and $d\sigma_{\beta}/d\mathscr{L}^{n}\approx (1-\beta)\delta^{-n-2s}$ on $U^{+}$, we can estimate
    \begin{equation}\label{eq:reverse-big-balls-2}
        \int_{U^{+}}\left(\frac{d\omega^{x_{0}}}{d\sigma_{\beta}}\right)^{p}d\sigma_{\beta}\approx \sigma_{\beta}(U^{+})\lesssim 1. 
    \end{equation}
    Putting \eqref{eq:reverse-big-balls-1} and \eqref{eq:reverse-big-balls-2} together, and using $\sigma_{\beta}(B_{r}(\xi))\gtrsim 1$ due to $r\ge r_{0}$, we deduce that the left-hand side in \eqref{eq:reverse-holder-general} is bounded above by a constant. This concludes the proof.
\end{proof}

\begin{proof}[Proof of Theorem \ref{thm:Dahlberg-intro}] This is Theorem \ref{thm:Dahlberg-general} in the special case $\beta=s$, $p=2$.
\end{proof}

\section{The Dirichlet problem in Lipschitz domains}\label{sec:Dir_pb}

The goal of this section is to develop the framework in which to formulate the fractional Dirichlet problem with unbounded exterior data in rough domains, and in particular to prove our second main result, Theorem~\ref{thm:Dirichlet-solvability-intro}; see Theorem~\ref{thm:Dirichlet-problem-general} for a more general statement.

Our analysis is based on the notion of a general Poisson integral, obtained by integrating the Martin kernel against an arbitrary finite measure defined in the exterior of the domain. After introducing suitable nonlocal analogs of the classical non-tangential maximal operator, we prove one of the key results of this part, Lemma \ref{lem:non-tangential-maximal-function-bound}, which provides a pointwise comparison between the non-tangential maximal function of a general Poisson integral and a boundary-centered Hardy-Littlewood maximal function of the corresponding data with respect to $s$-harmonic measure. This bound, together with the reverse-H\"{o}lder inequality from Theorem \ref{thm:Dahlberg-general}, constitutes the main ingredient in the derivation of \eqref{eq:L2-estimate-dirichlet-intro}. 

The uniqueness part in Theorem \ref{thm:Dirichlet-solvability-intro} follows from the general criterion established in Lemma \ref{lem:general-uniqueness}, which identifies the integrability of the nonlocal non-tangential maximal function with respect to harmonic measure as a condition for well-posedness of the exterior Dirichlet problem.  

\subsection{General Poisson integrals and non-tangential maximal functions}

To treat general measure-valued exterior data in the Dirichlet problem, we use the Martin kernel
normalized at a fixed reference pole; we refer the reader to \cite{bogdan1999representationMartin} and references therein for some background on this construction in Lipschitz domains. 
\begin{defn}
    Let $\Omega\subset \R^{n}$ be a bounded Lipschitz domain, and let $x_{0}\in \Omega$. The Martin kernel of $\Omega$ with reference point $x_{0}$ is the function $K^{x_{0}}:\Omega\times \Omega^{c}\to (0,\infty)$ defined as
    \begin{equation*}
        K^{x_{0}}(x,y)\coloneqq \lim_{r\to 0^{+}}\frac{\omega^{x}(B_{r}(y))}{\omega^{x_{0}}(B_{r}(y))}.
    \end{equation*}
\end{defn}

\begin{rmk}\label{rmk:martin-kernel}
By \cite[Lemmas 6 and 7]{bogdan1999representationMartin}, for any $x_{0}\in \Omega$, the Martin kernel $K^{x_{0}}$ is well-defined and continuous in $\Omega \times \Omega^{c}$, $K^{x_{0}}(x_{0},y)=1$, and 
    \begin{equation*}
        K^{x_{0}}(x,y)= \begin{cases}
            \frac{P^{x}(y)}{P^{x_{0}}(y)}\qquad &\text{if $y\in \interior \Omega^{c}$},\\
            \underset{\Omega\ni z\to y}{\lim}\frac{G^{x}(z)}{G^{x_{0}}(z)}\qquad &\text{if $y\in \partial \Omega$}.
        \end{cases}
    \end{equation*}
In addition, if we extend $K^{x_{0}}(\cdot, y)$ to a distribution in $\R^{n}$ defined as
\begin{equation*}
    \widetilde K^{x_{0}}(\cdot, y)\coloneqq\begin{cases}
   \mathbbm{1}_{\Omega} K^{x_{0}}(\cdot,y)+\frac{\delta_{y}}{P^{x_{0}}(y)}\qquad &\text{if $y\in \interior \Omega^{c}$},\\
    K^{x_{0}}(\cdot,y)\mathbbm{1}_{\Omega}\qquad &\text{if $y\in \partial \Omega$},
\end{cases}
\end{equation*}
then $(-\Delta)^{s}\widetilde{K}^{x_{0}}(\cdot,y)=0$ in $\Omega$ distributionally, and $\widetilde{K}^{x_{0}}(\cdot, y)$ vanishes continuously in $\Omega^{c}\setminus \{y\}$. In the sequel, we will not distinguish between $K^{x_{0}}$ and $\widetilde{K}^{x_{0}}$ for simplicity of notation. 
\end{rmk}

The continuity of the Martin kernel with respect to the exterior variable allows it to be tested against arbitrary finite measures defined in the exterior. The resulting function is the corresponding Poisson integral, which is $s$-harmonic in the interior variable.  
\begin{defn}
    Let $\Omega\subset \R^{n}$ be a bounded Lipschitz domain, and let $x_{0}\in \Omega$. For every finite measure $\nu \in \mathcal{M}(\Omega^{c})$, the Poisson integral of $\nu$ with respect to the reference point $x_{0}$ is the function $H^{x_{0}}\nu:\Omega\to \R$ defined as
    \begin{equation*}
        H^{x_{0}}\nu(x)\coloneqq \int_{\Omega^{c}}K^{x_{0}}(x,y)d\nu(y)\qquad \forall x\in \Omega.
    \end{equation*}
\end{defn} 

\begin{rmk}
    By Remark \ref{rmk:martin-kernel}, $H^{x_{0}}\nu$ is well-defined in $\Omega$ for any $x_{0}\in \Omega$ and any $\nu \in \mathcal{M}(\Omega^{c})$. Moreover, if we extend $H^{x_{0}}\nu$ to the distribution in $\R^{n}$ defined as
    \begin{equation*}
        \widetilde{H}^{x_{0}}\nu\coloneqq \mathbbm{1}_{\Omega}H^{x_{0}}\nu +\frac{1}{P^{x_{0}}(\cdot)}\nu \res (\interior \Omega^{c}),
    \end{equation*}
    then $(-\Delta)^{s}\widetilde{H}^{x_{0}}\nu =0$ in $\Omega$, distributionally. Also here, for simplicity, we will not distinguish between $H^{x_{0}}\nu$ and $\widetilde{H}^{x_{0}}\nu$ in the sequel.
\end{rmk}

We now introduce several geometric objects that play an important role in the definition of solvability classes for the exterior Dirichlet problem. We begin with the following nonlocal analogs of the classical notions of non-tangential approach region and non-tangential maximal function; see Figure \ref{fig:nonlocal-cone}.
\begin{defn}\label{def:non-tangential-maximal-function}
    Let $\Omega\subset \R^{n}$ be any open set. For every $y \in \Omega^{c}$, the \emph{non-tangential region relative to $y$} is defined as
    \begin{equation*}
        \Gamma(y)\coloneqq \left\{x\in \Omega: \delta(x)\ge \frac{|x-y|}{2}\right\}.
    \end{equation*}
    For every $v:\Omega\to \R$, the \emph{non-tangential maximal function of $v$} is the function $v^{*}:\Omega^{c}\to [0,\infty]$ defined as
    \begin{equation*}
        v^{*}(y)\coloneqq \begin{cases}
            \sup_{x\in \Gamma(y)}|v(x)|\qquad &\text{if $\Gamma(y)\neq \varnothing$},\\
            0\qquad &\text{if $\Gamma(y)=\varnothing$},
        \end{cases}\qquad \forall y \in \Omega^{c}.
    \end{equation*}
\end{defn}

\begin{rmk}
    By the definition of non-tangential cone region, $\Gamma(y)=\varnothing$ for all $y\in \Omega^{c}$ such that $\delta(y)>\diam \Omega$. In particular, for every $v:\Omega\to \R$, the non-tangential maximal function $v^{*}:\Omega^{c}\to [0,\infty]$ is supported in $\Omega^{c}\cap \{\delta\le \diam \Omega\}$. 
\end{rmk}

\begin{figure}[H]
\centering
\includegraphics[width=0.5\textwidth]{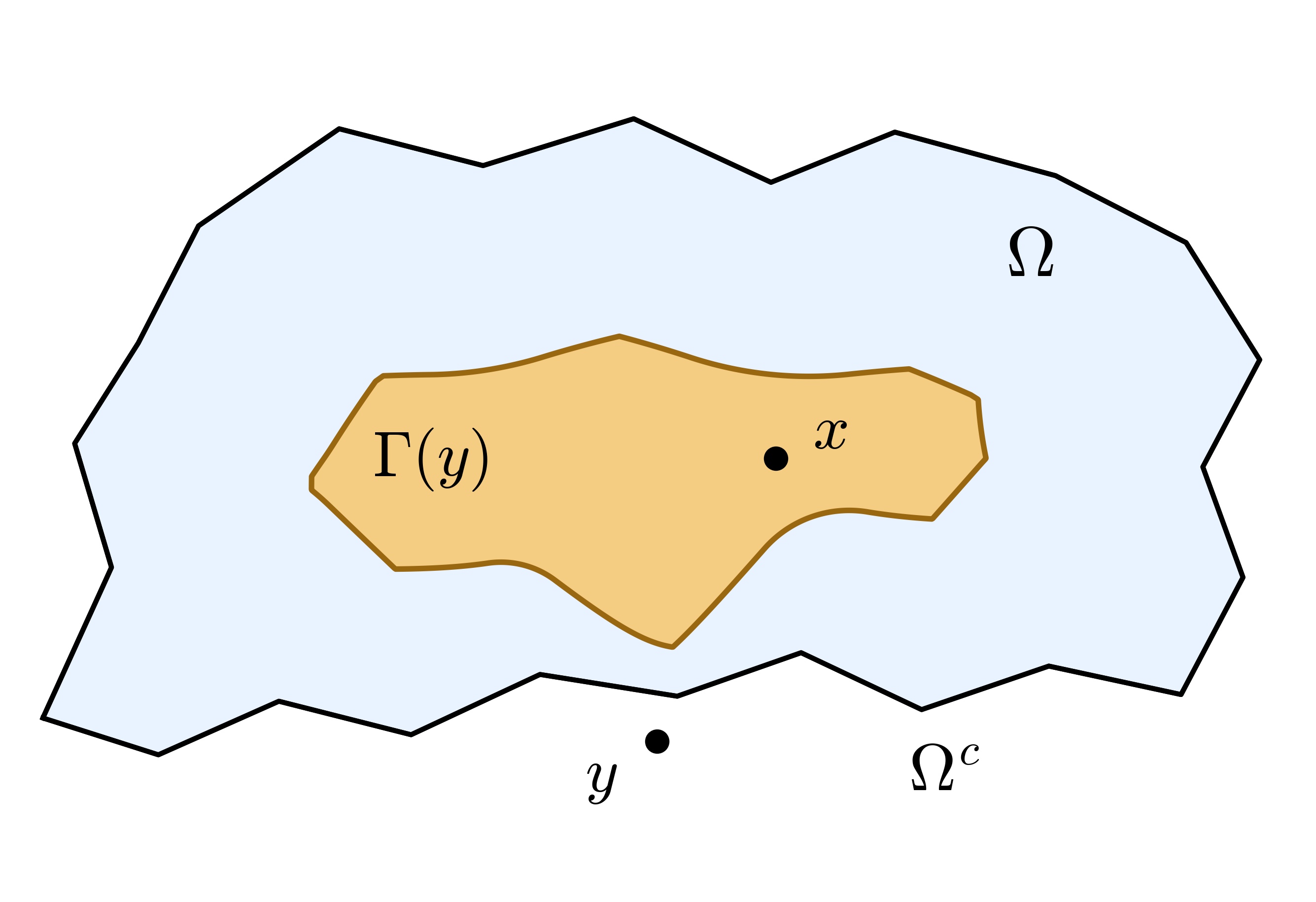}
\caption{Schematic picture of the nonlocal non-tangential region $\Gamma(y)$.}
\label{fig:nonlocal-cone}
\end{figure}

As shown in the following elementary lemma, for Lipschitz domains, exterior integral control of the non-tangential maximal function implies (and, in fact, is much stronger than) the corresponding interior estimate.

\begin{lem}\label{lem:comparison-inside-outside-nontangential}
    Let $\Omega\subset \R^{n}$ be a bounded Lipschitz domain. Then, for every $v:\Omega\to \R$ we have 
    \begin{equation}\label{eq:geometric-inequality-nontang-max-function}
        \int_{\Omega}\sup_{B_{\delta(x)/4}(x)}|v|\,dx\le C\int_{\Omega^{c}}v^{*}(y)\,dy,
    \end{equation}
    where $C$ depends only on the Lipschitz character of $\Omega$. 
\end{lem}
\begin{proof}
    Let $r_{0}>0$ be sufficiently small with respect to the Lipschitz character of $\Omega$. Then, there is a small constant $c_{0}>0$ depending only on $r_{0}$ such that, for any $x\in \Omega\cap \{\delta\ge r_{0}\}$, there is a boundary ball $B^{x}=B_{c_{0}}(\xi_{x})$ with $\xi_{x}\in \partial \Omega$ such that $B_{\delta(x)/4}(x)\subset \Gamma(y)$ for all $y\in B^{x}$. In particular 
    \begin{equation}\label{eq:bound-inside-geometric-lemma}
    \begin{aligned}
        \int_{\Omega\cap \{\delta\ge r_{0}\}}\sup_{B_{\delta(x)/4}(x)}|v|\,dx&\le \int_{\Omega\cap \{\delta\ge r_{0}\}}\fint_{B^{x}\cap \Omega^{c}}v^{*}(y)\,dy\,dx\\
        &\lesssim \frac{|\Omega\cap \{\delta\ge r_{0}\}|}{c_{0}^{n}}\int_{\Omega^{c}}v^{*}(y)\,dy\lesssim \int_{\Omega^{c}}v^{*}(y)\,dy.
    \end{aligned}
    \end{equation}
    On the other hand, $\Omega\cap \{\delta\le r_{0}\}$ can be covered by the images $V_{i}$ of $N$ bilipschitz maps $F_{i}:U_{i}\to V_{i}$ where $U_{i}\subset \Omega^{c}$, $V_{i}\subset \Omega$, and $B_{\delta(F_{i}(y))/4}(F_{i}(y))\subset \Gamma(y)$ for all $y\in U_{i}$. Here $N$ and the Lipschitz constants of $F_{i}$ depend only on the Lipschitz character of $\Omega$. Therefore, by the change of variables formula
    \begin{equation}\label{eq:bound-boundary-geometric-lemma}
        \begin{aligned}
            \int_{\Omega\cap \{\delta\le r_{0}\}}\sup_{B_{\delta(x)/4}(x)}|v|\,dx &\le \sum_{i=1}^{N}\int_{V_{i}}\sup_{B_{\delta(x)/4}(x)}|v|\,dx\\
            &\lesssim \sum_{i=1}^{N} \int_{U_{i}} \sup_{B_{\delta(F_{i}(y))/4}(F_{i}(y))}|v|\,dy\le \sum_{i=1}^{N}\int_{U_{i}}v^{*}(y)\,dy\lesssim \int_{\Omega^{c}}v^{*}(y)\,dy.
        \end{aligned}
    \end{equation}
    Putting \eqref{eq:bound-inside-geometric-lemma} and \eqref{eq:bound-boundary-geometric-lemma} together, we deduce the desired inequality.
\end{proof}
\begin{rmk}\label{rmk:exterior-interior-weighted}
     We shall also use an exterior-interior comparison of distance-weighted $L^{q}$-norms. Let $\beta\ge 0$ and $q\ge 1$. Then, given $u:\Omega\to \R$, by the definition of the non-tangential region, we have $(|u|^{q}\delta^{-\beta})^{*}\lesssim (u^{*})^{q}\delta^{-\beta}$. Therefore, Lemma \ref{lem:comparison-inside-outside-nontangential} yields  
     \begin{equation*}
         \lVert u\rVert_{L^{q}(\Omega,\delta^{-\beta})}\lesssim \lVert u^{*}\rVert_{L^{q}(\Omega^{c},\delta^{-\beta})}.
     \end{equation*}
\end{rmk}

Next, we introduce a Hardy-Littlewood maximal operator built from boundary-centered balls. When the reference measure is doubling on such balls, this operator enjoys the usual $L^{p}$ estimates for maximal functions; see Lemma \ref{lem:maximal-inequality-general-weights}.

\begin{defn}\label{def:boundary-balls-maximal-operator}
    For any open set $\Omega\subset \R^{n}$, we denote by $\mathcal{B}(\Omega)$ the set of boundary balls
\begin{equation*}
    \mathcal{B}(\Omega)\coloneqq \{B_{r}(\xi):\xi\in \partial\Omega,\, r>0\}.
\end{equation*}
For any choice of finite measures $\nu \in \mathcal{M}(\Omega^{c})$ and $\mu\in \mathcal{M}_{+}(\Omega^{c})$ such that $\mu(B)>0$ for all $B\in \mathcal{B}(\Omega)$, the boundary maximal function $M^{\Omega}_{\mu}\nu: \Omega^{c}\to [0,\infty]$ is defined as 
\begin{equation*}
    M^{\Omega}_{\mu}\nu(y)\coloneqq\sup_{y\in B\in \mathcal{B}(\Omega)}\frac{|\nu|(B)}{\mu(B)}\qquad \forall y\in \Omega^{c}.
\end{equation*}
\end{defn}

In the following lemma we show that the non-tangential maximal function of the Poisson integral corresponding to a given finite exterior measure $\nu$ is pointwise controlled by the boundary
maximal operator of $\nu$ with respect to harmonic measure.

\begin{lem}\label{lem:non-tangential-maximal-function-bound}
Let $\Omega\subset \R^{n}$ be a bounded Lipschitz domain, and let $x_{0}\in \Omega$. For any finite measure $\nu \in \mathcal{M}(\Omega^{c})$, the following holds:
\begin{equation*}
    (H^{x_{0}}\nu)^{*}(y)\lesssim M^{\Omega}_{\omega^{x_{0}}}\nu (y)\qquad \forall y \in \Omega^{c},
\end{equation*}
where the constant depends only on $n,s$, the Lipschitz character of $\Omega$, and $\delta(x_{0})$. Moreover, if $\nu$ is nonnegative, then 
\begin{equation*}
    (H^{x_{0}}\nu)^{*}(y)\approx M^{\Omega}_{\omega^{x_{0}}}\nu (y)\qquad \forall y \in \Omega^{c},\quad \Gamma(y)\neq \varnothing.
\end{equation*}
\end{lem}

\begin{proof}
 Fix $y\in \Omega^{c}$. If $\Gamma(y)=\varnothing$ there is nothing to prove, so let us assume that $\delta(y)\le \diam \Omega$ and $\Gamma(y)\neq \varnothing$. Take $x\in \Gamma(y)$, and call $r\coloneqq |x-y|\in [\delta(x),2\delta(x)]$. 

In the sequel, $r_{0}>0$ will be chosen sufficiently small with respect to the localization radius of the domain and $\delta(x_{0})$. Suppose first that $r\ge r_{0}$. Then, the fact that $K^{x_{0}}(x_{0},z)=1$ and the Harnack inequality applied to $K^{x_{0}}(\cdot,z)$ yield
    \begin{equation*}
        K^{x_{0}}(x,z)\approx 1\qquad \forall z\in \Omega^{c}.
    \end{equation*}
    In particular, since $\omega^{x_{0}}$ is a probability measure, in this case we obtain
    \begin{equation*}
        |H^{x_{0}}\nu(x)|=\left|\int_{\Omega^{c}}K^{x_{0}}(x,z)\,d\nu(z)\right|\lesssim |\nu|(\R^{n})\le M^{\Omega}_{\omega^{x_{0}}}\nu(y).
    \end{equation*}
    
    We now consider the case $r<r_{0}$. Let $\xi \in \partial \Omega$ be a point in the segment that connects $x$ to $y$, and note that we must have $|x-\xi|\in [r/2,r]$. We then define
    \begin{equation*}
        R_{0}\coloneqq B_{r}(\xi)\cap \Omega^{c},\qquad R_{k}\coloneqq(B_{2^{k}r}(\xi)\setminus B_{2^{k-1}r}(\xi))\cap \Omega^{c}\quad \forall k\ge 1.
    \end{equation*}
    Let $\bar k$ be the largest nonnegative integer for which $2^{\bar k}r\le r_{0}$. 
    We claim that 
    \begin{equation}\label{eq:claim-lem-max-functions}
        \sup_{z\in R_{k}}K^{x_{0}}(x,z)\lesssim 2^{-k\alpha}\frac{1}{\omega^{x_{0}}(B_{2^{k}r}(\xi))}\qquad\forall k\in \{0,\dots,\bar k\}.
    \end{equation}
    We first prove \eqref{eq:claim-lem-max-functions} in the case $k=0$. Take any $z\in B_{r}(\xi)\cap \Omega^{c}$, and let $\xi'\in \partial \Omega$ be such that $|z-\xi'|=\delta(z)=:r'$. Then, using the comparison Lemmas \ref{lem:comparison-omega-G}, \ref{lem:comparison-P-G}, and the change of pole formula from Lemma \ref{lem:change-of-pole}, we derive the following chain of equivalences:
    \begin{equation}\label{eq:equivalence-martin-first-ball}
        K^{x_{0}}(x,z)=\frac{P^{x}(z)}{P^{x_{0}}(z)}\approx \frac{G^{x}(A_{r'}(\xi'))}{G^{x_{0}}(A_{r'}(\xi'))}\approx \frac{1}{\omega^{x_{0}}(B_{r}(\xi))}.
    \end{equation}
    Let now $k\in \{1,\dots,\bar k\}$, and take $z\in R_{k}$. By Remark \ref{rmk:martin-kernel}, $K^{x_{0}}(\cdot,z)$ is $s$-harmonic in $B_{2^{k-1}r}(\xi)\cap \Omega$ and vanishes continuously in $B_{2^{k-1}r}(\xi)\cap \Omega^{c}$. Therefore, after applying Theorem \ref{thm:holder-boundary-lip-domain} and adjusting with the Harnack inequality, we get
    \begin{equation*}
        K^{x_{0}}(x,z)\lesssim \left(\frac{\delta(x)}{\delta(A_{2^{k}r}(\xi))}\right)^{\alpha}K^{x_{0}}(A_{2^{k}r}(\xi),z)\approx 2^{-k\alpha}\frac{1}{\omega^{x_{0}}(B_{2^{k}r}(\xi))},
    \end{equation*}
    where the last equivalence follows from the same argument as in the case $k=0$, with $r$ now replaced by $2^{k}r$. This concludes the proof of claim \eqref{eq:claim-lem-max-functions}. 

    Arguing as in case $k=\bar k$ above, and noting that, by Harnack, $K^{x_{0}}(A_{2^{\bar k}r}(\xi),z)\approx 1$, we also get
    \begin{equation}\label{eq:tails-lem-max-functions}
        \sup_{\Omega^{c}\setminus B_{2^{\bar k}r}(\xi)}K^{x_{0}}(x,\cdot)\lesssim 1.
    \end{equation}
    Finally, putting \eqref{eq:claim-lem-max-functions} and \eqref{eq:tails-lem-max-functions} together, and using the fact that $y\in B_{2^{k}r}(\xi)$ for all $k\ge 0$, we find
    \begin{align*}
        |H^{x_{0}}\nu(x)|&=\left|\int_{\Omega^{c}}K^{x_{0}}(x,z)\,d\nu(z)\right|\\
        &\le \sum_{k=0}^{\bar k}\sup_{R_{k}}K^{x_{0}}(x,\cdot)|\nu|(R_{k})+ |\nu|(\Omega^{c}\setminus B_{2^{\bar k}r}(\xi))\sup_{\Omega^{c}\setminus B_{2^{\bar k}r}(\xi)}K^{x_{0}}(x,\cdot)\\
        &\lesssim \sum_{k=0}^{\bar k}2^{-k\alpha}\frac{|\nu|(B_{2^{k}r}(\xi))}{\omega^{x_{0}}(B_{2^{k}r}(\xi))}+|\nu|(\R^{n})\lesssim M^{\Omega}_{\omega^{x_{0}}}\nu(y). 
    \end{align*}
    By the arbitrariness of $x\in \Gamma(y)$, we get the desired inequality. To prove that this becomes an equivalence when $\nu$ is nonnegative, for a given $y\in \Omega^{c}$ for which $\Gamma(y)\neq \varnothing$, with the notations of the above proof, it suffices to choose $x\in \Gamma(y)$ such that $B_{r}(\xi)$ achieves a positive fraction of the supremum in the definition of maximal operator. Then, the equivalence in \eqref{eq:equivalence-martin-first-ball} allows us to conclude that $H^{x_{0}}\nu(x)\gtrsim M^{\Omega}_{\omega^{x_{0}}}\nu (y)$.
\end{proof}

\subsection[Lq solvability of the Dirichlet problem]{\texorpdfstring{$L^{q}$-solvability}{Lq solvability} of the Dirichlet problem}

The following lemma provides a general uniqueness criterion for the exterior Dirichlet problem, formulated in terms of the integrability of the nonlocal non-tangential maximal function with respect to $s$-harmonic measure. The sharpness of this condition is discussed in Remark \ref{rmk:sharpness-uniqueness-criterion}.
\begin{lem}\label{lem:general-uniqueness}
    Let $\Omega\subset \R^{n}$ be a bounded Lipschitz domain, and let $u\in L^{1}(\R^{n},w_{s})$ be a distributional solution of $(-\Delta)^{s}u=0$ in $\Omega$ such that $u=0$ in $\Omega^{c}$. If $u^{*}\in L^{1}(\Omega^{c}, \omega^{x_{0}})$ for some $x_{0}\in \Omega$, then  $u\equiv 0$.
\end{lem}
\begin{proof}
    First of all, note that since harmonic measures with different poles are comparable, $u^{*}\in L^{1}(\Omega^{c},\omega^{x})$ for all $x\in \Omega$. Therefore, it suffices to prove that $u(x_{0})=0$. For any $\eps\in (0,\delta(x_{0}))$, we consider a smooth function $\varphi_{\eps}\in C^{\infty}(\R^{n})$ with values in $[0,1]$ such that
    \begin{equation*}
        \varphi_{\eps}=1\quad \text{in $\{x\in \R^{n}: \delta(x)\ge \eps\}$},\qquad \varphi_{\eps}=0 \quad \text{in $\{x\in \R^{n}: \delta(x)\le \eps/2\}$},\qquad |\nabla \varphi_{\eps}|\lesssim \eps^{-1}.
    \end{equation*} 
    Using the fact that $\varphi_{\eps}(x_{0})=1$ and the $s$-harmonicity of $u$, we may write 
    \begin{align*}
        u(x_{0})&=u(x_{0})\varphi_{\eps}(x_{0})\\&= \int_{\Omega}G^{x_{0}}(-\Delta)^{s}(u\varphi_{\eps})
       = \int_{\Omega}uB_{s}(G^{x_{0}},\varphi_{\eps})-\int_{\Omega}G^{x_{0}}B_{s}(u,\varphi_{\eps})=:I_{\eps}+II_{\eps}.
    \end{align*}
    
    We will prove that $II_{\eps}\to 0$ as $\eps\to 0^{+}$. A similar argument with $u$ and $G^{x_{0}}$ swapped shows that the same conclusion holds for $I_{\eps}$, concluding the proof.
    By Fubini's theorem, we can bound
    \begin{align*}
        |II_{\eps}|&\le \int_{\Omega\cap \{\delta\ge r_{0}\}}G^{x_{0}}|B_{s}(u,\varphi_{\eps})|+ \int_{\Omega\cap \{\delta< r_{0}\}}G^{x_{0}}|B_{s}(u,\varphi_{\eps})|\\ 
        &\lesssim \int_{\Omega\cap \{\delta\ge r_{0}\}}G^{x_{0}}|B_{s}(u,\varphi_{\eps})|+\int_{\Omega\cap \{\delta< r_{0}\}}\fint_{B_{\delta(z)/8}(z)}G^{x_{0}}|B_{s}(u,\varphi_{\eps})|=:II_{\eps}^{1}+II_{\eps}^{2}.
    \end{align*}
    Here $r_{0}>0$ is chosen sufficiently small with respect to the localization radius of the domain and $\delta(x_{0})$. 
    
    Let us treat $II_{\eps}^{1}$ first. Since $x\in \Gamma(y)$ for all $y\in \Omega^{c}$ in some boundary ball with radius  comparable to $\delta(x)$, thanks to the assumption $u^{*}\in L^{1}(\Omega^{c},\omega^{x_{0}})$, $|u(x)|$ will be uniformly bounded for $x\in \{\delta\ge r_{0}\}$. Then
    for $\eps\ll r_{0}$ and $x\in \Omega \cap \{\delta\ge r_{0}\}$, by the definition of $\varphi_{\eps}$, we have 
    \begin{align*}
        |B_{s}(u,\varphi_{\eps})(x)|\lesssim \int_{\{\delta\le \eps\}}\frac{|u(x)-u(y)|}{|x-y|^{n+2s}}\,dy\le \left(\frac{2}{r_{0}}\right)^{n+2s}\left(|u(x)||\{\delta\le \eps\}|+\int_{\{\delta\le \eps\}}|u|\,dy\right),
    \end{align*}
    which is uniformly bounded in $\{\delta\ge r_{0}\}$ by the previous observation and the fact that $u\in L^{1}(\Omega)$, and converges to zero pointwise as $\eps\to 0^{+}$. Then, the basic bound \eqref{eq:basic-estimate-green-function-1} together with dominated convergence theorem implies that $II_{\eps}^{1}\to 0$ as $\eps\to 0^{+}$.

    Next we consider $II_{\eps}^{2}$. Given $z\in \Omega\cap \{\delta\le r_{0}\}$, by the properties of $\varphi_{\eps}$, we can bound
    \begin{equation}\label{eq:bound-Hs-phi-eps-uniqueness}
        [\varphi_{\eps}]_{H^{s}(B_{\delta(z)/4}(z))}^{2}\lesssim \delta(z)^{n-2s}\mathbbm{1}_{\{\delta\le 4\eps/3\}}(z).
    \end{equation}
    Moreover, by rescaling the Caccioppoli inequality from Lemma \ref{lem:Caccioppoli}, we obtain
    \begin{equation}\label{eq:bound-Hs-u-uniqueness}
        [u]_{H^{s}(B_{\delta(z)/4}(z))}^{2}\lesssim \delta(z)^{-2s}\lVert u\rVert_{L^{2}(B_{\delta(z)/2}(z))}^{2}+\lVert u\rVert_{L^{1}(B_{\delta(z)/2}(z))}\int_{\R^{n}}\frac{|u(y)|}{\delta(z)^{n+2s}+|y-z|^{n+2s}}\,dy. 
    \end{equation}
    We then split
     \begin{align*}
         \fint_{B_{\delta(z)/8}(z)}|B_{s}(u,\varphi_{\eps})|&\lesssim \delta(z)^{-n}[u]_{H^{s}(B_{\delta(z)/4}(z))}[\varphi_{\eps}]_{H^{s}(B_{\delta(z)/4}(z))}\\
         &\quad\,+\frac{c_{n,s}}{2}\fint_{B_{\delta(z)/8}(z)}\int_{B^{c}_{\delta(z)/4}(z)}\frac{|u(x)-u(y)||\varphi_{\eps}(x)-\varphi_{\eps}(y)|}{|x-y|^{n+2s}}\,dy\,dx.\\
     \end{align*}
     Using \eqref{eq:bound-Hs-phi-eps-uniqueness} and \eqref{eq:bound-Hs-u-uniqueness}, we see that the right-hand side converges to zero as $\eps\to 0^{+}$ for every fixed $z\in \{\delta< r_{0}\}$, and that for all $z\in \{\delta<r_{0}\}$ and all $\eps$, it is bounded, up to a constant, by the function
     \begin{equation*}
         \zeta(z)\coloneqq \delta(z)^{-2s}\sup_{B_{\delta(z)/2}(z)}|u|+\int_{\Omega}\frac{|u(y)|}{\delta(z)^{n+2s}+|y-z|^{n+2s}}\,dy.
     \end{equation*}
     Since in addition $G^{x_{0}}(x)\approx G^{x_{0}}(z)$ for all $x\in B_{\delta(z)/8}(z)$ by the Harnack inequality, 
     the dominated convergence theorem implies that $II_{\eps}^{2}\to 0$ as $\eps\to 0^{+}$ provided that
     \begin{equation}\label{eq:integral-DCT-uniqueness}
         \int_{\Omega\cap \{\delta< r_{0}\}}G^{x_{0}}(z)\zeta(z)\,dz<\infty.
     \end{equation}
     However,
     by Lemma \ref{lem:comparison-P-G} and the definition of non-tangential maximal function we have 
     \begin{equation*}
         \int_{\Omega\cap \{\delta<r_{0}\}}G^{x_{0}}(z)\delta(z)^{-2s}\sup_{B_{\delta(z)/2}(z)}|u|\,dz\lesssim \int_{\Omega^{c}}P^{x_{0}}(y)u^{*}(y)\,dy=\lVert u^{*}\rVert_{L^{1}(\Omega^{c},\omega^{x_{0}})}<\infty.
     \end{equation*}
     Similarly, by Fubini's theorem 
     \begin{align*}
         \int_{\Omega\cap \{\delta<r_{0}\}}G^{x_{0}}(z)\int_{\Omega\cap \{\delta\le 2r_{0}\}}\frac{|u(y)|}{\delta(z)^{n+2s}+|y-z|^{n+2s}}\,dy\,dz &\le  \int_{\Omega\cap \{\delta\le 2r_{0}\}}\int_{\Omega}\frac{|u(y)|G^{x_{0}}(z)}{\delta(z)^{n+2s}+|y-z|^{n+2s}}\,dz\,dy\\
         &\lesssim \int_{\Omega^{c}}u^{*}(y')P^{x_{0}}(y')\,dy'.
     \end{align*}
     Here we also used the fact that $\delta(z)+|z-y|\approx |z-y'|$ for all $z\in \Omega$, $y\in \Omega \cap \{\delta\le 2r_{0}\}$ and $y'\in \Omega^{c}$ such that $\delta(y')\approx |y-y'|\approx \delta(y)$. 
     Since we already saw that $u$ is uniformly bounded in $\Omega \cap \{\delta\ge 2r_{0}\}$, this concludes the proof of \eqref{eq:integral-DCT-uniqueness}. 
\end{proof}

\begin{rmk}\label{rmk:sharpness-uniqueness-criterion}
    The integrability condition on the non-tangential maximal function with respect to harmonic measure is actually sharp in Lemma \ref{lem:general-uniqueness}. Consider for instance the function $u\coloneqq H^{x_{0}}\nu$ with $\nu=\delta_{\xi}$, for some $\xi\in \partial \Omega$: then $u\not\equiv 0$, but it vanishes in the exterior and is $s$-harmonic in $\Omega$. As a consequence of Lemma \ref{lem:non-tangential-maximal-function-bound}, for every $r\in (0,r_{0})$ we find
    \begin{equation*}
        u^{*}(y)\approx \frac{1}{\omega^{x_{0}}(B_{r}(\xi))}\qquad \forall y\in \Omega^{c}\cap \left(B_{r}(\xi)\setminus B_{r/2}(\xi)\right). 
    \end{equation*}
    Using the fact that $r^{n-\alpha}\lesssim \omega^{x_{0}}(B_{r}(\xi))\lesssim r^{n-2s+\alpha}$ (as a consequence of Lemma \ref{lem:comparison-omega-G} and Theorem \ref{thm:holder-boundary-lip-domain}), one deduces that $u^{*}\in L^{1}(\Omega^{c})$, which in turn implies $u\in L^{1}(\Omega)$ thanks to Lemma \ref{lem:comparison-inside-outside-nontangential}, but $u^{*}\notin L^{1}(\Omega^{c},\omega^{x_{0}})$.
\end{rmk}

Combining the reverse-H\"{o}lder result from Theorem \ref{thm:Dahlberg-general}, the maximal estimate in Lemma \ref{lem:non-tangential-maximal-function-bound}, and the 
uniqueness criterion from Lemma \ref{lem:general-uniqueness} gives the following $L^{q}$-solvability theory for the Dirichlet problem:
\begin{thm}\label{thm:Dirichlet-problem-general}
    Let $\Omega\subset \R^{n}$ be a bounded Lipschitz domain. Given $\beta \in (2s-1,s]$, let $\sigma_{\beta}$ be the measure in \eqref{eq:def-sigma-beta}. Then, there exists $q_{0}\in [1,2)$ depending only on $n, s, \beta$, and the Lipschitz character of $\Omega$, such that the following holds. For every $q>q_{0}$ and every $g\in L^{q}(\Omega^{c}, \sigma_{\beta})$, the Poisson integral 
    \begin{equation}\label{eq:Poisson-solution-statement-Lp-solvability}
        u(x)\coloneqq \int_{\Omega^{c}}P^{x}(y)g(y)\,dy
    \end{equation}
    is well-defined, and once extended by $u=g$ on $\Omega^c$, is the unique distributional solution of
    \begin{equation*}
    \left\{
    \begin{array}{rclll}
         (-\Delta)^{s}u&=&0\quad &\text{in $\Omega$},\\
            u&=&g\quad &\text{in $\Omega^{c}$},\\
            \lVert u^{*}\rVert_{L^{q}(\Omega^{c}, \sigma_{\beta})}&<&\infty.
    \end{array}
    \right.
    \end{equation*}
    Moreover, it holds 
    \begin{equation}\label{eq:Lq-bound-maximal-function}
        \lVert u^{*}\rVert_{L^{q}(\Omega^{c},\sigma_{\beta})}\le C\lVert g\rVert_{L^{q}(\Omega^{c},\sigma_{\beta})},
    \end{equation}
    where the constant $C>0$ depends only on $n, s, \beta, q$, and the Lipschitz character of $\Omega$.
\end{thm}
\begin{proof}
    Let us fix a reference point $x_{0}\in \Omega$ whose distance from the boundary is approximately bounded below by the localization radius of the domain. Let $p_{0}>2$ be the exponent given by Theorem \ref{thm:Dahlberg-general}, and set $q_{0}\coloneqq p_{0}'<2$, the H\"{o}lder conjugate of $p_{0}$. Given $q>q_{0}$ and $g\in L^{q}(\Omega^{c},\sigma_{\beta})$, let $p=q'<p_{0}$, and notice that by Theorem \ref{thm:Dahlberg-general} and H\"{o}lder inequality, we have
    \begin{equation*}
        \int_{\Omega^{c}}|g|\,d\omega^{x_{0}}=\int_{\Omega^{c}}|g|\frac{d\omega^{x_{0}}}{d\sigma_{\beta}}\,d\sigma_{\beta}\le \left(\int_{\Omega^{c}}|g|^{q}\,d\sigma_{\beta}\right)^{1/q}\left(\int_{\Omega^{c}}\left(\frac{d\omega^{x_{0}}}{d\sigma_{\beta}}\right)^{p}d\sigma_{\beta}\right)^{1/p}\lesssim \lVert g\rVert_{L^{q}(\Omega^{c},\sigma_{\beta})}.
    \end{equation*}
    Therefore $g\omega^{x_{0}}$ is a well-defined finite measure in $\Omega^{c}$ and the function $u$ in \eqref{eq:Poisson-solution-statement-Lp-solvability} is well-defined and coincides with the Poisson integral $u=H^{x_{0}}(g\omega^{x_{0}})$. In particular, by Lemma \ref{lem:non-tangential-maximal-function-bound}, we have $u^{*}\lesssim M^{\Omega}_{\omega^{x_{0}}}(g\omega^{x_{0}})$. Now, let $\tilde{q}\coloneqq (q+q_{0})/2>q_{0}$ and $\tilde{p}\coloneqq\tilde{q}'<p_{0}$. For every boundary ball $B\in \mathcal{B}(\Omega)$, \eqref{eq:reverse-holder-general} implies
    \begin{align*}
        \fint_{B}|g|\,d\omega^{x_{0}}&=\frac{\sigma_{\beta}(B)}{\omega^{x_{0}}(B)}\fint_{B}|g|\frac{d\omega^{x_{0}}}{d\sigma_{\beta}}\,d\sigma_{\beta}\\
        &\le \frac{\sigma_{\beta}(B)}{\omega^{x_{0}}(B)} \left(\fint_{B}|g|^{\tilde q}\,d\sigma_{\beta}\right)^{1/\tilde q}\left(\fint_{B}\left(\frac{d\omega^{x_{0}}}{d\sigma_{\beta}}\right)^{\tilde p}d\sigma_{\beta}\right)^{1/\tilde p}\lesssim \left(\fint_{B}|g|^{\tilde q}\,d\sigma_{\beta}\right)^{1/\tilde q}.
    \end{align*}
    Therefore, $u^{*}\lesssim M^{\Omega}_{\omega^{x_{0}}}(g\omega^{x_{0}})\lesssim \left(M^{\Omega}_{\sigma_{\beta}}(|g|^{\tilde q}\sigma_{\beta})\right)^{1/\tilde q}$, and \eqref{eq:Lq-bound-maximal-function} follows from the strong $L^{r}(\Omega^{c},\sigma_{\beta})$-estimate   with $r=q/\tilde{q}>1$ on the maximal operator $h\mapsto M^{\Omega}_{\sigma_{\beta}}(h \sigma_{\beta})$ proved in Lemma \ref{lem:maximal-inequality-general-weights}.

    Uniqueness follows directly from Lemma \ref{lem:general-uniqueness} after noticing that by H\"{o}lder inequality, $d\omega^{x_{0}}/d\sigma_{\beta}\in L^{p}(\Omega^{c},\sigma_{\beta})$ and $u^{*}\in L^{q}(\Omega^{c},\sigma_{\beta})$ together imply $u^{*}\in L^{1}(\Omega^{c},\omega^{x_{0}})$.
\end{proof}

\begin{proof}[Proof of Theorem \ref{thm:Dirichlet-solvability-intro}]
    This is Theorem \ref{thm:Dirichlet-problem-general} in the special case $\beta=s, q=2$.
\end{proof}

\section{Proofs of the corollaries}\label{sec:proof-corollaries}
In this section, we prove the corollaries stated in Section \ref{subsec:applications}. We start with Corollary \ref{cor:Linfty-L2-parallel-sets-harmonic}, which we establish as a consequence of the square estimate for the Green function in Theorem \ref{thm:uniform-L2-Green-function} together with a localization argument based on the boundary Harnack principle. 

\begin{proof}[Proof of Corollary \ref{cor:Linfty-L2-parallel-sets-harmonic}]
    Let us consider the auxiliary bounded Lipschitz domain $\widetilde\Omega\coloneqq \Omega \cap B_{3/4}$. We let $v\in C(B_1)\cap L^{1}(\R^{n},w_{s})$ be the nonnegative solution of the Dirichlet problem 
    \begin{equation*}
    \left\{
    \begin{array}{rclll}
         (-\Delta)^{s}v&=&0\quad &\text{in $\widetilde \Omega$},\\
            v&=&|u|\quad &\text{in $\widetilde{\Omega}^{c}$}.
    \end{array}
    \right.
    \end{equation*}
    Then, the maximum principle applied to $v+ u$ and $v-u$ gives $|u|\le v$ in $\widetilde \Omega$. Moreover, at the corkscrew point $e_{n}/2 \in \widetilde \Omega$, $v$ can be bounded by
    \begin{align*}
        v\left(\frac{e_{n}}{2}\right)&= \int_{\widetilde{\Omega}^{c}}P^{e_{n}/2}_{\widetilde \Omega}(y)|u(y)|\,dy \\
        &=\int_{B_{1}\setminus B_{3/4}}P^{e_{n}/2}_{\widetilde \Omega}(y)|u(y)|\,dy+ \int_{B_{1}^{c}}P^{e_{n}/2}_{\widetilde \Omega}(y)|u(y)|\,dy\lesssim \lVert u\rVert_{L^{\infty}(B_{1})}+(1-s)\lVert u\rVert_{L^{1}(\R^{n},w_{s})}.
    \end{align*}
    On the other hand, by \eqref{eq:basic-estimate-green-function-2}, the Green function of $\widetilde \Omega$ with pole $2e_{n}/3$ satisfies $G^{2e_{n}/3}_{\widetilde \Omega}(e_{n}/2)\approx 1$. Then, the boundary Harnack principle (Theorem \ref{thm:BHP}) gives
    \begin{equation*}
        |u(x)|\le v(x)\lesssim  \frac{v(e_{n}/2)}{G^{2e_{n}/3}_{\widetilde \Omega}(e_{n}/2)}G^{2e_{n}/3}_{\widetilde\Omega}(x)\lesssim \left(\lVert u\rVert_{L^{\infty}(B_{1})}+(1-s)\lVert u\rVert_{L^{1}(\R^{n},w_{s})}\right)G^{2e_{n}/3}_{\widetilde\Omega}(x)\qquad \forall x\in \Omega \cap B_{1/2},
    \end{equation*}
    and the desired conclusion follows immediately from Theorem \ref{thm:uniform-L2-Green-function}.  
\end{proof}

We now turn to Corollary \ref{cor:Sobolev-Dirichlet}. The argument relies on a localized application of the fractional Caccioppoli inequality (Lemma \ref{lem:Caccioppoli}), combined with the $L^{2}$-estimate for the non-tangential maximal function from Theorem \ref{thm:Dirichlet-solvability-intro}.
\begin{proof}[Proof of Corollary \ref{cor:Sobolev-Dirichlet}]
    Let us denote by $K:\Omega\to [0,\infty]$ the function
    \begin{equation*}
        K(x)\coloneqq\int_{\Omega}\frac{(u(y)-u(x))^{2}}{|y-x|^{n+2s}}\,dy\qquad \forall x\in \Omega. 
    \end{equation*}
    We divide the proof into two steps.  

    \smallskip
    \noindent \textbf{Step 1:} We first prove a weighted integral estimate for the function $K$:
    \begin{equation}\label{eq:weighted-integral-carré-Hbeta}
        \int_{\Omega}K(x)\delta(x)^{2s-\beta}\,dx\lesssim \lVert g\rVert_{L^{2}(\Omega^{c}, \sigma_{\beta})}^{2}.
    \end{equation}
    Given any $z\in \Omega$, rescaling the Caccioppoli inequality from Lemma \ref{lem:Caccioppoli}, we may bound
    \begin{align*}
            \fint_{B_{\delta(z)/8}(z)}K(x)\,dx
            \lesssim  \int_{\R^{n}}\frac{u^{2}(y)}{\delta(z)^{n+2s}+|y-z|^{n+2s}}\,dy. 
    \end{align*}
    On the other hand, an elementary geometric argument gives 
    \begin{align*}
        \int_{\Omega}\frac{\delta(z)^{2s-\beta}}{\delta(z)^{n+2s}+|y-z|^{n+2s}}\,dz\approx \frac{\delta(y)^{-\beta}}{1+\delta(y)^{n+2s-\beta}}\qquad \forall y\in \R^{n}\setminus \partial \Omega.
    \end{align*}
    Therefore, by Fubini's theorem, 
    \begin{align*}
        \int_{\Omega}K(x)\delta(x)^{2s-\beta}\,dx&\approx \int_{\Omega}\delta(z)^{2s-\beta}\fint_{B_{\delta(z)/8}(z)} K(x)\,dx\,dz\\
        &\lesssim \int_{\R^{n}}u(y)^{2}\int_{\Omega}\frac{\delta(z)^{2s-\beta}}{\delta(z)^{n+2s}+|y-z|^{n+2s}}\,dz\,dy\lesssim \int_{\Omega}u^{2}\delta^{-\beta}+\int_{\Omega^{c}}g^{2}d\sigma_{\beta}. 
    \end{align*}
    The inequality \eqref{eq:weighted-integral-carré-Hbeta} is then obtained after noting that  $\lVert u\rVert_{L^{2}(\Omega,\delta^{-\beta})}\lesssim\lVert u^{*}\rVert_{L^{2}(\Omega^{c},\sigma_{\beta})}\lesssim \lVert g\rVert_{L^{2}(\Omega^{c},\sigma_{\beta})}$, due to Remark \ref{rmk:exterior-interior-weighted} and Theorem \ref{thm:Dirichlet-problem-general}.
    
    \smallskip
    
    \noindent \textbf{Step 2:} Next, we use \eqref{eq:weighted-integral-carré-Hbeta} and an interpolation argument to conclude the bound \eqref{eq:bound-Hbetahalfs-dirichlet} for the $H^{\beta/2}(\Omega)$-norm of the solution. Let $S\subset \Omega\times \Omega$ be the set
    \begin{equation*}
        S\coloneqq \left\{(x,y)\in \Omega\times \Omega: |x-y|< \min \{\delta(x),\delta(y)\}\right\}.
    \end{equation*}
    Then, on the one hand
     \begin{equation}\label{eq:first-term-interpolation-Hbeta}
        \int\int_{S}\frac{(u(x)-u(y))^{2}}{|x-y|^{n+\beta}}\,dx\,dy\le \int\int_{S}\frac{(u(x)-u(y))^{2}}{|x-y|^{n+2s}}(\delta(x)^{2s-\beta}+\delta(y)^{2s-\beta})\,dx\,dy
        \le 2\int_{\Omega}K(x)\delta(x)^{2s-\beta}\,dx,
    \end{equation}
    where in the last step we used the symmetry of the double integral.
    
    On the other hand, since for any $(x,y)\in (\Omega\times \Omega)\setminus S$ we have $|x-y|\ge \max\{\delta(x), \delta(y)\}/2$, then
    \begin{equation}\label{eq:second-term-interpolation-Hbeta}
        \begin{aligned}
        \int\int_{(\Omega\times \Omega)\setminus S}\frac{(u(x)-u(y))^{2}}{|x-y|^{n+\beta}}\,dx\,dy &\le 2\int\int_{(\Omega\times \Omega)\setminus S}\frac{u(x)^{2}+u(y)^{2}}{|x-y|^{n+\beta}}\,dx\,dy\\
        &\le 4 \int_{\Omega}u(x)^{2}\int_{\Omega \setminus B_{\delta(x)/2}(x)}\frac{1}{|x-y|^{n+\beta}}\,dy\,dx\lesssim \int_{\Omega}u^{2}\delta^{-\beta}.
    \end{aligned}    
    \end{equation}
    Putting together  \eqref{eq:weighted-integral-carré-Hbeta}, \eqref{eq:first-term-interpolation-Hbeta}, \eqref{eq:second-term-interpolation-Hbeta}, and using again that $\lVert u\rVert_{L^{2}(\Omega,\delta^{-\beta})}\lesssim \lVert g\rVert_{L^{2}(\Omega^{c},\sigma_{\beta})}$, we finally obtain 
    \begin{equation*}
        [u]_{H^{\beta/2}(\Omega)}^{2}\lesssim \int_{\Omega}K(x)\delta(x)^{2s-\beta}\,dx+\int_{\Omega}u^{2}\delta^{-\beta}\lesssim \lVert g\rVert_{L^{2}(\Omega^{c},\sigma_{\beta})}^{2},
    \end{equation*}
    as desired.
\end{proof}

Finally, to prove Corollary \ref{cor:optimal-Sobolev-Poisson}, we use a duality argument to convert the Dirichlet estimates into Sobolev regularity for the inhomogeneous problem.

\begin{proof}[Proof of Corollary \ref{cor:optimal-Sobolev-Poisson}]
Let us first consider the case $\beta\ge 3s/2$.
We set $\gamma\coloneqq\beta-2s$ and $\theta\coloneqq -2\gamma=4s-2\beta$. Then $\gamma\in[-s/2,\min\{1/2-s,0\})$ and $\theta\in (\max\{2s-1,0\},s]$. Since $-\gamma\in (0,1/2)$, the fractional Hardy inequality from \cite[Theorem 1.1, case T2]{dyda2004fractional} implies that $H^{-\gamma}(\Omega^{c})$ embeds continuously in $L^{2}(\Omega^{c},\sigma_{\theta})$. As a consequence, by Theorem \ref{thm:Dirichlet-problem-general} and Corollary \ref{cor:Sobolev-Dirichlet}, for every $g\in H^{-\gamma}(\Omega^{c})$, the Poisson integral solution $w_{g}$ of the Dirichlet problem with exterior datum $g$ is well-defined, and the following estimate holds:
\begin{equation*}
    \lVert w_{g}\rVert_{H^{-\gamma}(\Omega)}\lesssim \lVert g\rVert_{L^{2}(\Omega^{c},\sigma_{\theta})}\lesssim \lVert g\rVert_{H^{-\gamma}(\Omega^{c})}.
\end{equation*}
Therefore, after noting that $\langle f, w_{g}\rangle_{\Omega}=-\langle (-\Delta)^{s}u,g\rangle_{\Omega^{c}}$, by duality we derive
\begin{equation*}
    \lVert (-\Delta)^{s}u\rVert_{H^{\gamma}(\Omega^{c})}\lesssim \lVert f\rVert_{H^{\gamma}(\Omega)}.
\end{equation*}
At this point, since $H^{\gamma}(\Omega)$ and $H^{\gamma}(\Omega^{c})$ embed continuously in $H^{\gamma}(\R^{n})$ through the zero-extension operator for $\gamma\in (-1/2, 0)$, classical Fourier analysis together with Poincaré inequality give
    \begin{align*}
        \lVert u\rVert_{H^{\beta}(\R^{n})}
        &\lesssim \lVert (-\Delta)^{s}u\rVert_{H^{\gamma}(\R^{n})}\\
        &\lesssim \lVert f\rVert_{H^{\gamma}(\Omega)}+\lVert (-\Delta)^{s}u\rVert_{H^{\gamma}(\Omega^{c})}
        \lesssim \lVert f\rVert_{H^{\gamma}(\Omega)},
    \end{align*}   
as desired. This concludes the proof in the case $\beta \in [3s/2,\min\{s+1/2,2s\})$.

Now, when $\beta=s$, the bound $\lVert u\rVert_{H^{s}(\R^{n})}\lesssim \lVert f\rVert_{H^{-s}(\Omega)}$ follows immediately from the variational structure of the problem. The conclusion for the intermediate exponents $\beta \in (s,3s/2)$ then follows from linear interpolation.
\end{proof}

\section{More general kernels}
\label{sec:general_kernels}

In this final section, we describe the modifications needed to extend the results proved in this paper to more general stable operators of order $2s$ that are comparable to the fractional Laplacian. We consider operators of the type 

\begin{equation*}
    \mathcal{L}_{\Theta}u(x)\coloneqq {\rm P.V.}\int_{\R^{n}}(u(x)-u(y))K_{\Theta}(x,y)\,dy,
    \qquad
    K_{\Theta}(x,y)\coloneqq\frac{c_{n,s}}{|x-y|^{n+2s}}\Theta\left(\frac{y-x}{|y-x|}\right),
\end{equation*}
where the anisotropy $\Theta$ is even and uniformly bounded above and below:
\begin{equation}\label{eq:ellipticity-LTheta}
    0<\lambda\le \Theta(\vartheta)\le \Lambda<\infty,
    \qquad
    \Theta(\vartheta)=\Theta(-\vartheta)\qquad \text{for a.e. $\vartheta\in \Sph^{n-1}$}.
\end{equation}
The evenness assumption in \eqref{eq:ellipticity-LTheta} makes $\mathcal{L}_{\Theta}$ self-adjoint. Thus $\mathcal{L}_{\Theta}$ is associated with the symmetric bilinear form
\begin{equation*}
    \mathcal{E}_{\Theta}(u,v)\coloneqq\frac{1}{2}\int_{\R^{n}}\int_{\R^{n}}(u(x)-u(y))(v(x)-v(y))K_{\Theta}(x,y)\,dx\,dy,
\end{equation*}
and, by \eqref{eq:ellipticity-LTheta}, this form is comparable to the usual $H^{s}(\R^{n})$ seminorm. We shall also use the pointwise carré du champ
\begin{equation*}
    B_{\Theta}(u,v)(x)\coloneqq\frac{1}{2}\int_{\R^{n}}(u(x)-u(y))(v(x)-v(y))K_{\Theta}(x,y)\,dy,
\end{equation*}
for which the product identity becomes
\begin{equation*}
    \mathcal{L}_{\Theta}(uv)=u\mathcal{L}_{\Theta}v+v\mathcal{L}_{\Theta}u-2B_{\Theta}(u,v)
\end{equation*}
whenever the quantities are well-defined. Given $f\in \mathscr{D}'(\Omega)$, we say that $u\in L^{1}(\R^{n},w_{s})$ is a distributional solution of $\mathcal{L}_{\Theta}u=f$ in $\Omega$ whenever
\begin{equation*}
    \int_{\R^{n}}u\mathcal{L}_{\Theta}\varphi=\langle f,\varphi\rangle\qquad \forall \varphi\in C^{\infty}_{c}(\Omega).
\end{equation*}
If $f\in H^{-s}(\Omega)$, we say that $u\in H^{s}_{\loc}(\Omega)\cap L^{1}(\R^{n},w_{s})$ is a weak solution whenever
\begin{equation*}
    \mathcal{E}_{\Theta}(u,\varphi)=\langle f,\varphi\rangle\qquad \forall \varphi\in C^{\infty}_{c}(\Omega).
\end{equation*}
The fundamental solution of $\mathcal{L}_{\Theta}$ in $\R^{n}$ will be denoted by $\Phi_{\Theta}$. It is homogeneous of degree $2s-n$ and satisfies
\begin{equation}\label{eq:potential-kernel-LTheta}
    c\frac{\kappa_{n,s}}{|x|^{n-2s}}\le \Phi_{\Theta}(x)\le C\frac{\kappa_{n,s}}{|x|^{n-2s}}\qquad \forall x\in \R^{n}\setminus\{0\},
\end{equation}
with constants $c,C>0$ depending only on $n,s,\lambda,\Lambda$; see \cite{bogdan2007estimates}. The Harnack inequality and the boundary regularity results in Lipschitz domains used in Section \ref{sec:prelim} extend to these kernels. In particular, Theorems \ref{thm:holder-boundary-lip-domain} and \ref{thm:BHP} hold, with $(-\Delta)^{s}$ replaced by $\mathcal{L}_{\Theta}$, for nonnegative $\mathcal{L}_{\Theta}$-harmonic functions vanishing in the exterior side of a Lipschitz patch, with constants depending also on $\lambda,\Lambda$; see \cite{bass2002harnack,bogdan2007estimates,bogdan2015boundary}.

\paragraph{Basic potential theory.}

For every $g\in C_{b}(\Omega^{c})$, the exterior Dirichlet problem with $\mathcal{L}_{\Theta}$ in place of $(-\Delta)^{s}$ admits a unique distributional solution $u_{g}\in C_{b}(\R^{n})$ (see, for instance, \cite[Chapter 3]{XROXFRbook}). For $x\in \Omega$, we define the $\mathcal{L}_{\Theta}$-harmonic measure $\omega_{\Theta,\Omega}^{x}\in \mathscr{P}(\R^{n})$, concentrated on $\Omega^{c}$, by
\begin{equation*}
    \int_{\Omega^{c}}g\,d\omega_{\Theta,\Omega}^{x}=u_{g}(x)\qquad \forall g\in C_{b}(\Omega^{c}).
\end{equation*}
The restriction of $\omega_{\Theta,\Omega}^{x}$ to $\interior \Omega^{c}$ is absolutely continuous with respect to $\mathscr{L}^{n}$. Its density, the Poisson kernel, is denoted by
\begin{equation*}
    P_{\Theta,\Omega}^{x}\coloneqq\frac{d\omega_{\Theta,\Omega}^{x}}{d\mathscr{L}^{n}}:\interior\Omega^{c}\to [0,\infty).
\end{equation*}
The Green function with pole $x\in \Omega$ is defined as
\begin{equation*}
    G_{\Theta,\Omega}^{x}(y)\coloneqq\Phi_{\Theta}(x-y)-R_{\Theta,\Omega}^{x}(y),
\end{equation*}
where $R_{\Theta,\Omega}^{x}$ is $\mathcal{L}_{\Theta}$-harmonic in $\Omega$ and equals $\Phi_{\Theta}(x-\cdot)$ in $\Omega^{c}$. Then, $G^{x}_{\Theta}(y)=G^{y}_{\Theta}(x)$ for all $x,y\in \Omega$, $x\neq y$, and $G_{\Theta,\Omega}^{x}$ solves $\mathcal{L}_{\Theta}G_{\Theta,\Omega}^{x}=\delta_{x}$ in $\Omega$ and $G_{\Theta,\Omega}^{x}=0$ in $\Omega^{c}$. The Poisson kernel and the Green function are related by the formula
\begin{equation*}
    P_{\Theta,\Omega}^{x}(y)=\int_{\Omega}G_{\Theta,\Omega}^{x}(z)K_{\Theta}(z,y)\,dz\qquad \forall y\in \interior\Omega^{c}.
\end{equation*}
Moreover, \eqref{eq:potential-kernel-LTheta} and the maximum principle yield the basic estimates
\begin{gather*}
    G_{\Theta,\Omega}^{x}(y)\le C\frac{\kappa_{n,s}}{|x-y|^{n-2s}}\qquad \forall x,y\in \Omega,\\
    G_{\Theta,\Omega}^{x}(y)\gtrsim \frac{\kappa_{n,s}}{|x-y|^{n-2s}}\qquad \forall x,y\in \Omega: |x-y|\le \frac{\delta_{\Omega}(x)}{2}.
\end{gather*}
In the sequel, when clear from the context, we drop the subscript $\Omega$. With these replacements, the proofs of Lemmas \ref{lem:approximation-smooth-domains}--\ref{lem:comparison-P-G} remain valid. In particular, for $x_{0}\in \Omega$, $\xi\in \partial \Omega$, and $0<r<r_{0}$ such that $|x_{0}-\xi|>2r$, one has the basic comparability estimates
\begin{gather*}
    \omega_{\Theta}^{x_{0}}(B_{r}(\xi))\approx r^{n-2s}G_{\Theta}^{x_{0}}(A_{r}(\xi)),\\[2pt]
    \omega_{\Theta}^{x_{0}}(B_{2r}(\xi))\lesssim \omega_{\Theta}^{x_{0}}(B_{r}(\xi)),\\[2pt]
    P_{\Theta}^{x_{0}}(A'_{r}(\xi))\approx (1-s)r^{-2s}G_{\Theta}^{x_{0}}(A_{r}(\xi)),
\end{gather*}
as well as the change of pole formulas, for $\xi'\in \partial \Omega$, $0<r'<cr$, $|\xi'-\xi|<r-r'$:
\begin{gather*}
    \omega_{\Theta}^{A_{r}(\xi)}(B_{r'}(\xi'))\approx \frac{\omega_{\Theta}^{x_{0}}(B_{r'}(\xi'))}{\omega_{\Theta}^{x_{0}}(B_{r}(\xi))},\\
    \lim_{\Omega\ni z\to \xi'}\frac{G_{\Theta}^{x_{0}}(z)}{G_{\Theta}^{A_{r}(\xi)}(z)}\approx r^{n-2s}G_{\Theta}^{x_{0}}(A_{r}(\xi)).
\end{gather*}
All constants depend on the same geometric quantities as in Section \ref{sec:prelim}, and also on $\lambda,\Lambda$.

\paragraph{Pohozaev identity and reverse-H\"{o}lder estimates.}
The only point in Section \ref{sec:Dahlberg} that requires an adaptation is the Pohozaev identity. For $e\in \Sph^{n-1}$, set
\begin{equation*}
    \mathcal{A}_{\Theta}(e)\coloneqq\rho_{n,s}\int_{\Sph^{n-1}}|e\cdot \vartheta|^{2s}\Theta(\vartheta)\,d\mathcal{H}^{n-1}(\vartheta),\qquad \rho_{n,s}=\frac{\Gamma(\frac{n+2s}{2})}{2\pi^{\frac{n-1}{2}}\Gamma(s+\frac{1}{2})}.
\end{equation*}
The normalizing constant is chosen so that $\mathcal{A}_{\Theta}(e)=1$ for $\Theta\equiv 1$, while in general \eqref{eq:ellipticity-LTheta} gives $0<\lambda\le \mathcal{A}_{\Theta}(e)\le \Lambda<\infty$ uniformly in $e\in \Sph^{n-1}$. The anisotropic Pohozaev identities and integration by parts formulas of \cite{rosoton2017pohozaev} give the following replacement for Lemma \ref{lem:Pohozaev}: if $\Omega$ is smooth and $\nu$ denotes the outer unit normal to $\partial\Omega$, then, for every $x_{0}\in \Omega$,
\begin{equation*}
    \int_{\partial\Omega}\mathcal{A}_{\Theta}(\nu)(\partial_{\nu}^{s}G_{\Theta}^{x_{0}}(\xi))^{2}(\xi-x_{0})\cdot \nu\,d\mathcal{H}^{n-1}(\xi)
    =\frac{n-2s}{\Gamma(1+s)^{2}}\int_{\Omega^{c}}P_{\Theta}^{x_{0}}(y)\Phi_{\Theta}(y-x_{0})\,dy.
\end{equation*}
Thus, compared with \eqref{eq:Pohozaev}, the boundary integral acquires the angular factor $\mathcal{A}_{\Theta}(\nu)$ and the right-hand side contains the anisotropic fundamental solution $\Phi_{\Theta}$. The proof of Proposition \ref{prop:L2-normal-a-priori} then goes through with $G^{x_{0}},P^{x_{0}}$ replaced by $G_{\Theta}^{x_{0}},P_{\Theta}^{x_{0}}$, since $\mathcal{A}_{\Theta}$ is bounded above and below and $\Phi_{\Theta}$ satisfies \eqref{eq:potential-kernel-LTheta}. We then derive, for $\xi\in \partial \Omega$ and $0<r<r_{0}$ such that $|\xi-x_{0}|>2r$:
\begin{equation*}
    \left(\fint_{B_{r}(\xi)\cap \partial\Omega}(\partial_{\nu}^{s}G_{\Theta}^{x_{0}})^{2}d\mathcal{H}^{n-1}\right)^{1/2}\approx \frac{G_{\Theta}^{x_{0}}(A_{r}(\xi))}{r^{s}}.
\end{equation*}
Consequently, Theorem \ref{thm:uniform-L2-Green-function} and Corollary \ref{cor:uniform-L2-Poisson-kernel} also remain valid, and for $t\in (0,c_{0}r)$ we get
\begin{gather*}
    \frac{1}{t^{s}}\left(\fint_{B_{r}(\xi)\cap \Omega\cap \{\delta=t\}}(G_{\Theta}^{x_{0}})^{2}d\mathcal{H}^{n-1}\right)^{1/2}\approx \frac{G_{\Theta}^{x_{0}}(A_{r}(\xi))}{r^{s}},\\
    \frac{t^{s}}{1-s}\left(\fint_{B_{r}(\xi)\cap \Omega^{c}\cap \{\delta=t\}}(P_{\Theta}^{x_{0}})^{2}d\mathcal{H}^{n-1}\right)^{1/2}\approx \frac{\omega_{\Theta}^{x_{0}}(B_{r}(\xi))}{r^{n-s}}.
\end{gather*}
The geometric arguments used in the proof of Theorem \ref{thm:Dahlberg-general} are unchanged. We therefore obtain the following analog of Theorem \ref{thm:Dahlberg-general}:

\begin{thm}\label{thm:Dahlberg-general-LTheta}
    Let $\Omega\subset \R^{n}$ be a bounded Lipschitz domain, and let $x_{0}\in \Omega$. Given $\beta\in (2s-1,s]$, let $\sigma_{\beta}$ be the measure in \eqref{eq:def-sigma-beta}. There is $p_{0}>2$ depending only on $n,s,\beta$, the Lipschitz character of $\Omega$, and $\lambda,\Lambda$ such that, for all $p\in [2,p_{0})$, and all balls $B\subset \R^{n}$ centered on $\partial \Omega$, we have
    \begin{equation}\label{eq:reverse-holder-general-LTheta}
        \left(\fint_{B}\left(\frac{d\omega_{\Theta}^{x_{0}}}{d\sigma_{\beta}}\right)^{p}d\sigma_{\beta}\right)^{1/p}\le C\frac{\omega_{\Theta}^{x_{0}}(B)}{\sigma_{\beta}(B)},
    \end{equation}
    where $C$ depends only on $n,s,\beta,p$, the Lipschitz character of $\Omega$, $\delta(x_{0})$, and $\lambda,\Lambda$.
\end{thm}

\paragraph{Martin kernels, Poisson integrals, and the Dirichlet problem.}
The Martin-kernel formalism from Section \ref{sec:Dir_pb} is unchanged. Fix $x_{0}\in \Omega$. We define the Martin kernel associated with $\mathcal{L}_{\Theta}$ by
\begin{equation*}
    K_{\Theta}^{x_{0}}(x,y)\coloneqq\lim_{r\to 0^{+}}\frac{\omega_{\Theta}^{x}(B_{r}(y))}{\omega_{\Theta}^{x_{0}}(B_{r}(y))}=\begin{cases}
        \displaystyle\frac{P_{\Theta}^{x}(y)}{P_{\Theta}^{x_{0}}(y)}\qquad &\text{if $y\in \interior \Omega^{c}$},\\[1.2em]
        \displaystyle\underset{\Omega\ni z\to y}{\lim}\frac{G_{\Theta}^{x}(z)}{G_{\Theta}^{x_{0}}(z)}\qquad &\text{if $y\in \partial\Omega$}\end{cases}\qquad x\in \Omega,\quad y\in \Omega^{c}.
\end{equation*}
The existence of the limit follows from the boundary Harnack principle and the Martin-boundary theory for stable processes in Lipschitz domains; see \cite{chen1998martin,bogdan2015boundary}.

For every finite measure $\nu\in \mathcal{M}(\Omega^{c})$, we define the corresponding Poisson integral by
\begin{equation*}
    H_{\Theta}^{x_{0}}\nu(x)\coloneqq\int_{\Omega^{c}}K_{\Theta}^{x_{0}}(x,y)d\nu(y)\qquad \forall x\in \Omega.
\end{equation*}
the proofs of Lemmas \ref{lem:non-tangential-maximal-function-bound} and \ref{lem:general-uniqueness} remain valid: for every finite measure $\nu\in \mathcal{M}(\Omega^{c})$,
\begin{equation}\label{eq:maximal-function-bound-LTheta}
    (H_{\Theta}^{x_{0}}\nu)^{*}(y)\lesssim M_{\omega_{\Theta}^{x_{0}}}^{\Omega}\nu(y)\qquad \forall y\in \Omega^{c},
\end{equation}
with equivalence when $\nu\ge 0$ and $\Gamma(y)\neq \varnothing$; if $u\in L^{1}(\R^{n},w_{s})$ is a distributional solution of $\mathcal{L}_{\Theta}u=0$ in $\Omega$, $u=0$ in $\Omega^{c}$, and $u^{*}\in L^{1}(\Omega^{c},\omega_{\Theta}^{x_{0}})$ for some $x_{0}\in \Omega$, then $u\equiv 0$. 

Combining \eqref{eq:reverse-holder-general-LTheta}, \eqref{eq:maximal-function-bound-LTheta}, and the uniqueness criterion gives the analog of Theorem \ref{thm:Dirichlet-problem-general}:
\begin{thm}\label{thm:Dirichlet-problem-general-LTheta}
    Let $\Omega\subset \R^{n}$ be a bounded Lipschitz domain. Given $\beta\in (2s-1,s]$, let $\sigma_{\beta}$ be the measure in \eqref{eq:def-sigma-beta}. Then, there exists $q_{0}\in [1,2)$ depending only on $n,s,\beta$, the Lipschitz character of $\Omega$, and $\lambda,\Lambda$, such that the following holds. For every $q>q_{0}$ and every $g\in L^{q}(\Omega^{c},\sigma_{\beta})$, the Poisson integral
    \begin{equation}\label{eq:Poisson-solution-statement-Lp-solvability-LTheta}
        u(x)\coloneqq\int_{\Omega^{c}}P_{\Theta}^{x}(y)g(y)\,dy
    \end{equation}
     is well-defined, and once extended by $u=g$ on $\Omega^c$, is the unique distributional solution of
    \begin{equation*}
    \left\{
    \begin{array}{rclll}
         \mathcal{L}_{\Theta}u&=&0\quad &\text{in $\Omega$},\\
            u&=&g\quad &\text{in $\Omega^{c}$},\\
            \lVert u^{*}\rVert_{L^{q}(\Omega^{c},\sigma_{\beta})}&<&\infty.
    \end{array}
    \right.
    \end{equation*}
    Moreover, it holds
    \begin{equation*}
        \lVert u^{*}\rVert_{L^{q}(\Omega^{c},\sigma_{\beta})}\le C\lVert g\rVert_{L^{q}(\Omega^{c},\sigma_{\beta})},
    \end{equation*}
    where the constant $C>0$ depends only on $n,s,\beta,q$, the Lipschitz character of $\Omega$, and $\lambda,\Lambda$.
\end{thm}

\paragraph{Corollaries of the theory.}
The corollaries stated in Section \ref{subsec:applications} also extend verbatim to the anisotropic setting. First, the square estimates above for $G^{x_{0}}_{\Theta}$ on parallel level sets together with the boundary Harnack principle for $\mathcal{L}_{\Theta}$ \cite{bogdan2015boundary, XROXFRbook} give: 

\begin{cor}\label{cor:Linfty-L2-parallel-sets-LTheta}
    Let $\Omega=\{x=(x',x_{n}):x_{n}>\phi(x')\}$ be a Lipschitz epigraph, where $\phi:\R^{n-1}\to \R$ is $L$-Lipschitz and $\phi(0)=0$. Let $u\in C(B_{1})\cap L^{1}(\R^{n},w_{s})$ be a solution of 
    \begin{equation*}
    \left\{
    \begin{array}{rclll}
         \mathcal{L}_{\Theta}u&=&0\quad &\text{in $B_{1}\cap \Omega$},\\
            u&=&0\quad &\text{in $B_{1}\setminus \Omega$}.
    \end{array}
    \right.
    \end{equation*}
    Then, there is $c_{0}\in (0,1)$ depending only on $L$ such that 
     \begin{equation*}
        \left(\fint_{B_{1/2}\cap \Omega\cap \{\delta=t\}}\left(\frac{u}{\delta^{s}}\right)^{2}\,d\mathcal{H}^{n-1}\right)^{1/2}\le C\left(\lVert u\rVert_{L^{\infty}(B_{1})}+(1-s)\lVert u\rVert_{L^{1}(\R^{n},w_{s})}\right)\qquad \forall t\in (0,c_{0}),
    \end{equation*}
    where $C$ depends only on $n,s$, $L$, $\lambda, \Lambda$, and is uniform as $s\to 1^{-}$.
\end{cor}

Next, the Caccioppoli inequality from Lemma \ref{lem:Caccioppoli} is replaced by its energy version for $\mathcal{E}_{\Theta}$. Namely, if $u\in H^{s}_{\loc}(B_{1})\cap L^{1}(\R^{n},w_{s})$ is a weak solution of $\mathcal{L}_{\Theta}u=0$ in $B_{1}$, then
\begin{equation*}
    \int_{B_{1/2}}\int_{B_{1/2}}(u(x)-u(y))^{2}K_{\Theta}(x,y)\,dx\,dy
    \le C\left(\lVert u\rVert_{L^{2}(B_{1})}^{2}+(1-s)\lVert u\rVert_{L^{1}(B_{1})}\lVert u\rVert_{L^{1}(\R^{n},w_{s})}\right),
\end{equation*}
where $C$ depends only on $n,s,\lambda,\Lambda$.
Consequently, the proof of Corollary \ref{cor:Sobolev-Dirichlet} remains valid and we obtain the following:
\begin{cor}\label{cor:Sobolev-Dirichlet-LTheta}
    Let $\Omega\subset \R^{n}$ be a bounded Lipschitz domain, and let $\beta\in (\max\{2s-1, 0\},s]$. Given $g\in L^{2}(\Omega^{c},\sigma_{\beta})$, let $u$ be the Poisson integral solution defined in \eqref{eq:Poisson-solution-statement-Lp-solvability-LTheta} of the Dirichlet problem $\mathcal{L}_{\Theta}(u)=0$ in $\Omega$ with exterior datum $g$. Then, $u\in H^{\beta/2}(\Omega)$ and 
    \begin{equation*}
        \lVert u\rVert_{H^{\beta/2}(\Omega)}\le C\lVert g\rVert_{L^{2}(\Omega^{c},\sigma_{\beta})},
    \end{equation*}
    where $C$ depends only on $n,s,\beta$, the Lipschitz character of $\Omega$, and $\lambda,\Lambda$. 
\end{cor}

Finally, the Fourier symbol of $\mathcal{L}_{\Theta}$ is given by
\begin{equation*}
    m_{\Theta}(\xi)= |\xi|^{2s}\mathcal{A}_{\Theta}(\xi/|\xi|)\qquad \forall \xi \in \R^{n}\setminus \{0\}.
\end{equation*}
Therefore, $\lambda|\xi|^{2s}\le m_{\Theta}(\xi)\le \Lambda |\xi|^{2s}$, and the duality and Fourier-analysis arguments in the proof of Corollary \ref{cor:optimal-Sobolev-Poisson} are unchanged after replacing $(-\Delta)^{s}$ by $\mathcal{L}_{\Theta}$; see also \cite{rosoton2016regularity,rosoton2016dirichlet}. We then obtain the following analogue of Corollary \ref{cor:optimal-Sobolev-Poisson}:
\begin{cor}\label{cor:optimal-Sobolev-Poisson-LTheta}
    Let $\Omega\subset \R^{n}$ be a bounded Lipschitz domain, and let $\beta\in [s,\min\{s+1/2,2s\})$. Given $f\in H^{\beta-2s}(\Omega)$, let $u\in H^{s}(\R^{n})$ be the weak solution of the Poisson problem
   \begin{equation*}
       \left\{
    \begin{array}{rclll}
         \mathcal{L}_{\Theta}u&=&f\quad &\text{in $\Omega$},\\
            u&=&0\quad &\text{in $\Omega^{c}$}.\\
    \end{array}
    \right.
   \end{equation*}
   Then, $u\in H^{\beta}(\R^{n})$ and
   \begin{equation*}
       \lVert u\rVert_{H^{\beta}(\R^{n})}\le C\lVert f\rVert_{H^{\beta-2s}(\Omega)},
   \end{equation*}
   where $C$ depends only on $n,s,\beta$, the Lipschitz character of $\Omega$, and $\lambda,\Lambda$.
\end{cor}

\addcontentsline{toc}{section}{Appendix}

\appendix
\section*{Appendix}

The appendix is divided into two parts. The first collects two lemmas on $s$-harmonic functions, while the second contains some harmonic analysis results for doubling weights.
\refstepcounter{section}

\subsection[Some useful lemmas on s-harmonic functions]{Some useful lemmas on $s$-harmonic functions}

\paragraph{A bound on tails.}
The following result, also known as the ``half-Harnack inequality'', gives a bound on the nonlocal tail of a nonnegative $s$-harmonic function in terms of its interior values. For a proof, we refer the reader to \cite[Theorem 3.3.1]{XROXFRbook}.
\begin{lem}\label{lem:half-harnack-tails}
    Let $u\in C(B_{1})\cap L^{1}(\R^{n},w_{s})$ be nonnegative in $\R^{n}$ and $s$-harmonic in $B_{1}$. Then, the following estimate holds:
    \begin{equation*}
       (1-s)\lVert u\rVert_{L^{1}(\R^{n},w_{s})}\le  C\inf_{B_{1/2}}u,
    \end{equation*}
    where $C$ depends only on $n,s$, and is uniform as $s\to 1^{-}$.
\end{lem}

\paragraph{Caccioppoli inequality.}
Next, we prove a nonlocal version of the well-known Caccioppoli inequality for weak solutions.

\begin{lem}\label{lem:Caccioppoli}
    Let $u\in H^{s}_{\loc}(B_{1})\cap L^{1}(\R^{n}, w_{s})$ be a weak solution of $(-\Delta)^{s}u=0$ in $B_{1}$. 
    Then, we have
    \begin{equation*}
        [u]_{H^{s}(B_{1/2})}^{2}\le C\left(\lVert u\rVert_{L^{2}(B_{1})}^{2}+ (1-s)\lVert u\rVert_{L^{1}(B_{1})}\lVert u\rVert_{L^{1}(\R^{n},w_{s})}\right),
    \end{equation*}
    where $C$ depends only on $n,s$, and is uniform as $s\to 1^{-}$.
\end{lem}

\begin{proof}
    Let $\psi\in C^{\infty}_{c}(B_{1})$ be a standard cutoff function such that $\psi\in [0,1]$, $\psi\equiv 1$ in $B_{3/4}$, and $|\nabla \psi|\lesssim 1$. We consider the function $v\coloneqq u\psi$, which vanishes outside of $B_{1}$, and is a weak solution of
    \begin{equation}\label{eq:equation-v-caccioppoli-localization}
        (-\Delta)^{s}v=f\quad \text{in $B_{1}$},\qquad f\coloneqq u(-\Delta)^{s}\psi-2B_{s}(u,\psi).
    \end{equation}
    Note that, by the definition of $\psi$, for every $x\in B_{2/3}$ we have 
    \begin{equation*}
        |B_{s}(u,\psi)(x)|\le \frac{c_{n,s}}{2}\int_{B_{3/4}^{c}}\frac{|u(x)-u(y)||\psi(x)-\psi(y)|}{|x-y|^{n+2s}}\,dy \lesssim (1-s)\left(|u(x)|+\int_{\R^{n}}\frac{|u(y)|}{1+|y|^{n+2s}}\,dy\right).
    \end{equation*}
    Therefore, the right-hand side $f$ can be bounded pointwise in $B_{2/3}$ by
    \begin{equation}\label{eq:unif-bound-rhs-caccioppoli-localization}
        |f(x)|\lesssim (1-s)\left(|u(x)|+\lVert u\rVert_{L^{1}(\R^{n},w_{s})}\right)\qquad \forall x\in B_{2/3}.
    \end{equation}
    
    Now we consider another standard cutoff function $\varphi \in C^{\infty}_{c}(B_{2/3})$ such that $\varphi\in [0,1]$, $\varphi\equiv1$ in $B_{1/2}$, and $|\nabla \varphi|\lesssim 1$. Testing equation \eqref{eq:equation-v-caccioppoli-localization} with $v\varphi^{2}$, and using the identity
    \[v(x)\varphi^{2}(x)-v(y)\varphi^{2}(y)=\varphi^{2}(x)(v(x)-v(y))+v(y)(\varphi(x)-\varphi(y))(\varphi(x)+\varphi(y)).\]
    we obtain
    \begin{align}\label{eq:equation-after-testing-caccioppoli}
        \int_{\R^{n}}\varphi^{2}B_{s}(v,v)=-\frac{c_{n,s}}{2}\int_{\R^{n}}\int_{\R^{n}}\frac{v(y)(v(x)-v(y))(\varphi(x)-\varphi(y))(\varphi(x)+\varphi(y))}{|x-y|^{n+2s}}\,dx\,dy+\int_{B_{1}}fv\varphi^{2}=:I+II.
    \end{align}
    By Young's inequality, the first term in the right-hand side of \eqref{eq:equation-after-testing-caccioppoli} can be controlled by
    \begin{equation}\label{eq:bound-I-caccioppoli}
        \begin{aligned}
        I&\le \frac{c_{n,s}}{8}\int_{\R^{n}}\int_{\R^{n}}\frac{(v(x)-v(y))^{2}(\varphi^{2}(x)+\varphi^{2}(y))}{|x-y|^{n+2s}}\,dx\,dy+\frac{c_{n,s}}{2}\int_{\R^{n}}\int_{\R^{n}}\frac{v^{2}(y)(\varphi(x)-\varphi(y))^{2}}{|x-y|^{n+2s}}\,dx\,dy\\
        &\le \frac{1}{2}\int_{\R^{n}}\varphi^{2}B_{s}(v,v)+\int_{\R^{n}}v^{2}B_{s}(\varphi, \varphi)\le \frac{1}{2}\int_{\R^{n}}\varphi^{2}B_{s}(v,v)+C\lVert u\rVert_{L^{2}(B_{1})}^{2},
        \end{aligned}
    \end{equation}
    where in the second step we used the symmetry of the first double integral. On the other hand, thanks to the pointwise bound in \eqref{eq:unif-bound-rhs-caccioppoli-localization} and the fact that $\varphi$ is supported in $B_{2/3}$, the second term in the right-hand side of \eqref{eq:equation-after-testing-caccioppoli} can be bounded by
    \begin{equation}\label{eq:bound-II-caccioppoli}
        II\lesssim (1-s)\int_{B_{2/3}}|v|(|u|+\lVert u\rVert_{L^{1}(\R^{n},w_{s})})\le (1-s)\left(\lVert u\rVert_{L^{2}(B_{1})}^{2}+ \lVert u\rVert_{L^{1}(B_{1})}\lVert u\rVert_{L^{1}(\R^{n},w_{s})}\right).
    \end{equation}
    Finally, substituting \eqref{eq:bound-I-caccioppoli} and \eqref{eq:bound-II-caccioppoli} into \eqref{eq:equation-after-testing-caccioppoli} and rearranging terms we obtain
    \begin{align*}
        [u]_{H^{s}(B_{1/2})}^{2}\le \int_{\R^{n}}\varphi^{2}B_{s}(v,v)\lesssim \lVert u\rVert_{L^{2}(B_{1})}^{2}+ (1-s)\lVert u\rVert_{L^{1}(B_{1})}\lVert u\rVert_{L^{1}(\R^{n},w_{s})}, 
    \end{align*}
    as desired.
\end{proof}

\subsection{Harmonic analysis tools}

\paragraph{A Gehring-type lemma.} In the following lemma, we prove the self-improvement of the reverse-H\"{o}lder inequality for a certain class of weights that satisfy a specific anisotropic doubling condition. The dyadic framework and the main strategy of proof follow the ideas presented in \cite{AndersonWeirich2018}. 

We call $\mathcal{D}$ the set of all half-open dyadic subcubes of $(0,1]^{n-1}$. For every $Q\in \mathcal{D}$, we denote by $\ell(Q)$ the side length of $Q$. A cube $\hat{Q}\in \mathcal{D}$ is the parent of $Q$ if $Q\subset \hat{Q}$ and $\ell(\hat{Q})=2\ell(Q)$. We say that $Q'\in \mathcal{D}$ is a descendant of $Q\in \mathcal{D}$ if $Q'\subseteq Q$ and denote the set of all descendants of $Q$ by
\begin{equation*}
    \mathcal{D}(Q)\coloneqq\{Q'\in \mathcal{D}:Q'\subseteq Q\}.
\end{equation*}
The tent over $Q$ is defined as $T(Q)\coloneqq Q\times (0,\ell(Q)]\subseteq (0,1]^{n}$. The upper half of the tent is denoted by $U(Q)\coloneqq Q\times (\ell(Q)/2, \ell(Q)]$; see Figure \ref{fig:dyadic-gehring}.

\begin{figure}[H]
\centering
\includegraphics[width=0.5\textwidth]{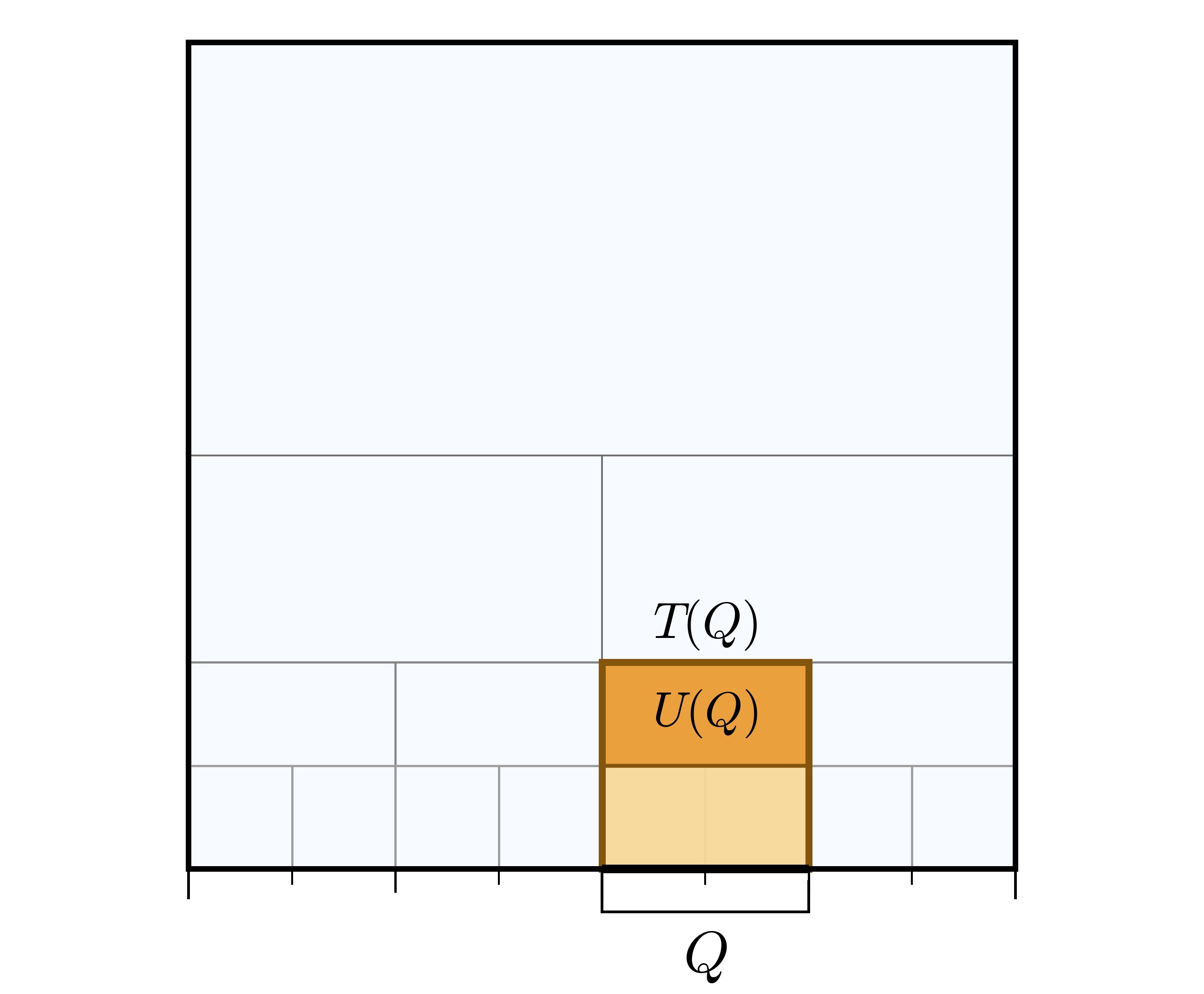}
\caption{The dyadic framework.}
\label{fig:dyadic-gehring}
\end{figure}

\begin{lem}\label{lem:gehring}
    Let $\sigma \in \mathcal{M}_{+}((0,1]^{n})$ be a finite nonnegative measure such that 
    \begin{equation*}
        0<\sigma(T(\hat{Q}))\le \bar{C} \sigma(T(Q))\qquad \forall Q\in \mathcal{D}, \quad \ell(Q)\le 1/2.
    \end{equation*}
    Given $f:(0,1]^{n}\to [0,\infty)$, 
    suppose that there exist $p\in (1,\infty)$ and $C_{0}, C_{1}\ge 1$ such that the following conditions hold:
    \begin{itemize}
        \item [(i)] $L^{p}$-reverse-H\"{o}lder inequality in tents: 
        \begin{equation*}
            \left(\fint_{T(Q)}f^{p}d\sigma\right)^{1/p}\le C_{0}\fint_{T(Q)}fd\sigma\qquad \forall Q\in \mathcal{D}.
        \end{equation*}
        \item [(ii)] Pointwise bound in the upper half of tents:
        \begin{equation*}
            \frac{1}{C_{1}}\fint_{T(Q)}fd\sigma\le f(x)\le C_{1}\fint_{T(Q)}fd\sigma\qquad \forall Q\in \mathcal{D},\quad\forall x\in U(Q).
        \end{equation*}
    \end{itemize}
    Then, there are $\eps>0$ and $C_{2}>1$ depending only on $n, \bar C, p, C_{0}$, and $C_{1}$, such that, for every $Q\in \mathcal{D}$, the $L^{p+\eps}$-reverse-H\"{o}lder inequality holds:
    \begin{equation}\label{eq:reverse-holder-ppluseps-gehring}
        \left(\fint_{T(Q)}f^{p+\eps}d\sigma\right)^{1/(p+\eps)}\le C_{2}\fint_{T(Q)}fd\sigma\qquad \forall Q\in \mathcal{D}.
    \end{equation}
\end{lem}

\begin{proof}
    In the sequel, we use the notation $\langle g\rangle _{Q}$ to denote the average $\fint_{T(Q)}gd\sigma$. 
    Let $\lambda>1$ be some number that will be chosen sufficiently large later. We call $\mathcal{J}(Q)\subset \mathcal{D}(Q)$ the maximal disjoint family of descendants $Q'\in \mathcal{D}(Q)$ such that
    \begin{equation}\label{eq:admissible-property-gehring}
        \langle f\rangle_{Q'}\notin \left(\lambda^{-1}\langle f\rangle _{Q}, \lambda \langle f\rangle _{Q}\right).
    \end{equation}
    More precisely, we let $Q'\in \mathcal{J}(Q)$ if and only if \eqref{eq:admissible-property-gehring} holds for $Q'$, but it does not hold for any $Q''\in \mathcal{D}(Q)$ such that $Q'\subsetneq Q''$. Notice in particular that $Q\notin \mathcal{J}(Q)$ and that any two elements of $\mathcal{J}(Q)$ are disjoint.
    We may split $\mathcal{J}(Q)$ as the disjoint union of the two families $\mathcal{J}^{+}(Q)$ and $\mathcal{J}^{-}(Q)$, where
    \begin{equation*}
        \mathcal{J}^{+}(Q)\coloneqq\{Q'\in \mathcal{J}(Q): \langle f\rangle_{Q'}\ge \lambda\langle f\rangle_{Q}\},\qquad \mathcal{J}^{-}(Q)\coloneqq\{Q'\in \mathcal{J}(Q): \langle f\rangle_{Q'}\le \lambda^{-1}\langle f\rangle_{Q}\}.
    \end{equation*}
    
    We also introduce the sets 
    \begin{gather*}
        \mathcal{B}^{+}(Q)\coloneqq\bigcup_{Q'\in \mathcal{J}^{+}(Q)}T(Q'),\qquad \mathcal{B}^{-}(Q)\coloneqq\bigcup_{Q'\in \mathcal{J}^{-}(Q)}T(Q'),\qquad \mathcal{B}(Q)\coloneqq \mathcal{B}^{+}(Q)\cup \mathcal{B}^{-}(Q),\\[3pt] G(Q)\coloneqq T(Q)\setminus \mathcal{B}(Q).
    \end{gather*}
    Note that $T(Q)=\mathcal{B}^{+}(Q)\cup \mathcal{B}^{-}(Q)\cup G(Q)$, and that the three sets are disjoint. 
    In the sequel, we will make extensive use of the following two facts:
    \begin{itemize}
        \item [(a)] For every $Q'\in \mathcal{J}(Q)$, we have $\langle f\rangle_{Q'}\le \bar C\lambda \langle f\rangle_{Q}$. In fact, calling $\hat{Q'}$ the parent of $Q'$, by the maximality of $\mathcal{J}(Q)$ we know that $\hat{Q'}$ does not satisfy \eqref{eq:admissible-property-gehring}. Therefore
        \begin{equation*}
            \langle f\rangle_{Q'}\le \frac{\sigma(T(\hat{Q'}))}{\sigma(T(Q'))}\langle f\rangle_{\hat{Q'}}\le \bar C\lambda \langle f \rangle_{Q}.
        \end{equation*}
        \item [(b)] We have 
        \begin{equation*}
            \frac{1}{C_{1}\lambda}\langle f\rangle_{Q}\le f(x)\le C_{1}\lambda \langle f\rangle_{Q}\qquad \forall x\in G(Q).
        \end{equation*}
        Indeed, if $x\in G(Q)$, then there is a unique $Q'\in \mathcal{D}(Q)$ such that $x\in U(Q')$. Moreover, $Q'$ does not satisfy \eqref{eq:admissible-property-gehring}, for otherwise we would have $x\in T(Q'')\subseteq \mathcal{B}(Q)$ for some $Q'\subseteq Q''\in \mathcal{J}(Q)$. The claim then follows from assumption $(ii)$. 
    \end{itemize}
    The rest of the proof is divided into three steps.
    
    \smallskip
    \noindent \textbf{Step 1:} In this step we fix $\lambda=2\max\{(3\bar C C_{0}^{p})^{\frac{1}{p-1}},3\}$ and we prove that 
    \begin{equation}\label{eq:gehring-outcome-step1}
        \sigma(\mathcal{B}(Q))\le c\sigma(T(Q))\qquad\forall Q\in \mathcal{D},
    \end{equation}
    where $c\coloneqq 1-(3C_{1}\lambda)^{-1}\in (0,1)$.
    Suppose by contradiction that there is some $Q\in \mathcal{D}$ for which $\sigma(\mathcal{B}(Q))> c\sigma(T(Q))$. Then, 
    \begin{equation*}
        \sigma(G(Q))< (1-c)\sigma(T(Q))=\frac{\sigma(T(Q))}{3C_{1}\lambda}.
    \end{equation*}
    This, together with point (b) above, implies
    \begin{equation}\label{eq:gehring-bound-integral-G}
        \int_{G(Q)}fd\sigma\le C_{1}\lambda\langle f\rangle_{Q}\sigma(G(Q))<\frac{1}{3}\int_{T(Q)}fd\sigma.
    \end{equation}
    On the other hand, since $\lambda>3$, we have
    \begin{equation}\label{eq:gehring-bound-integral-B-}
        \int_{\mathcal{B}^{-}(Q)}fd\sigma =\sum_{Q'\in \mathcal{J}^{-}(Q)}\langle f\rangle_{Q'}\sigma(T(Q'))\le \lambda^{-1}\langle f\rangle_{Q}\sigma(\mathcal{B}^{-}(Q))<\frac{1}{3}\int_{T(Q)}fd\sigma.
    \end{equation}
    Since $G(Q), \mathcal{B}^{+}(Q)$, and $\mathcal{B}^{-}(Q)$ are disjoint, from \eqref{eq:gehring-bound-integral-G} and \eqref{eq:gehring-bound-integral-B-} we deduce  
    \begin{equation}\label{eq:gehring-bound-integral-B+}
        \int_{\mathcal{B}^{+}(Q)}fd\sigma >\frac{1}{3}\int_{T(Q)}fd\sigma.
    \end{equation}
    By property (a) above, we also find
    \begin{equation}\label{eq:gehring-bound-fintegral-B+}
        \fint_{\mathcal{B}^{+}(Q)}fd\sigma= \frac{1}{\sigma(\mathcal{B}^{+}(Q))}\sum_{Q'\in \mathcal{J}^{+}(Q)}\sigma(T(Q'))\langle f\rangle_{Q'}\le \bar C\lambda \langle f\rangle_{Q}.
    \end{equation}
    As a consequence of \eqref{eq:gehring-bound-integral-B+} and \eqref{eq:gehring-bound-fintegral-B+} we deduce
    \begin{equation}\label{eq:gehring-bound-below-sigma-B+}
        \sigma(\mathcal{B}^{+}(Q))\ge \frac{1}{\bar C\lambda \langle f\rangle_{Q}}\int_{\mathcal{B}^{+}(Q)}fd\sigma>\frac{1}{3\bar C\lambda \langle f\rangle_{Q}}\int_{T(Q)}fd\sigma=\frac{\sigma(T(Q))}{3\bar C\lambda}.
    \end{equation}
    The contradiction now comes from combining the reverse-H\"{o}lder inequality in assumption $(i)$ with Jensen's inequality, the definition of $\mathcal{B}^{+}(Q)$, the lower bound in \eqref{eq:gehring-bound-below-sigma-B+}, and the choice we made of $\lambda$:
    \begin{align*}
        C_{0}^{p}\langle f\rangle_{Q}^{p}&\ge \langle f^{p}\rangle _{Q}\\
        &\ge \frac{1}{\sigma(T(Q))}\sum_{Q'\in \mathcal{J}^{+}(Q)}\sigma(T(Q'))\langle f^{p}\rangle_{Q'} \\
        &\ge \frac{1}{\sigma(T(Q))}\sum_{Q'\in \mathcal{J}^{+}(Q)}\sigma(T(Q'))\langle f\rangle_{Q'}^{p} \\
        &\ge \frac{\sigma(\mathcal{B}^{+}(Q))}{\sigma(T(Q))}\lambda^{p}\langle f\rangle_{Q}^{p}\\
        &>\frac{\lambda^{p-1}}{3\bar C}\langle f\rangle_{Q}^{p}>C_{0}^{p}\langle f\rangle_{Q}^{p}.
    \end{align*}

    To conclude this step, let us iterate the estimate in \eqref{eq:gehring-outcome-step1}. For any $Q\in \mathcal{D}$, we recursively define
    \begin{equation*}
        \mathcal{J}_{1}(Q)\coloneqq \mathcal{J}(Q),\qquad \mathcal{J}_{k+1}(Q)\coloneqq\bigcup_{Q'\in \mathcal{J}_{k}(Q)}\mathcal{J}(Q')\quad \forall k\ge 1.
    \end{equation*}
    Next we set 
    \begin{equation*}
        \mathcal{B}_{0}(Q)\coloneqq T(Q),\qquad \mathcal{B}_{k}(Q)\coloneqq \bigcup_{Q'\in \mathcal{J}_{k}(Q)}T(Q')\quad \forall k\ge 1,
    \end{equation*}
    and
    \begin{equation*}
        G_{k}(Q)\coloneqq \mathcal{B}_{k-1}(Q)\setminus \mathcal{B}_{k}(Q)\quad \forall k\ge 1.
    \end{equation*}
    Note that, up to a $\sigma$-null set, $T(Q)$ is the disjoint union of the sets $G_{k}(Q)$ for $k\ge 1$. Indeed, the residual set is $\bigcap_{k\ge 1}\mathcal{B}_{k}(Q)$, whose $\sigma$-measure is zero by the estimate below, which we deduce iterating \eqref{eq:gehring-outcome-step1}: 
    \begin{equation*}
        \sigma(\mathcal{B}_{k}(Q))\le c^{k}\sigma(T(Q))\quad \forall k\ge 1.
    \end{equation*}
    
    \smallskip
    \noindent \textbf{Step 2:} In this step we set $a\coloneqq 1-(1-c)(C_{0}^{p}C_{1}^{p}\lambda^{p})^{-1}\in (0,1)$, and we prove 
    \begin{equation}\label{eq:gehring-outcome-step2}
        \int_{G_{k}(Q)}f^{p}d\sigma\le a^{k-1}\int_{T(Q)}f^{p}d\sigma\qquad \forall Q\in \mathcal{D},\quad \forall k\ge 1.
    \end{equation}
    By property (b), we have $f\ge (C_{1}\lambda)^{-1}\langle f\rangle_{Q}$ in $G_{1}(Q)=G(Q)$. Moreover, $\sigma(G_{1}(Q))\ge (1-c)\sigma(T(Q))$ by \eqref{eq:gehring-outcome-step1}. Therefore, using the reverse-H\"{o}lder inequality from assumption $(i)$, for any $Q\in \mathcal{D}$ we derive
    \begin{equation*}
        \int_{G_{1}(Q)}f^{p}d\sigma\ge \sigma(G_{1}(Q))\frac{\langle f\rangle_{Q}^{p}}{C_{1}^{p}\lambda^{p}}\ge \frac{1-c}{C_{1}^{p}\lambda^{p}}\frac{1}{C_{0}^{p}}\int_{T(Q)}f^{p}d\sigma= (1-a)\int_{T(Q)}f^{p}d\sigma.
    \end{equation*}
    As a consequence, since by construction $G_{k}(Q)=\bigcup_{Q'\in \mathcal{J}_{k-1}(Q)}G_{1}(Q')$, we deduce
    \begin{align*}
        \int_{G_{k}(Q)}f^{p}d\sigma&=\sum_{Q'\in \mathcal{J}_{k-1}(Q)}\int_{G_{1}(Q')}f^{p}d\sigma\\
        &\ge (1-a)\sum_{Q'\in \mathcal{J}_{k-1}(Q)}\int_{T(Q')}f^{p}d\sigma=(1-a)\int_{\mathcal{B}_{k-1}(Q)}f^{p}d\sigma.
    \end{align*}
    This in turn implies
    \begin{equation*}
        \int_{\mathcal{B}_{k}(Q)}f^{p}d\sigma\le a \int_{\mathcal{B}_{k-1}(Q)}f^{p}d\sigma.
    \end{equation*}
    Iterating this bound, and using $G_{k}(Q)\subseteq \mathcal{B}_{k-1}(Q)$ and $\mathcal{B}_{0}(Q)=T(Q)$, we deduce the inequality in \eqref{eq:gehring-outcome-step2}.

    \smallskip
    \noindent \textbf{Step 3:} In this final step, we conclude the proof of the lemma, showing that the $L^{p+\eps}$-reverse-H\"{o}lder inequality holds for some $\eps>0$ and $C_{2}>1$. First note that iterating properties (a) and (b) yields
    \begin{equation}\label{eq:gehring-iteration-properties-a-b}
        \sup_{G_{k}(Q)}f\le C_{1}\bar C^{k-1}\lambda^{k}\langle f\rangle_{Q}\quad \forall k\ge 1.
    \end{equation}
    Then, recalling that $T(Q)=\bigcup_{k\ge 1}G_{k}(Q)$ up to a $\sigma$-null set, and combining \eqref{eq:gehring-iteration-properties-a-b} with the upper bound \eqref{eq:gehring-outcome-step2}, we find
    \begin{align*}
        \fint_{T(Q)}f^{p+\eps}d\sigma&=\frac{1}{\sigma(T(Q))}\sum_{k\ge 1}\int_{G_{k}(Q)}f^{p+\eps}d\sigma\\
        &\le \frac{1}{\sigma(T(Q))}\sum_{k\ge 1}\Big(\sup_{G_{k}(Q)}f\Big)^{\eps}\int_{G_{k}(Q)}f^{p}d\sigma\\
        &\le \Big(\sum_{k\ge 1}(\bar C^{\eps}\lambda^{\eps}a)^{k}\Big)\frac{C_{1}^{\eps}}{\bar C^{\eps} a}\langle f\rangle_{Q}^{\eps}\fint_{T(Q)}f^{p}d\sigma\le \Big(\sum_{k\ge 1}(\bar C^{\eps}\lambda^{\eps}a)^{k}\Big)\frac{C_{1}^{\eps}C_{0}^{p}}{\bar C^{\eps}a}\langle f\rangle_{Q}^{p+\eps},
    \end{align*}
    where in the last step we also used assumption $(i)$.
    Choosing $\eps>0$ so small that $\bar C^{\eps}\lambda^{\eps}a=(1+a)/2<1$, we deduce that the $(p+\eps)$-reverse-H\"{o}lder inequality in \eqref{eq:reverse-holder-ppluseps-gehring} holds with a constant
    \begin{equation*}
        C_{2}\coloneqq \left(\frac{C_{1}^{\eps}C_{0}^{p}}{\bar C^{\eps}a}\sum_{k\ge 1}(\bar C^{\eps}\lambda^{\eps}a)^{k} \right)^{\frac{1}{p+\eps}},
    \end{equation*}
    concluding the proof.
\end{proof}

\paragraph{Estimates on a maximal operator with boundary balls.}

Recall the notation $\mathcal{B}(\Omega)$ for the set of all balls with center on the boundary of a given open set $\Omega$, as well as the maximal operator $M^{\Omega}_{\mu}$ as in Definition \ref{def:boundary-balls-maximal-operator}. In the next lemma we show that $f\mapsto M_{\mu}^{\Omega}(f\mu)$ satisfies the usual maximal $L^{q}$-estimates provided that $\mu$ is doubling on boundary-centered balls.
\begin{lem}\label{lem:maximal-inequality-general-weights}
    Let $\Omega\subset \R^{n}$ be an open set, and let $\mu\in \mathcal{M}_{+}(\Omega^{c})$ be a finite nonnegative measure such that
    \begin{equation}\label{eq:doubling-condition-lemma-appendix}
        0<\mu(B_{2r}(\xi))\le \bar C \mu(B_{r}(\xi))\qquad \forall B_{r}(\xi)\in \mathcal{B}(\Omega).
    \end{equation}
    Then, the maximal operator $Tf\coloneqq M^{\Omega}_{\mu}(f\mu)$ is of weak type $(1,1)$ and of strong type $(p,p)$ with respect to $\mu$ for every $p\in (1,\infty]$, with operator norms depending only on $n$ and $\bar C$. 
\end{lem}
\begin{proof}
    We clearly have the strong $L^{\infty}$-bound. By the Marcinkiewicz interpolation theorem it then suffices to prove the weak $L^{1}$-estimate. Let $\lambda>0$, and consider the set $E_{\lambda}\coloneqq\{M^{\Omega}_{\mu}(f\mu)>\lambda\}$. For every $y\in E_{\lambda}$, we pick a boundary ball $B_{r_{y}}(\xi_{y})\in \mathcal{B}(\Omega)$ such that 
    \begin{equation*}
        y\in B_{r_{y}}(\xi_{y}),\qquad\int_{B_{r_{y}}(\xi_{y})}|f|\,d\mu\ge \lambda \mu(B_{r_{y}}(\xi_{y})).
    \end{equation*}
    Given any compact set $K\subset \R^{n}$,
    by the Vitali covering theorem, we can extract countably many disjoint balls $B_{i}=B_{r_{y_{i}}}(\xi_{y_i})$ such that the union of $5B_{i}$ covers $E_{\lambda}\cap K$. Then, the doubling property of $\mu$ in \eqref{eq:doubling-condition-lemma-appendix} implies 
    \begin{equation*}
        \mu(E_{\lambda}\cap K)\le \sum_{i=1}^{\infty}\mu(5B_{i})\le \bar C^{3}\sum_{i=1}^{\infty}\mu(B_{i})\le \frac{\bar C^{3}}{\lambda}\sum_{i=1}^{\infty}\int_{B_{i}}|f|\,d\mu\le \frac{\bar C^{3}}{\lambda}\int_{\Omega^{c}}|f|\,d\mu,
    \end{equation*}
    and the weak $L^{1}$-estimate follows by the arbitrariness of $K$.
\end{proof}

\paragraph{Acknowledgments.} R.C. and X.F. are supported by the Swiss State Secretariat for Education, Research and Innovation (SERI) under contract number MB22.00034 through the project TENSE, and by the Swiss National Science Foundation (SNF grant PZ00P2\_208930). X.F. is further supported by the AEI project PID2024-156429NB-I00 (Spain). X.R. is supported by the European Union under the ERC Consolidator Grant No. 101123223 (SSNSD), by the AEI project PID2024-156429NB-I00 (Spain), the AEI-DFG project PCI2024-155066-2 (Spain-Germany), the AEI Grant RED2024-153842-T (Spain), and the AEI Maria de Maeztu Program for Centers and Units of Excellence in R\&D CEX2020-001084-M.

\bibliography{Nonlocal_Dahlberg_Refs.bib}
\bibliographystyle{alpha}

\end{document}